\documentclass[12pt]{article}

\usepackage{amsmath,amssymb,amsthm,bm,amscd}
\usepackage{pb-diagram}

\allowdisplaybreaks

\numberwithin{equation}{subsection}

\theoremstyle{plain}
\newtheorem{thm}{Theorem}[subsection]
\newtheorem{lem}[thm]{Lemma}
\newtheorem{prop}[thm]{Proposition}
\newtheorem{cor}[thm]{Corollary}

\newtheorem{claim}{Claim}[thm]

\newtheorem*{claim*}{Claim}

\newtheorem*{ithm}{Theorem}

\theoremstyle{definition}
\newtheorem{define}[thm]{Definition}

\theoremstyle{remark}
\newtheorem{rem}[thm]{Remark}

\newcommand{\Fg}{\mathfrak{g}}
\newcommand{\Fh}{\mathfrak{h}}
\newcommand{\FL}{\mathfrak{L}}
\newcommand{\FM}{\mathfrak{M}}
\newcommand{\FS}{\mathfrak{S}}

\newcommand{\BC}{\mathbb{C}}
\newcommand{\BR}{\mathbb{R}}
\newcommand{\BQ}{\mathbb{Q}}
\newcommand{\BZ}{\mathbb{Z}}
\newcommand{\BB}{\mathbb{B}}
\newcommand{\BP}{\mathbb{P}}

\newcommand{\CB}{\mathcal{B}}

\newcommand{\Rep}{\mathop{\rm Rep}\nolimits}
\newcommand{\Hom}{\mathop{\rm Hom}\nolimits}
\newcommand{\wt}{\mathop{\rm wt}\nolimits}

\newcommand{\Par}{\mathop{\rm Par}\nolimits}

\newcommand{\Turn}{\mathop{\rm Turn}\nolimits}
\newcommand{\Conn}{\mathop{\rm Conn}\nolimits}
\newcommand{\af}{\mathrm{af}}
\newcommand{\Ht}{\mathrm{ht}}

\newcommand{\si}{\frac{\infty}{2}}
\newcommand{\sil}{\ell^{\frac{\infty}{2}}}

\newcommand{\QB}{\mathrm{QB}^J}

\newcommand{\SB}{\mathrm{SiB}^{J}}
\newcommand{\SBo}[1]{\mathrm{SiB}^{#1}}
\newcommand{\SBa}{\mathrm{SiB}(\lambda\,;\,a)}
\newcommand{\SBb}[1]{\mathrm{SiB}(\lambda\,;\,#1)}

\newcommand{\LP}{\mathrm{LZ}(\lambda)}
\newcommand{\LPa}{\mathrm{LZ}(\lambda\,;\,a)}
\newcommand{\LPb}[1]{\mathrm{LZ}(\lambda\,;\,#1)}

\newcommand{\Ia}{I(\lambda\,;\,a)}
\newcommand{\Jca}{J^{c}(\lambda\,;\,a)}
\newcommand{\Jcb}[1]{J^{c}(\lambda\,;\,#1)}

\newcommand{\rr}{\Delta_{\af}}
\newcommand{\prr}{\Delta_{\af}^{+}}

\newcommand{\jad}{Q^{\vee,\,\text{\rm $J$-ad}}}
\newcommand{\kad}{Q^{\vee,\,\text{\rm $K$-ad}}}

\newcommand{\PJ}{\Pi^{J}}
\newcommand{\QJp}[1]{Q_{#1}^{\vee+}}
\newcommand{\pJ}[1]{[#1]}
\newcommand{\mcr}[1]{\lfloor #1 \rfloor}

\newcommand{\edge}[1]{\xrightarrow{\ #1 \ }}
\newcommand{\rsa}[1]{\stackrel{#1}{\text{-\!-\!-\!-\!-\!-\!-}}}

\newcommand{\sLS}{\mathbb{B}^{\si}(\lambda)}
\newcommand{\sLSo}{\mathbb{B}^{\si}_{0}(\lambda)}

\newcommand{\LS}{\mathbb{B}(\lambda )}

\newcommand{\B}{\mathcal{B}(\lambda)}
\newcommand{\Bo}{\mathcal{B}_{0}(\lambda)}

\newcommand{\pair}[2]{\langle #1,\,#2 \rangle}

\newcommand{\ve}{\varepsilon}
\newcommand{\vp}{\varphi}
\newcommand{\vpi}{\varpi}

\newcommand{\ha}[1]{\widehat{#1}}
\newcommand{\ti}[1]{\widetilde{#1}}

\newcommand{\ol}[1]{\overline{#1}}

\newcommand{\bzero}{{\bf 0}}

\newcommand{\bqed}{\quad \hbox{\rule[-0.5pt]{3pt}{8pt}}}

\newenvironment{enu}{%
 \begin{enumerate}%
}{\end{enumerate}}

\begin{document}

\baselineskip=18pt

\title{\Large\bf Semi-infinite Lakshmibai-Seshadri path model \\[2mm]
for level-zero extremal weight modules \\[2mm]
over quantum affine algebras%
}
\author{
Motohiro Ishii \\
 \small Research Center for Pure and Applied Mathematics, \\
 \small Graduate School of Information Sciences, Tohoku University, \\
 \small Aramaki aza Aoba 6-3-09, Aoba-ku, Sendai 980-8579, Japan \\
 \small (e-mail: {\tt ishii@math.is.tohoku.ac.jp}) \\[5mm]
 Satoshi Naito \\ 
 \small Department of Mathematics, Tokyo Institute of Technology, \\
 \small 2-12-1 Oh-okayama, Meguro-ku, Tokyo 152-8551, Japan \\
 \small (e-mail: {\tt naito@math.titech.ac.jp}) \\[5mm]
Daisuke Sagaki \\ 
 \small Institute of Mathematics, University of Tsukuba, \\
 \small Tsukuba, Ibaraki 305-8571, Japan \\
 \small (e-mail: {\tt sagaki@math.tsukuba.ac.jp})
}
\date{}
\maketitle

%
\begin{abstract} \setlength{\baselineskip}{16pt}
We introduce semi-infinite Lakshmibai-Seshadri paths 
by using the semi-infinite Bruhat order 
(or equivalently, Lusztig's generic Bruhat order) 
on affine Weyl groups in place of the usual Bruhat order. 
These paths enable us to give an explicit realization of 
the crystal basis of an extremal weight module of 
an arbitrary level-zero dominant integral extremal weight 
over a quantum affine algebra. This result can be thought 
of as a full generalization of our previous result 
(which uses Littelmann's Lakshmibai-Seshadri paths), 
in which the level-zero dominant integral weight 
is assumed to be a positive-integer multiple of 
a level-zero fundamental weight.
\end{abstract}
%
%
\section{Introduction.} 
\label{sec:intro}

In our previous papers \cite{NS03, NS06}, 
we gave a combinatorial realization of 
the crystal basis $\CB(m_{i} \vpi_{i})$ of 
the extremal weight module $V(m_{i} \vpi_{i})$ of 
extremal weight $m_{i} \vpi_{i}$, 
where $m_{i} \in \BZ_{\geq 1}$ and 
$\vpi_{i}$ is the $i$-th level-zero fundamental weight 
for $i \in I$, over the quantum affine algebra 
$U_{q}(\mathfrak{g}_{\af})$ in terms of 
Lakshmibai-Seshadri (LS for short) paths of 
shape $m_{i} \vpi_{i}$ in the sense of \cite{Lit95}; 
however, in \cite{NS08}, we showed that it is 
impossible to give a realization of 
the crystal basis $\CB(\lambda)$ of 
the extremal weight module $V(\lambda)$ of 
a general level-zero dominant integral extremal weight 
$\lambda$ in terms of Littelmann's LS paths of shape $\lambda$. 
The purpose of this paper is to overcome this difficulty, 
and to give an explicit realization of the crystal basis 
of the extremal weight module $V(\lambda)$ for 
a level-zero dominant integral weight $\lambda$ 
in full generality; however, we assume that 
an affine Lie algebra $\Fg_{\af}$ is 
of untwisted type throughout this paper.

Extremal weight modules over 
the quantized universal enveloping algebras of 
symmetrizable Kac-Moody algebras 
were introduced by Kashiwara \cite{Kas94}
in his study of the level-zero part of 
the modified quantized universal enveloping algebra 
(see \cite{Lus92}) of an affine Lie algebra.

Let $\lambda$ be an integral weight 
for an affine Lie algebra $\Fg_{\af}$. 
If $\lambda$ is of positive (resp., negative) level, 
then the extremal weight module $V(\lambda)$ is just 
the integrable highest (resp., lowest) weight module over $U_{q}(\Fg_{\af})$. 
However, in the case when $\lambda$ is of level-zero, 
the structure of $V(\lambda)$ is 
much more complicated than in the case of positive or negative level. 
In fact, it is known (\cite[Remark~2.15]{Nak04}; 
see also \cite[Proposition~4.5]{CP01}) that 
$V(\lambda)$ is isomorphic to the quantum Weyl module $W_{q}(\lambda)$ 
introduced by Chari and Pressley (\cite{CP01}).

Also, in the case when $\Fg_{\af}$ is 
an untwisted affine Lie algebra of type $A$, $D$, or $E$, 
an extremal weight module can be thought of 
as a universal standard module. Here standard modules $M_P$, 
parametrized by Drinfeld polynomials $P$, 
were constructed by Nakajima (\cite{Nak01}) 
by use of quiver varieties, as a new basis of 
the Grothendieck ring $\Rep U_{q}(L\Fg)$ of 
finite-dimensional modules (of type $\mathrm{1}$) 
over the quantum loop algebra $U_{q}(L\Fg)$, 
where $\Fg \subset \Fg_{\af}$ is the canonical finite-dimensional 
subalgebra (of type $A$, $D$, or $E$), 
and $L\Fg=\BC[t,\,t^{-1}] \otimes_{\BC} \Fg$; 
unique irreducible quotients $L_{P}$ of 
the standard modules $M_{P}$ 
form another basis of $\Rep U_{q}(L\Fg)$.

More precisely, for a level-zero dominant integral weight 
$\lambda = \sum_{i \in I} m_{i} \vpi_{i}$ 
with $m_i \in \BZ_{\ge 0}$, $i \in I$, 
the universal standard module $M(\lambda)$ is 
defined to be the Grothendieck group 
$K^{G_{\lambda} \times \BC^{*}}(\FL(\lambda))$ 
of $G_{\lambda} \times \BC^{*}$-equivariant coherent sheaves 
on a certain Lagrangian subvariety $\FL(\lambda)$ of 
the quiver variety $\FM(\lambda)$, 
where $G_{\lambda} := \prod_{i \in I} GL_{m_i}(\BC)$. 
For an $I$-tuple $P = (P_{i}(u))_{i \in I}$ of 
monic polynomials with coefficients in $\BC(q)$ 
such that $\deg P_{i}(u) = m_{i}$ 
(called a Drinfeld polynomial), 
the corresponding standard module $M_P$ is 
obtained from $M(\lambda)$ as the specialization 
$M(\lambda) \otimes_{R(G_{\lambda})[q,q^{-1}]} \BC(q)$, 
where $R(G_{\lambda}) \cong \bigotimes_{i \in I} 
\BZ[x^{\pm 1}_{i, 1},\,\ldots,\,x^{\pm 1}_{i, m_{i}}]^{\FS_{m_i}}$ 
is the representation ring of $G_{\lambda}$, and 
by the algebra homomorphism $R(G_{\lambda})[q, q^{-1}] \rightarrow \BC(q)$, 
the indeterminates $x_{i, 1},\,\ldots,\,x_{i, m_i}$ are 
sent to the roots of the polynomial $P_{i}(u)$ for $i \in I$. 
In \cite[Theorem~2]{Nak04}, Nakajima proved that 
there exists a $U_{q}(\Fg_{\af}') 
\otimes_{\BZ[q, q^{-1}]} R(G_{\lambda})[q, q^{-1}]$-module isomorphism 
between the extremal weight module $V(\lambda)$ and 
$M(\lambda) \otimes_{\BZ[q, q^{-1}]} \BC(q)$, 
as a by-product of his proof of Kashiwara's conjecture 
(see \cite[\S 13]{Kas02b}) on the structure of 
extremal weight modules of level-zero extremal weights, 
where $\Fg_{\af}'=\bigl(\BC[t,\,t^{-1}] \otimes_{\BC} \Fg\bigr) \oplus \BC c$. 

Now, let $\lambda$ be a level-zero dominant integral weight 
of the form $\lambda = \sum_{i \in I} m_{i} \vpi_{i}$, 
with $m_{i} \in \BZ_{\geq 0}$ for all $i \in I$. 
In \cite{BN04}, Beck and Nakajima proved Kashiwara's conjecture above 
for all affine Lie algebras, and in particular, 
showed that there exists an isomorphism of crystals
\begin{align}\label{eq:tensor}
\CB(\lambda) \cong \bigotimes_{i \in I} \CB(m_{i} \vpi_{i}). \tag{1}
\end{align}
Soon afterward, in \cite{NS03, NS06}, 
we gave a combinatorial realization of 
the crystal basis $\CB(m_{i} \vpi_{i})$ 
for each $i \in I$ in terms of LS paths of 
shape $m_{i} \vpi_{i}$. Namely, we proved that 
the set $\BB(m_{i} \vpi_{i})$ of LS paths 
of shape $m_{i} \vpi_{i}$ is isomorphic as a crystal 
to $\CB(m_{i} \vpi_{i})$, though neither of 
(the crystal graphs of) these two crystals is 
connected if $m_{i} > 1$. However, it turned out in \cite{NS08} 
that the crystals $\BB(\lambda)$ and $\CB(\lambda)$ are 
not necessarily isomorphic for a general 
level-zero dominant integral weight $\lambda$; 
for example, if $\lambda$ is of the form 
$\sum_{i \in K} \vpi_{i}$ for $K \subset I$ with $\# K \geq 2$, 
then the crystals $\BB(\lambda)$ and $\CB(\lambda)$ are 
never isomorphic, though both of these crystals are 
connected (for details, see \cite[Appendix]{NS08}). 
This is mainly because each connected component of 
the crystal $\BB(\lambda)$ has fewer extremal elements 
than that of the crystal $\CB(\lambda)$ has.

Let us explain the situation above more precisely. 
As a consequence of the isomorphism \eqref{eq:tensor}, 
we see (Proposition~\ref{prop:stab}) that 
the set $\bigl\{ u_x := x u_{\lambda} \mid x \in W_{\af} \bigr\}$ of 
extremal elements in the connected component 
$\CB_{0}(\lambda)$ of $\CB(\lambda)$ containing 
the extremal element $u_{\lambda}$ of extremal weight $\lambda$ 
is in bijective correspondence 
with the quotient set $W_{\af} / (W_J)_{\af}$, where 
$W_{\af} = W \ltimes \bigl\{ t_{\xi} \mid \xi \in Q^{\vee} \bigr\}$ 
is the (affine) Weyl group of $\Fg_{\af}$, and 
$(W_J)_{\af} := W_J \ltimes 
\bigl\{ t_{\xi} \mid \xi \in Q_J^{\vee} \bigr\}$, 
with $J :=\bigl\{i \in I \mid m_i = 0 \bigr\}$, 
is identical to the stabilizer 
$\bigl\{ x \in W_{\af} \mid x u_{\lambda} = u_{\lambda} \bigr\}$ 
of $u_{\lambda}$ in $W_{\af}$ .
In contrast, the set 
$\bigl\{ \pi_{x\lambda} := (x\lambda\,;\,0,\,1) \mid x \in W_{\af} \bigr\}$
of extremal elements in the connected component 
$\BB_{0}(\lambda)$ of $\BB(\lambda)$ containing 
the straight line path $\pi_{\lambda} := (\lambda\,;\,0,\,1)$ 
of weight $\lambda $ is in bijective correspondence 
with the quotient set $W_{\af}/(W_{\af})_{\lambda}$, 
where $(W_{\af})_{\lambda}$ is the stabilizer of $\lambda$ 
in $W_{\af}$, and is identical to $W_{J} \ltimes 
\bigl\{ t_{\xi} \mid \pair{\xi}{\lambda} = 0 \bigr\}$. 
Here, we have $(W_J)_{\af} \subset (W_{\af})_{\lambda}$ in general, 
with equality if and only if $\lambda$ is 
a nonnegative integer multiple of $\vpi_{i}$ 
for some $i \in I$. 

In order to overcome the difficulty mentioned above, 
we introduce the notion of 
semi-infinite Lakshmibai-Seshadri 
(SiLS for short) paths of shape $\lambda$. 
A SiLS path of shape $\lambda$ is, by definition, 
a pair $(\bm{x} ; \bm{a})$ of a decreasing sequence 
$\bm{x} : x_1 >_{\si} x_2 >_{\si} \cdots >_{\si} x_s$ 
in the set $(W^J)_{\af}$ of Peterson's coset representatives 
for the cosets in $W_{\af} / (W_J)_{\af}$, 
equipped with the semi-infinite Bruhat order $\ge_{\si}$, and 
an increasing sequence $\bm{a} : 0 = a_0 < a_1 < \cdots < a_s =1$ 
of rational numbers, while a (usual) LS path of shape $\lambda$ 
is a pair $(\bm{\lambda } ; \bm{a})$ of a decreasing sequence 
$\bm{\lambda } : \lambda_1 > \lambda_2 > \cdots > \lambda_s$ 
of elements in the affine Weyl group  orbit $W_{\af} \lambda$ through $\lambda$, 
equipped with the partial order $\ge$ which 
Littelmann defined in \cite{Lit95}, and 
an increasing sequence $\bm{a}$ of rational numbers as above. 

The coset representatives $(W^J)_{\af}$ were 
originally introduced by Peterson (\cite{Pet97}; see also \cite{LS10}) 
in his study of the relationship between 
the $T$-equivariant homology (ring) of 
the affine Grassmannian $G(\BC (\!( t )\!)) / G(\BC [\![ t ]\!])$ 
and the $T$-equivariant (small) quantum cohomology ring 
$QH_T^{\bullet}(G/P)$ of the partial flag variety $G/P$, 
where $G$ denotes a simply-connected simple algebraic group 
over $\BC$, $P \subset G$ a parabolic subgroup, 
and $T \subset G$ a maximal torus. 
Also, we see from \cite[Claim~4.14]{Soe97} that 
the semi-infinite Bruhat order on the affine Weyl group $W_{\af}$ is 
a slight modification of Lusztig's generic Bruhat order on $W_{\af}$ 
(for details, see Appendix \ref{sec:Appendix}); 
the generic Bruhat order was originally introduced 
by Lusztig (\cite{Lus80}) in his study of 
the conjectural character formula for 
the irreducible quotient of the Weyl module of 
a simply-connected almost simple algebraic group 
over an algebraically closed field of positive characteristic. 
As for the geometric meaning of the semi-infinite Bruhat order, 
it is known (\cite[\S5]{FFKM99}; see also \cite[\S4]{FF90}) that 
the semi-infinite Bruhat order describes the closure relation 
among the fine Schubert strata, parametrized by $W_{\af}$, 
of the Drinfeld compactification of the variety of 
algebraic maps of a fixed degree from 
the complex projective line $\BP^{1}$ to 
the flag variety $G/B$. In addition, we remark that 
in Peterson's lecture note (\cite{Pet97}), 
the semi-infinite Bruhat order (or, stable Bruhat order 
in his terminology) plays an important role, and 
that some of our arguments in the study of SiLS paths 
use (parabolic) quantum Bruhat graphs (\cite{BFP99, LNSSS13a}), 
which appear in the equivariant quantum Chevalley formula 
for $QH_T^{\bullet} (G/P)$ (\cite{Mih07}). 

Now we are ready to state our main results. 
First, we prove that the natural surjection 
from the poset $(W^J)_{\af}$ 
(equipped with the semi-infinite Bruhat order) 
onto the poset $W_{\af} \lambda$ 
(equipped with Littelmann's partial order) 
is order-preserving, and hence that 
there exists a surjection from the set $\sLS$ 
of SiLS paths of shape $\lambda$ onto 
the set $\LS$ of LS paths of shape $\lambda$ (Proposition~\ref{prop:pi_}). 
Next, we define a crystal structure on $\sLS$ for 
$U_q(\Fg_{\af})$ in such a way that 
this surjection becomes a morphism of crystals 
(Theorem~\ref{thm:stability}). Then, each connected component of the 
resulting crystal $\sLS$ indeed has as many extremal elements as 
that of the crystal $\CB(\lambda)$ has (Proposition~\ref{prop:conn-isom}).
Also, we can prove that the crystal $\sLS$ has as many connected components 
as the crystal $\CB(\lambda)$ has (Proposition~\ref{prop:BNforSLS}). 
Combining the results above, we finally obtain the following (Theorem~\ref{thm:main}).

\begin{ithm}
Let $\Fg_{\af}$ be an untwisted affine Lie algebra, 
and $\lambda = \sum_{i \in I} m_{i} \vpi_{i}$ 
a level-zero dominant integral weight, 
with $m_{i} \in \BZ_{\geq 0}$ for all $i \in I$. 
Let $\B$ denote the crystal basis of 
the extremal weight module $V(\lambda)$ 
over $U_q (\Fg_{\af})$, and let $\sLS$ denote 
the set of SiLS paths of shape $\lambda$, 
equipped with the $U_q (\Fg_{\af})$-crystal structure as above. 
Then, we have an isomorphism of crystals 
\begin{align} \label{eq:isomorphism}
\B \cong \sLS. \tag{2}
\end{align}
\end{ithm}

Remark that for each $i \in I$, 
we have a natural identification 
$\BB^{\si}(m_{i} \vpi_{i}) = \BB(m_{i} \vpi_{i})$, 
since the equality $(W_{I \setminus \{i\}})_{\af} = 
(W_{\af})_{m_{i} \vpi_{i}}$ holds. 
Hence we recover our previous results in \cite{NS05, NS06}. 
Also, we should mention that Hernandez and Nakajima (\cite{HN06}) 
gave a monomial realization of the connected component $\CB_{0}(\lambda)$; 
however, their realization is given in a recursive way, 
and it is difficult to determine 
all the elements in $\CB_{0}(\lambda)$ explicitly in this realization.

As an application of the isomorphism theorem above,
in our paper \cite{NS-Dem} sequel to the present one,
we prove that the graded character of a particular Demazure submodule
$V_{w_{0}}^{+}(\lambda)$ of $V(\lambda)$ 
for the longest element $w_{0} \in W$ is 
identical to a $q$-Whittaker function,
which is nothing but the graded character of a global Weyl module over the current
algebra (see \cite{BF} for simply-laced cases). 
This result yields a purely crystal-theoretical explanation of 
the relation between a $q$-Whittaker function and
the specialization at $t = 0$ of a symmetric Macdonald polynomial,
which is nothing but the graded character of 
a local Weyl module over the current algebra (see \cite{LNSSS2}).

This paper is organized as follows. In \S \ref{sec:prelim}, 
we fix our notation for untwisted affine root data, and review 
some basic facts about LS paths and (parabolic) semi-infinite Bruhat 
graphs. 
In \S \ref{sec:main}, we introduce SiLS paths 
and define a crystal structure on them, postponing to 
\S\ref{sec:pr-prop:pi_} the proof of the stability property of
the set of semi-infinite LS paths under root operators. 
Also, we state our main result, i.e., 
the isomorphism theorem above between $\B$ and $\sLS$. 
In \S\ref{sec:pr-prop:pi_}, we first study some fundamental 
properties of semi-infinite Bruhat graphs, and 
then prove that there exists a canonical surjection 
from $\sLS$ to $\LS$. Using this surjection, we give 
the postponed proof of the stability property above of SiLS paths. 
In \S\ref{sec:pr-prop:conn-isom}, 
we prove that the connected component $\Bo$ of $\B$ 
containing $u_{\lambda}$ is isomorphic, as a crystal, 
to the connected component $\sLSo$ of $\sLS$ 
containing the element $\eta_e := (e\,;\,0,1)$.
In \S\ref{sec:directed-path}, 
we obtain a condition for the existence of a directed path 
in a (parabolic) semi-infinite Bruhat graph from a translation 
to another translation; this result is used in \S\ref{sec:ConnComp} 
to give a parametrization of the connected components of $\sLS$. 
Finally, in \S\ref{sec:ConnComp}, by combining the results 
in \S\ref{sec:pr-prop:pi_}, \S\ref{sec:pr-prop:conn-isom}, 
and \S\ref{sec:directed-path}, we obtain the desired isomorphism 
$\B \cong \sLS$. 
In Appendix \ref{sec:Appendix}, 
we give another (but equivalent) definition of semi-infinite Bruhat graphs, 
and also mention the relation between the semi-infinite Bruhat 
order and Lusztig's generic Bruhat order. 

\paragraph{Acknowledgements.}
The second and third authors would like to 
thank Christian Lenart, Anne Schilling, and Mark Shimozono 
for related collaborations. 
The third author thanks Professor Peter Littelmann, Professor Ghislain Fourier, 
Dr. Deniz Kus, and students attending at Professor Littelmann's seminar for 
giving us valuable comments on this work.
%
%
\section{LS paths and semi-infinite Bruhat graphs.}
\label{sec:prelim}
%
%
\subsection{Untwisted affine root data.}
\label{subsec:affalg}

Let $\Fg_{\af}$ be an untwisted affine Lie algebra 
over $\BC$ with Cartan subalgebra $\Fh_{\af}$. 
Let $\{ \alpha_i \}_{i \in I_{\af}} \subset 
\Fh_{\af}^* := \Hom_{\BC}(\Fh_{\af},\,\BC)$ 
and $\{ \alpha_i^{\vee} \}_{i \in I_{\af}} 
\subset \Fh_{\af}$ be the sets of simple roots and 
simple coroots, respectively.
Let $\pair{\cdot\,}{\cdot} : 
\Fh_{\af} \times \Fh_{\af}^* \rightarrow \BC$ 
denote the canonical pairing. 
Throughout this paper, we take and fix an integral weight lattice 
$P_{\af} \subset \Fh_{\af}^{*}$ satisfying the conditions that 
$\alpha_i \in P_{\af}$ and 
$\alpha_i^{\vee} \in \Hom_{\BZ}(P_{\af},\,\BZ)$ 
for all $i \in I_{\af}$, and that for each $i \in I_{\af}$
there exists $\Lambda_i \in P_{\af}$ such that 
$\pair{\alpha_{j}^{\vee}}{\Lambda_i} = \delta_{ij}$ 
for all $j \in I_{\af}$. 
Let $\delta = \sum_{i \in I_{\af}} a_i \alpha_i \in \Fh_{\af}^{*}$ 
and $c = \sum_{i \in I_{\af}} a_i^{\vee} \alpha_i^{\vee} \in \Fh_{\af}$ 
be the null root and the canonical central element, respectively. 
We take and fix $0 \in I_{\af}$ such that $a_0 = a_0^{\vee} =1$, 
and set $I := I_{\af} \setminus \{ 0 \}$; note that 
the subset $I$ of $I_{\af}$ corresponds to 
the index set for the finite-dimensional 
simple Lie subalgebra $\Fg$ of $\Fg_{\af}$. 
For each $i \in I$, we define $\varpi_i := 
\Lambda_i - \pair{c}{\Lambda_i} \Lambda_0$; 
note that $\pair{c}{\varpi_i}=0$ for all $i \in I$. 
Set 
\begin{align*}
Q := \bigoplus_{i \in I} \BZ \alpha_i, &&
Q^{\vee} := \bigoplus_{i \in I} \BZ \alpha_i^{\vee}, &&
P^+ := \sum_{i \in I} \BZ_{\ge 0} \varpi_i; 
\end{align*} 
we call an element of $P^{+}$ a level-zero dominant integral weight. 

Let $W_{\af} := \langle r_i \mid i \in I_{\af} \rangle$ be 
the (affine) Weyl group of $\Fg_{\af}$, 
where $r_i$ denotes the simple reflection 
with respect to $\alpha_i$, and 
set $W := \langle r_i \mid i \in I \rangle \subset W_{\af}$, 
which can be regarded as the Weyl group of $\Fg$. 
Let $e \in W_{\af}$ be the unit element, and 
$\ell : W_{\af} \rightarrow \mathbb{Z}_{\ge 0}$ 
the length function. 
For $\xi \in Q^{\vee}$, denote by $t_{\xi} \in W_{\af}$ 
the translation with respect to $\xi$ (see \cite[\S6.5]{Kac90}). 
We know from \cite[Proposition 6.5]{Kac90} that 
$\bigl\{ t_{\xi} \mid \xi \in Q^{\vee} \bigr\}$ forms 
an abelian normal subgroup of $W_{\af}$, 
for which $t_{\xi} t_{\zeta} = t_{\xi + \zeta}$, 
$\xi,\,\zeta \in Q^{\vee}$, and 
$W_{\af} = W \ltimes \bigl\{ t_{\xi} \mid \xi \in Q^{\vee} \bigr\}$; 
remark that for $w \in W$ and $\xi \in Q^{\vee}$, we have 
%
%
\begin{align}\label{eq:lv0action}
w t_{\xi} \mu = w \mu - \pair{\xi}{\mu}\delta \quad 
\text{if $\mu \in \Fh_{\af}^{*}$ satisfies $\pair{c}{\mu}=0$}.
\end{align}

Denote by $\Delta_{\af}$ the set of real roots of 
$\Fg_{\af}$, and $\prr$ the set of positive real roots 
of $\Fg_{\af}$; we know from \cite[Proposition 6.3]{Kac90} that
\begin{align*}
\begin{split}
\rr & = 
\bigl\{ \alpha + n \delta \mid \alpha \in \Delta,\, n \in \BZ \bigr\}, \\
\prr & = 
\Delta^{+} \sqcup 
\bigl\{ \alpha + n \delta \mid \alpha \in \Delta,\, n \in \BZ_{> 0}\bigr\},
\end{split}
\end{align*}
where $\Delta := \Delta_{\af} \cap Q$ is 
the (finite) root system corresponding to $I$, and 
$\Delta^{+} := \Delta \cap \sum_{i \in I} \BZ_{\ge 0} \alpha_i$. 
For $\beta \in \Delta_{\af}$, 
denote by $\beta^{\vee} \in Q^{\vee}$ the coroot of $\beta $, 
and by $r_{\beta} \in W_{\af}$ 
the reflection with respect to $\beta$; 
if $\beta \in \Delta_{\af}$ is 
of the form $\beta = \alpha + n \delta $ 
with $\alpha \in \Delta $ and $n \in \BZ$, then 
%
%
\begin{equation}\label{eq:refl}
r_{\beta} =r_{\alpha} t_{n\alpha^{\vee}} .
\end{equation}
%

For a subset $J$ of $I$, we set 
\begin{align*}
Q_J &:= \bigoplus_{j \in J} \BZ \alpha_j, &
Q_J^{\vee} &:= \bigoplus_{j \in J} \BZ \alpha_j^{\vee}, &
\QJp{J} &:= \sum_{j \in J} \BZ_{\ge 0} \alpha_j^{\vee}, \\[3mm]
\Delta_J &:= \Delta \cap Q_J, & 
\Delta_J^+ &:= 
 \Delta_J \cap \sum_{i \in I} \BZ_{\ge 0}  \alpha_i , &
W_J &:= \langle r_j \mid j \in J \rangle . &
\end{align*}
Let $W^J$ denote the set of minimal(-length) coset representatives 
for $W / W_J$; we see from \cite[\S 2.4]{BB05} that 
%
%
\begin{align}\label{eq:W^J}
W^J = \bigl\{ w \in W \mid 
\text{$w \alpha \in \Delta^+$ for all $\alpha \in \Delta_J^+$}\bigr\} .
\end{align}
For $w \in W$, we denote by $\mcr{w}=\mcr{w}^{J} \in W^J$ 
the minimal coset representative for the coset $w W_J$ in $W / W_J$.

%
\subsection{Lakshmibai-Seshadri paths.}
\label{subsec:LSpath}

In this subsection, we briefly review 
some basic facts about crystals of 
Lakshmibai-Seshadri (LS for short) paths, 
introduced by Littelmann \cite{Lit94}, \cite{Lit95}. 
In this subsection, we fix $\lambda \in P^+$. 

\begin{define}[{\cite[\S 4]{Lit95}}] \label{def:wtposet}
We define a partial order $\le$ on $W_{\af} \lambda $ as follows: 
for $\mu,\,\nu \in W_{\af} \lambda $, we write $\mu \le \nu $ 
if there exists a sequence 
$\mu = \nu_0,\,\nu_1,\,\ldots,\,\nu_k = \nu$ of 
elements in $W_{\af} \lambda$ and 
a sequence $\beta_1,\,\ldots,\,\beta_k$ 
of elements in $\prr$ such that 
$\nu_m = r_{\beta_m} \nu_{m-1}$ and 
$\pair{\beta_m^{\vee}}{\nu_{m-1}} \in \BZ_{>0}$ 
for all $m=1,\,2,\,\dots,\,k$. 
We call the poset $(W_{\af} \lambda,\,\le)$ 
the level-zero weight poset of shape $\lambda$. 
\end{define}

\begin{rem} \label{rem:cover}
Let $\nu \in P$, and $\beta,\,\gamma \in \prr$. 
If $r_{\beta}\nu=r_{\gamma}\nu \ne \nu$, then 
$\beta=\gamma$. Indeed, we have
$\nu-\pair{\beta^{\vee}}{\nu}\beta = 
 \nu-\pair{\gamma^{\vee}}{\nu}\gamma$, 
and hence $\pair{\beta^{\vee}}{\nu}\beta =
\pair{\gamma^{\vee}}{\nu}\gamma$. 
Also, since $r_{\beta}\nu \ne \nu$ and 
$r_{\gamma}\nu \ne \nu$, it follows immediately that 
$\pair{\beta^{\vee}}{\nu} \ne 0$ and 
$\pair{\gamma^{\vee}}{\nu} \ne 0$. 
Hence, by \cite[Proposition~5.1\,b)]{Kac90}, 
we obtain $\beta=\gamma$. 
\end{rem}

\begin{define}\label{def:X}
\mbox{}
\begin{enu}
\item
Define $\LP$ to be the $\prr$-labeled, directed graph 
with vertex set $W_{\af} \lambda$ and 
$\prr$-labeled, directed edges of the following form: 
$\mu \edge{\beta} \nu$ for $\mu,\,\nu \in W_{\af} \lambda$, 
where $\nu$ covers $\mu$ in the poset $W_{\af} \lambda$, and 
the label $\beta$ of the edge 
is a unique positive real root $\beta \in \prr$ 
such that $\nu = r_{\beta } \mu$ and 
$\pair{\beta^{\vee}}{\mu} > 0$ (see Remark~\ref{rem:cover}).

\item
Let $0 < a \le 1$ be a rational number. 
Define $\LPa$ to be the subgraph of $\LP$ 
with the same vertex set but having only the edges of the form:
%
\begin{align}\label{eq:Xa}
\mu \edge{\beta} \nu \quad \text{with} \quad 
a \pair{\beta^{\vee}}{\mu} \in \BZ;
\end{align}
note that $\LPb{1} = \LP$.
\end{enu}
\end{define}

\begin{define}[{\cite[\S 4]{Lit95}}] \label{def:LSpath}
An LS path of shape $\lambda $ is, 
by definition, a pair $(\bm{\nu}\,;\,\bm{a})$ of 
a decreasing sequence $\bm{\nu} : \nu_1 > \cdots > \nu_s$ 
of elements in $W_{\af} \lambda $ and an increasing sequence 
$\bm{a} : 0 = a_0 < a_1 < \cdots < a_s =1$ of 
rational numbers satisfying the condition that 
there exists a directed path from $\nu_{u+1}$ to $\nu_u$ 
in $\LPb{a_u}$ for each $u = 1,\,2,\,\ldots,\,s-1$. 
Let $\LS$ denote the set of LS paths of shape $\lambda$.
\end{define}

We identify $\pi = 
(\nu_1,\,\ldots,\,\nu_s\,;\,a_0,\,\ldots,\,a_s) \in \LS$ 
with the piecewise-linear, continuous map 
$\pi : [0,1] \rightarrow \BR \otimes_{\BZ} P_{\af}$ 
whose ``direction vector'' for the interval 
$[a_{u-1},\,a_{u}]$ is equal to $\nu_{u}$ 
for each $1 \le u \le s$, that is, 
\begin{equation} \label{eq:LSpath}
\pi (t) = 
\sum_{p = 1}^{u-1}(a_p - a_{p-1}) \nu_p + (t - a_{u-1}) \nu_u 
\quad
\text{for $t \in [a_{u-1},\,a_u]$, $1 \le u \le s$}.
\end{equation}
We also express $\pi = 
(\nu_1,\,\ldots,\,\nu_s\,;\,a_0,\,\ldots,\,a_s) \in \LS$ as:
\begin{equation*}
0= a_{0} \rsa{\nu_1} a_{1} \rsa{\nu_2} \cdots 
\rsa{\nu_{s-1}} a_{s-1}
\rsa{\nu_{s}} a_{s}=1.
\end{equation*}

Now, we equip the set $\LS$ with a crystal structure with weights 
in $P_{\af}$ as follows (for the definition of crystals, see 
\cite[\S 7.2]{Kas95} and \cite[Definition 4.5.1]{HK02} for example). 
First, we define $\wt : \LS \rightarrow P_{\af}$ by 
$\wt(\pi) := \pi (1) \in P_{\af}$; 
we know from \cite[Lemma~4.5\,a)]{Lit95} that 
$\pi(1) \in P_{\af}$ for all $\pi \in \LS$.
Next, for $\pi \in \LS$ and $i \in I_{\af}$, we set
\begin{equation} \label{eq:H}
\begin{cases}
H_i^{\pi}(t) := \pair{\alpha_i^{\vee}}{\pi(t)} \quad 
\text{for $t \in [0,1]$}, \\[1.5mm]
m_i^{\pi} := 
 \min \bigl\{ H_i^{\pi} (t) \mid t \in [0,1] \bigr\}.
\end{cases}
\end{equation}
\begin{rem}\label{rem:locmini}
We see from \cite[Lemma 4.5 d)]{Lit95} that 
for each $\pi \in \LS$ and $i \in I_{\af}$, 
all local minima of the function 
$H^{\pi}_{i}(t)$, $t \in [0,1]$, are integers. 
In particular, the minimum $m^{\pi}_{i}$ 
is a nonpositive integer (recall that $\pi(0)=0$, 
and hence $H^{\pi}_{i}(0)=0$). 
\end{rem}

Following \cite[\S1]{Lit95} (see also \cite[\S 1]{NS06}), 
we define the root operators $e_i$, $f_i$, $i \in I_{\af}$, 
on $\LS \sqcup \bigl\{\bzero\bigr\}$ as follows. Here, 
$\bzero$ is an additional element 
not contained in any crystal. 
%
%
\begin{define} \label{dfn:ro}
Let $\pi = (\nu_1,\,\ldots,\,\nu_s\,;\,a_0,\,\ldots,\,a_s) \in \LS$, 
and $i \in I_{\af}$. 
\begin{enu}
\item
If $m_i^{\pi} = 0$, then we define $e_i \pi := \bzero$. 
If $m_i^{\pi} \le -1$, then set
%
%
\begin{align} \label{eq:t-e}
\begin{cases}
t_1 := 
  \min \bigl\{ t \in [0,1] \mid 
    H_i^{\pi}(t) = m_i^{\pi} \bigr\}, \\[1.5mm]
t_0 := 
  \max \bigl\{ t \in [0,t_1] \mid 
    H_i^{\pi}(t) = m_i^{\pi} + 1 \bigr\} ;
\end{cases}
\end{align}
we deduce from Remark \ref{rem:locmini} that $H_i^{\pi}(t)$ is 
strictly decreasing on $[t_0,\,t_1]$. 
Notice that 
there exists $1 \le q \le s$ such that $t_1 = a_q$. 
Let $1 \le p \le q$ be such that $a_{p-1} \le t_0 < a_p$. 
Then we define $e_{i}\pi$ to be 
%
%
%
%
%
%
%
\begin{equation*}
0= a_{0} \rsa{\nu_1} \cdots \rsa{\nu_{p-1}} a_{p-1} \rsa{\nu_p} 
\underbrace{t_{0} \rsa{r_{i}\nu_p} a_{p} \rsa{r_{i}\nu_{p+1}} \cdots 
\rsa{r_{i}\nu_{q}} a_{q}=t_{1}}_{\text{``reflected'' by $r_{i}$}}
\rsa{\nu_{q+1}} \cdots \rsa{\nu_{s}} a_{s}=1,
\end{equation*}
that is, 
\begin{align*}
e_i \pi := ( 
& \nu_1,\,\ldots,\,\nu_p,\,r_i \nu_p,\,r_i \nu_{p+1},\,\ldots,\,
  r_i \nu_q,\,\nu_{q+1},\,\ldots,\,\nu_s ; \\
& a_0,\,\ldots,\,a_{p-1},\,t_0,\,a_p,\,\ldots,\,a_{q}=t_1,\,
\ldots,\,a_s);
\end{align*}
if $t_0 = a_{p-1}$, then we drop $\nu_p$ and $a_{p-1}$, and if $r_i \nu_q = \nu_{q+1}$, then we drop $\nu_{q+1}$ and $a_{q}=t_1$.

\item
If $H_i^{\pi}(1) - m_i^{\pi}  = 0$, then 
we define $f_i \pi := \bzero$. 
If $H_i^{\pi}(1) - m_i^{\pi}  \ge 1$, 
then set
\begin{align} \label{eq:t-f}
\begin{cases}
t_0 := 
 \max \bigl\{ t \in [0,1] \mid H_i^{\pi}(t) = m_i^{\pi} \bigr\}, \\[1.5mm]
t_1 := 
 \min \bigl\{ t \in [t_0,1] \mid H_i^{\pi}(t) = m_i^{\pi} + 1 \bigr\};
\end{cases}
\end{align}
we deduce from Remark \ref{rem:locmini} that $H_i^{\pi}(t)$ is 
strictly increasing on $[t_0,\,t_1]$. 
Notice that 
there exists $0 \le p \le s-1$ such that $t_0 = a_p$. 
Let $p \le q \le s-1$ be such that $a_{q} < t_1 \le a_{q+1}$. 
Then we define $f_{i}\pi$ to be 
\begin{equation*}
0= a_{0} \rsa{\nu_1} \cdots \rsa{\nu_{p}} 
\underbrace{a_{p}=t_{0} \rsa{r_{i}\nu_{p+1}} 
\cdots \rsa{r_{i}\nu_{q}} a_{q} \rsa{r_{i}\nu_{q+1}} t_{1}%
}_{\text{``reflected'' by $r_{i}$}}
\rsa{\nu_{q+1}} a_{q+1} 
\rsa{\nu_{q+2}} \cdots \rsa{\nu_{s}} a_{s}=1,
\end{equation*}
that is, 
\begin{align*}
f_i \pi := ( 
& \nu_1,\,\ldots,\,\nu_{p},\,r_i\nu_{p+1},\,\dots,\,
  r_i\nu_{q},\,r_i\nu_{q+1},\,\nu_{q+1},\,\ldots,\,\nu_s ; \\
& a_0,\,\ldots,\,a_{p}=t_0,\,\ldots,\,a_{q},\,t_1,\,
  a_{q+1},\,\ldots,\,a_s);
\end{align*}
if $t_1 = a_{q+1}$, then we drop $\nu_{q+1}$ and $a_{q+1}$, and 
if $\nu_{p} = r_i \nu_{p+1}$, then we drop $\nu_{p}$ and $a_{p}=t_0$.

\item 
Set $e_i\bzero=f_i\bzero:=\bzero$ for all $i \in I_{\af}$.
\end{enu}
\end{define}

We know from \cite[Corollary 2 a)]{Lit95} that 
the set $\LS \sqcup \bigl\{ \bzero \bigr\}$ is 
stable under the action of the root operators $e_i$, $f_i$, 
$i \in I_{\af}$. Now we define 
%
%
\begin{equation} \label{eq:vevp}
\begin{cases}
\ve_i (\pi) := 
 \max \bigl\{ n \in \BZ_{\ge 0} \mid e_i^n \pi \neq \bzero \bigr\}, \\[1.5mm]
\vp_i (\pi) := 
 \max \bigl\{ n \in \BZ_{\ge 0} \mid f_i^n \pi \neq \bzero \bigr\}
\end{cases}
\end{equation}
for $\pi \in \LS$ and $i \in I_{\af}$. 
We know from \cite[\S2 and \S4]{Lit95} that 
the set $\LS$, equipped with the maps $\wt$, $e_i$, $f_i$, $i \in I_{\af}$, 
and $\ve_i$, $\vp_i$, $i \in I$, defined above, is a crystal 
with weights in $P_{\af}$.

%
\subsection{Peterson's coset representatives.}
\label{subsec:W^J_af}

Let $J$ be a subset of $I$. Following \cite{Pet97} (see also \cite[\S 10]{LS10}), we define
\begin{align}
(\Delta_J)_{\af} 
  & := \bigl\{ \alpha + n \delta \mid 
  \alpha \in \Delta_J , n \in \BZ \bigr\} \subset \Delta_{\af}, \\
(\Delta_J)_{\af}^+
  &:= (\Delta_J)_{\af} \cap \prr = 
  \Delta_J^+ \sqcup \bigl\{ \alpha + n \delta \mid 
  \alpha \in \Delta_J,\,n \in \BZ_{> 0} \bigr\}, \\
\label{eq:stabilizer}
(W_J)_{\af} 
 & := W_J \ltimes \bigl\{ t_{\xi} \mid \xi \in Q_J^{\vee} \bigr\} , \\
\label{eq:Pet}
(W^J)_{\af}
 &:= \bigl\{ x \in W_{\af} \mid 
 \text{$x\beta \in \prr$ for all $\beta \in (\Delta_J)_{\af}^+$} \bigr\}. 
\end{align}

\begin{rem} \label{rem:stabilizer}
We can easily show, using \eqref{eq:refl}, that 
$(W_J)_{\af} = \langle r_{\beta} \mid \beta \in (\Delta_J)_{\af}^+ \rangle$.
%
%
\end{rem}
We know the following proposition from \cite{Pet97} 
(see also \cite[Lemma 10.6]{LS10}).
%
%
\begin{prop} \label{prop:P}
For each $x \in W_{\af}$, there exist a unique 
$x_1 \in (W^J)_{\af}$ and a unique $x_2 \in (W_J)_{\af}$ 
such that $x = x_1 x_2$.
\end{prop}

We define a (surjective) map $\PJ : W_{\af} \rightarrow (W^J)_{\af}$ 
by $\PJ (x) := x_1$ if $x= x_1 x_2$ with $x_1 \in (W^J)_{\af}$ and 
$x_2 \in (W_J)_{\af}$.
%
%
\begin{lem}[{\cite{Pet97}; see also \cite[Proposition 10.10]{LS10}}] 
\label{lem:PJ}
\mbox{}
\begin{enu}
\item $\PJ (w) = \lfloor w \rfloor $ for every $w \in W$.
\item $\PJ (x t_{\xi}) = \PJ (x) \PJ (t_{\xi})$ 
for every $x \in W_{\af}$ and $\xi \in Q^{\vee}$.
\end{enu}
\end{lem}
%
%
\begin{define}[{see \cite[Lemma~3.8]{LNSSS13a}}] \label{def:J-adj} 
An element $\xi \in Q^{\vee}$ is said to be $J$-adjusted 
if $\pair{\xi}{\gamma} \in \bigl\{ -1,\,0 \bigr\}$ 
for all $\gamma \in \Delta_J^+$. 
Let $\jad$ denote the set of $J$-adjusted elements.
\end{define}
%
%
\begin{lem} \label{lem:J-adj}
\mbox{}
\begin{enu}

\item For each $\xi \in Q^{\vee}$, there exists 
a unique $\phi_J (\xi) \in Q_J^{\vee}$ such that 
$\xi + \phi_J (\xi) \in \jad$. In particular, 
$\xi \in \jad$ if and only if $\phi_{J}(\xi)=0$. 

\item For each $\xi \in Q^{\vee}$, 
the element $\PJ (t_{\xi}) \in (W^{J})_{\af}$ is of the form 
$\PJ (t_{\xi}) = z_{\xi} t_{\xi + \phi_J (\xi)}$ 
for some $z_{\xi} \in W_J$. Therefore, by Lemma~\ref{lem:PJ}, 
$\PJ (w t_{\xi}) = 
\mcr{w} z_{\xi} t_{\xi + \phi_J (\xi)}$ 
for every $w \in W$ and $\xi \in Q^{\vee}$.

\item 
We have 
%
%
\begin{equation}\label{eq:W^J_af}
(W^J)_{\af} = 
\bigl\{ w z_{\xi} t_{\xi} \mid w \in W^J,\,\xi \in \jad \bigr\}.
\end{equation}

\end{enu}
\end{lem}

\begin{proof}
Parts (1), (2), and (3) follow immediately 
from \cite[(3.6), (3.7), and Lemma~3.7]{LNSSS13a}, respectively. 
\end{proof}

Now we prove a few easy lemmas, which will be used later. 
%
%
\begin{lem}\label{lem:r_i}
Let $x \in (W^J)_{\af}$ and $i \in I_{\af}$. 
Then, $x^{-1} \alpha_i \notin (\Delta_J)_{\af}$ if and only if 
$r_i x \in (W^J)_{\af}$.
\end{lem}

\begin{proof}
First, let us show the ``only if'' part; 
by definition \eqref{eq:Pet} of $(W^J)_{\af}$, 
it suffices to show that $r_i x \beta \in \prr$ 
for all $\beta \in (\Delta_J)_{\af}^+$. 
Let $\beta \in (\Delta_J)_{\af}^+$. 
Since $x \in (W^J)_{\af}$, we have $x \beta \in \prr$. 
Also, since $x^{-1} \alpha_i \notin (\Delta_J)_{\af}$ by the assumption, 
it follows that $x^{-1} \alpha_i \neq \beta$, 
and hence $x \beta \neq \alpha_i$. Therefore, we obtain 
$r_i x \beta \in \prr$. 

Next, let us show the ``if'' part. 
Suppose, for a contradiction, 
that $x^{-1} \alpha_i \in (\Delta_J)_{\af}$. 
Then, since $r_{x^{-1}\alpha_i} \in (W_{J})_{\af}$, 
we see by Proposition~\ref{prop:P} that 
$\PJ(r_i x) = \PJ(xr_{x^{-1}\alpha_{i}}) = \PJ(x)$. 
Since $r_{i}x,\,x \in (W^{J})_{\af}$ by the assumption, we have 
$\PJ(r_i x)=r_{i}x$ and $\PJ(x)=x$. Therefore, we obtain $r_{i}x=x$, 
which is a contradiction. Thus, 
$x^{-1} \alpha_i \not\in (\Delta_J)_{\af}$. 
This proves the lemma.
\end{proof}
%
%
\begin{lem} \label{lem:xt}
Let $x \in W_{\af}$, and $\xi \in \jad$. 
Then, $x z_{\xi} t_{\xi} \in (W^J)_{\af}$ 
if and only if $x \in (W^J)_{\af}$.
\end{lem}

\begin{proof}
First, we remark that 
\begin{align}
\PJ (x z_{\xi} t_{\xi}) 
& = \PJ (x z_{\xi}) \PJ (t_{\xi}) 
  \qquad \text{by Lemma \ref{lem:PJ}\,(2)} \notag \\
&= \PJ (x) z_{\xi} t_{\xi} 
  \qquad \text{by Lemmas \ref{lem:PJ}\,(1) and \ref{lem:J-adj}\,(1), (2)}. 
\label{eq:xzt}
\end{align}
Let us prove the ``only if" part. 
Assume that $x z_{\xi} t_{\xi} \in (W^J)_{\af}$; 
note that $\PJ (x z_{\xi} t_{\xi}) = x z_{\xi} t_{\xi}$. 
Combining this equality and \eqref{eq:xzt}, 
we obtain $\PJ (x) z_{\xi} t_{\xi} = x z_{\xi} t_{\xi}$, 
and hence $\PJ (x) =x$, 
which implies that $x \in (W^J)_{\af}$. 
Now, let us prove the ``if" part. 
Assume that $x \in (W^J)_{\af}$; 
note that $\PJ (x) =x$. 
Combining this equality and \eqref{eq:xzt}, 
we obtain $\PJ (x z_{\xi} t_{\xi}) = 
\PJ (x) z_{\xi} t_{\xi} = x z_{\xi} t_{\xi}$, 
which implies that $x z_{\xi} t_{\xi} \in (W^J)_{\af}$. 
This proves the lemma.
\end{proof}
%
%
\subsection{Semi-infinite Bruhat graphs.} 
\label{subsec:SBG}

\begin{define}[{\cite{Pet97}}] \label{def:sil}
Let $x \in W_{\af}$, and 
write it as $x = v t_{\zeta }$ with $v \in W$ and $\zeta \in Q^{\vee}$. 
Then we define the semi-infinite length $\sil(x)$ of $x$ by
\begin{equation*}
\sil (x) := \ell (v) + 2 \pair{\zeta}{\rho}, 
\end{equation*}
where $\rho := (1/2)\sum_{\alpha \in \Delta^+} \alpha$. 
\end{define}

\begin{define}\label{def:sib}
Let $J$ be a subset of $I$. 
\begin{enu}

\item
Define the (parabolic) semi-infinite Bruhat graph $\SB$ 
to be the $\prr$-labeled, directed graph with vertex set $(W^J)_{\af}$ 
and $\prr$-labeled, directed edges of the following form:
$x \edge{\beta} r_{\beta} x$ for $x \in (W^J)_{\af}$ and $\beta \in \prr$, 
where $r_{\beta } x \in (W^J)_{\af}$ and 
$\sil (r_{\beta} x) = \sil (x) + 1$.

\item
The semi-infinite Bruhat order is a partial order 
$\le_{\si}$ on $(W^J)_{\af}$ defined as follows: 
for $x,\,y \in (W^J)_{\af}$, we write $x \le_{\si} y$ 
if there exists a directed path from $x$ to $y$ in $\SB$. 

\end{enu}
\end{define}


%
\section{Isomorphism theorem.}
\label{sec:main}
%
%
\subsection{Semi-infinite Lakshmibai-Seshadri paths.}
\label{subsec:SLS}

\begin{define}\label{def:QBaf(a)}
Let $\lambda \in P^+$, and set 
$J=J_{\lambda}:= \bigl\{ i \in I \mid 
\pair{\alpha_i^{\vee}}{\lambda}=0 \bigr\} \subset I$.
For a rational number $0 < a \le 1$, 
define $\SBa$ to be the subgraph of $\SB$ 
with the same vertex set but having only the edges of the form:
%
\begin{align}\label{eq:siba}
x \edge{\beta} y \quad \text{with} \quad 
a\pair{\beta^{\vee}}{x\lambda} \in \BZ;
\end{align}
note that $\SBb{1} = \SB$. 
\end{define}

\begin{define}\label{def:SLSpath}
Let $\lambda \in P^+$, and set $J:=J_{\lambda}=
\bigl\{ i \in I \mid 
\pair{\alpha_i^{\vee}}{\lambda}=0 \bigr\}$. 
A semi-infinite Lakshmibai-Seshadri (SiLS for short) path of 
shape $\lambda $ is, by definition, a pair $(\bm{x} ; \bm{a})$ of 
a decreasing sequence $\bm{x} : x_1 >_{\si} \cdots >_{\si} x_s$ 
of elements in $(W^J)_{\af}$ and an increasing sequence 
$\bm{a} : 0 = a_0 < a_1 < \cdots  < a_s =1$ of rational numbers 
satisfying the condition that there exists a directed path 
from $x_{u+1}$ to  $x_u$ in $\SBb{a_u}$ 
for each $u = 1,2,\ldots , s-1$;
we express this element as:
\begin{equation*}
0= a_{0} \rsa{x_1} a_{1} \rsa{x_2} \cdots 
\rsa{x_{s-1}} a_{s-1}
\rsa{x_{s}} a_{s}=1.
\end{equation*}
Denote by $\sLS$ the set of all SiLS paths of shape $\lambda$.
\end{define}

For the rest of this subsection, 
we fix $\lambda \in P^{+}$, and set $J:=J_{\lambda}=
\bigl\{i \in I \mid \pair{\alpha_{i}^{\vee}}{\lambda}=0\bigr\}$. 
For $\eta = (x_1,\,\dots,\,x_s ; a_0,\,\dots,\,a_s) \in \sLS$, 
we set
\begin{equation*}
\ol{\eta}:=
(x_1\lambda,\,\dots,\,x_s\lambda ; a_0,\,\dots,\,a_s).
\end{equation*}
\begin{prop}[which will be proved in \S\ref{subsec:SB-LP}] \label{prop:pi_}
It holds that $\ol{\eta} \in \LS$ 
for every $\eta \in \sLS$. 
Thus we obtain a map $\ol{\phantom{\eta}}:\sLS \rightarrow \LS$, 
$\eta \mapsto \ol{\eta}$. 
\end{prop}

We equip the set $\sLS$ with a crystal structure with weights in $P_{\af}$ 
as follows. We define $\wt:\sLS \rightarrow P_{\af}$ by: 
$\wt(\eta):=\wt(\ol{\eta})=\ol{\eta}(1) \in P_{\af}$. 
We define operators $e_i$, $f_{i}$, $i \in I_{\af}$, 
which we call root operators for $\sLS$, 
in the same manner as for $\LS$. 
\begin{define} \label{def:e_if_i}
Let $\eta = (x_1,\,\ldots,\,x_s ; a_0,\,\ldots,\,a_s) \in \sLS$, 
and $i \in I_{\af}$. 
\begin{enu}
\item
If $m_i^{\ol{\eta}} = 0$, then we define $e_i \eta := \bzero$. 
If $m_i^{\ol{\eta}} \le -1$, then define $t_{0},\,t_{1} \in [0,1]$ 
by \eqref{eq:t-e}, with $\pi=\ol{\eta}$. 
Let $1 \le p \le q \le s$ be such that 
$a_{p-1} \le t_0 < a_p$ and $t_1 = a_q$. 
Then we define $e_{i}\eta$ to be 
\begin{equation*}
0= a_{0} \rsa{x_1} \cdots \rsa{x_{p-1}} a_{p-1} \rsa{x_p} 
\underbrace{t_{0} \rsa{r_{i}x_p} a_{p} \rsa{r_{i}x_{p+1}} \cdots 
\rsa{r_{i}x_{q}} a_{q}=t_{1}}_{\text{``reflected'' by $r_{i}$}}
\rsa{x_{q+1}} \cdots \rsa{x_{s}} a_{s}=1,
\end{equation*}
that is, 
\begin{align*}
e_i \eta := ( 
& x_1,\,\ldots,\,x_p,\,r_i x_p,\,r_i x_{p+1},\,\ldots,\,
  r_i x_q,\,x_{q+1},\,\ldots,\,x_s ; \\
& a_0,\,\ldots,\,a_{p-1},\,t_0,\,a_p,\,\ldots,\,a_{q}=t_1,\,
\ldots,\,a_s);
\end{align*}
if $t_0 = a_{p-1}$, then we drop $x_p$ and $a_{p-1}$, and 
if $r_i x_q = x_{q+1}$, then we drop $x_{q+1}$ and $a_{q}=t_1$.

\item
If $H_i^{\ol{\eta}}(1) - m_i^{\ol{\eta}} = 0$, 
then we define $f_i \eta := \bzero$. 
If $H_i^{\ol{\eta}}(1) - m_i^{\ol{\eta}}  \ge 1$, 
then define $t_{0},\,t_{1} \in [0,1]$ 
by \eqref{eq:t-f}, with $\pi=\ol{\eta}$. 
Let $0 \le p \le q \le s-1$ be such that $t_0 = a_p$, and 
$a_{q} < t_1 \le a_{q+1}$. Then we define $f_{i}\eta$ to be 
\begin{equation*}
0= a_{0} \rsa{x_1} \cdots \rsa{x_{p}} 
\underbrace{a_{p}=t_{0} \rsa{r_{i}x_{p+1}} 
\cdots \rsa{r_{i}x_{q}} a_{q} \rsa{r_{i}x_{q+1}} t_{1}%
}_{\text{``reflected'' by $r_{i}$}}
\rsa{x_{q+1}} a_{q+1} 
\rsa{x_{q+2}} \cdots \rsa{x_{s}} a_{s}=1,
\end{equation*}
that is, 
\begin{align*}
f_i \pi := ( 
& x_1,\,\ldots,\,x_{p},\,r_ix_{p+1},\,\dots,\,
  r_ix_{q},\,r_ix_{q+1},\,x_{q+1},\,\ldots,\,x_s ; \\
& a_0,\,\ldots,\,a_{p}=t_0,\,\ldots,\,a_{q},\,t_1,\,
  a_{q+1},\,\ldots,\,a_s);
\end{align*}
if $t_1 = a_{q+1}$, then we drop $x_{q+1}$ and $a_{q+1}$, and 
if $x_{p} = r_i x_{p+1}$, then we drop $x_{p}$ and $a_{p}=t_0$.

\item
Set $e_i \bzero = f_i \bzero := \bzero$ for all $i \in I_{\af}$.
\end{enu}
\end{define}

\begin{thm}[which will be proved in \S\ref{subsec:pr-thm:stability}]
\label{thm:stability}
\mbox{}
\begin{enu}
\item
The set $\sLS \sqcup \bigl\{ \bzero \bigr\}$ is 
stable under the action of the root operators 
$e_i$ and $f_i$, $i \in I_{\af}$.

\item
For each $\eta \in \sLS$ and $i \in I_{\af}$, we set 
\begin{equation*}
\begin{cases}
\ve_i (\eta) := 
 \max \bigl\{ n \ge 0 \mid e_i^n \eta \neq \bzero \bigr\}, \\[1.5mm]
\vp_i (\eta) := 
 \max \bigl\{ n \ge 0 \mid f_i^n \eta \neq \bzero \bigr\} .
\end{cases}
\end{equation*}
Then, the set $\sLS$, equipped with the maps $\wt$, $e_i$, $f_{i}$, 
$i \in I_{\af}$, and $\ve_i$, $\vp_i$, $i \in I_{\af}$, defined above, is 
a crystal with weights in $P_{\af}$.
\end{enu}
\end{thm}
%
%
\subsection{Isomorphism theorem between $\B$ and $\sLS$.}

Let $V(\lambda)$ denote the extremal weight module of 
extremal weight $\lambda \in P^+$ over 
the quantized universal enveloping algebra $U_q (\Fg_{\af})$ 
associated with $\Fg_{\af}$, which is 
an integrable $U_q (\Fg_{\af})$-module generated 
by a single element $v_{\lambda }$ with 
the defining relation that $v_{\lambda }$ is 
an ``extremal weight vector'' of weight $\lambda$
(for details, see \cite[\S8]{Kas94} and \cite[\S3]{Kas02b}). 
We know from \cite[\S8]{Kas94} that $V(\lambda)$ has 
a crystal basis $\B$. The main result of this paper is 
the following theorem.

\begin{thm}\label{thm:main}
Let $\lambda \in P^+$. The crystal basis $\B$ of 
the extremal weight module $V(\lambda)$ of 
extremal weight $\lambda $ is isomorphic, as a crystal, to 
the crystal $\sLS$ of SiLS paths of shape $\lambda $.
\end{thm}

Let us give an outline of the proof of Theorem \ref{thm:main}. 
Let $\sLSo$ denote the connected component of $\sLS$ containing 
$\eta_e := (e ; 0,1) \in \sLS$. Also, let $u_{\lambda}$ be 
the element of $\B$ corresponding to the generator 
$v_{\lambda}$ of $V(\lambda )$, and let $\Bo$ denote
the connected component of $\B$ containing 
$u_{\lambda} \in \B$. 

\begin{prop}[which will be proved in \S \ref{sec:pr-prop:conn-isom}]
\label{prop:conn-isom}
There exists a unique isomorphism 
$\Bo \stackrel{\sim}{\rightarrow} \sLSo$ of crystals 
that maps $u_{\lambda}$ to $\eta_e$.
\end{prop}

We write $\lambda \in P^+$ as 
$\lambda = \sum_{i \in I} m_i \varpi_i$ 
with $m_i \in \BZ_{\ge 0}$, $i \in I$, and define
\begin{align}\label{eq:partition}
\Par(\lambda ) := 
 \bigl\{ \bm{\rho} = (\rho^{(i)})_{i \in I} \mid 
 \text{$\rho^{(i)}$ is a partition of length less than $m_i$ 
 for each $i \in I$} \bigr\} ;
\end{align}
we understand that $\rho^{(i)}$ is the empty partition if $m_i = 0$. 
We equip the set $\Par(\lambda)$ with a crystal structure as follows: 
for each $\bm{\rho} = (\rho^{(i)})_{i \in I} \in \Par(\lambda)$, 
we set
\begin{align}
\begin{cases}
e_i \bm{\rho} = f_i \bm{\rho} := \bzero,\ 
\ve_i (\bm{\rho}) = \vp_i (\bm{\rho}) := -\infty 
  & \text{for}\ i\in I_{\af}, \\[1.5mm]
\wt(\bm{\rho}) := - \sum_{i \in I} |\rho^{(i)}| \delta , &
\end{cases}
\end{align}
where for a partition 
$\chi = (\chi_1 \ge \chi_2 \ge \cdots \ge \chi_{k} \ge 0)$, 
we set $|\chi| := \sum_{l=1}^{k} \chi_{l}$. 
By Proposition \ref{prop:conn-isom}, 
we have the isomorphism 
\begin{align} \label{eq:P*B}
\Par(\lambda ) \otimes \Bo \cong \Par(\lambda ) \otimes \sLSo
\end{align}
of crystals. Let $\CB$ be either $\Bo$ or $\sLSo$. 
For each $\bm{\rho} \in \Par(\lambda )$, 
we set $\{ \bm{\rho} \} \otimes \mathcal{B} := 
\bigl\{ \bm{\rho} \otimes b \mid b \in \CB \bigr\} \subset 
\Par(\lambda ) \otimes \CB$. Then it is easily seen from 
the tensor product rule for crystals that 
\begin{equation*}
\Par(\lambda ) \otimes \CB = 
\bigsqcup_{\bm{\rho} \in \Par(\lambda )} 
\{ \bm{\rho} \} \otimes \mathcal{B}
\end{equation*}
is the decomposition of $\Par(\lambda ) \otimes \CB$ into 
its connected components. Moreover, 
the map $\CB \rightarrow \{ \bm{\rho} \} \otimes \mathcal{B}$, 
$b \mapsto \bm{\rho} \otimes b$, is 
bijective and commutes with Kashiwara operators. 

Now, we know the following proposition 
from \cite[Theorem 4.16\,(i)]{BN04}. 

\begin{prop}\label{prop:BN}
For $\lambda \in P^+$, there exists an isomorphism 
$\B \stackrel{\sim}{\rightarrow} \Par(\lambda) \otimes \Bo$ of crystals. 
\end{prop}

Also, we have the following proposition. 

\begin{prop}[which will be proved in \S \ref{subsec:prf-BNforSLS}] 
\label{prop:BNforSLS}
For $\lambda \in P^+$, 
there exists an isomorphism $\sLS \stackrel{\sim}{\rightarrow} 
\Par(\lambda ) \otimes \sLSo$ of crystals. 
\end{prop}

Combining all the results above, we finally obtain
\begin{align*}
\B & \cong \Par(\lambda ) \otimes \Bo \quad 
     \text{by Proposition \ref{prop:BN}} \\
& \cong \Par(\lambda ) \otimes \sLSo \quad 
     \text{by \eqref{eq:P*B}} \\
& \cong \sLS \quad 
     \text{by Proposition \ref{prop:BNforSLS}},
\end{align*}
as desired. 

In the remainder of this paper, we will give proofs of 
the results above; we prove 
Proposition~\ref{prop:pi_} in \S\ref{subsec:SB-LP}, 
Theorem~\ref{thm:stability} in \S\ref{subsec:pr-thm:stability}, 
Proposition~\ref{prop:conn-isom} in \S\ref{sec:pr-prop:conn-isom}, 
and Proposition~\ref{prop:BNforSLS} 
in \S\ref{subsec:prf-BNforSLS}.
%
%
\section{Proofs of Proposition \ref{prop:pi_} and Theorem \ref{thm:stability}.} \label{sec:pr-prop:pi_}
%
%
\subsection{Some technical lemmas.} 
\label{subsec:basiclem}

\begin{lem}[{\cite[Lemma 4.3]{BFP99}}] \label{lem:ineq}
We have $\ell (r_{\alpha }) \le 2 \pair{\alpha^{\vee}}{\rho}-1$ 
for all $\alpha \in \Delta^+$.
\end{lem}
%
%
\begin{lem} \label{lem:sil}
Let $x=vt_{\zeta} \in W_{\af}$ with $v \in W$ and $\zeta \in Q^{\vee}$, 
and let $\beta=\alpha+n\delta \in \prr$ with $\alpha \in \Delta$ and 
$n \in \BZ_{\ge 0}$. Then, $\ell^{\si}(r_{\beta}x) > \ell^{\si}(x)$ 
if and only if $v^{-1}\alpha$ is a positive root. 
In particular, if $\beta=\alpha_{i}$ for some $i \in I_{\af}$ 
(note that $\alpha=\alpha_{i}$ and $n=0$ if $i \in I$, and 
$\alpha=-\theta$ and $n=1$ if $i=0$), then 
%
%
\begin{equation} \label{eq:simple}
\sil(r_{i}x)=
 \begin{cases}
 \sil(x)+1 & \text{\rm if $v^{-1}\alpha$ is a positive root}, \\[1mm]
 \sil(x)-1 & \text{\rm if $v^{-1}\alpha$ is a negative root}.
 \end{cases}
\end{equation}
\end{lem}

\begin{proof}
Assume that $\alpha \in \Delta$ is a negative root; 
note that $n \ge 1$. We see by \eqref{eq:refl} that 
$r_{\beta}x=r_{\alpha}t_{n\alpha^{\vee}}vt_{\zeta}=
vr_{v^{-1}\alpha} t_{nv^{-1}\alpha^{\vee}+\zeta}$. 
If $v^{-1}\alpha$ is a positive root, then we have 
\begin{align*}
\ell^{\si}(r_{\beta}x) 
 & = \ell(vr_{v^{-1}\alpha}) + 2\pair{nv^{-1}\alpha^{\vee}+\zeta}{\rho}
   \ge \ell(v) - \ell(r_{v^{-1}\alpha}) + 
       2n\pair{v^{-1}\alpha^{\vee}}{\rho}+2\pair{\zeta}{\rho} \\
 & \ge \ell(v) - 2\pair{v^{-1}\alpha^{\vee}}{\rho}+1+
        2n\pair{v^{-1}\alpha^{\vee}}{\rho}+2\pair{\zeta}{\rho} 
   \quad \text{by Lemma~\ref{lem:ineq}} \\
 & \ge \ell(v)+2\pair{\zeta}{\rho}+1
   \quad \text{since $v^{-1}\alpha \in \Delta^{+}$ and $n \ge 1$} \\
 & =\ell^{\si}(x)+1 > \ell^{\si}(x).
\end{align*}
If $v^{-1}\alpha$ is a negative root, then we have
\begin{align*}
\ell^{\si}(r_{\beta}x) 
 & = \ell(vr_{v^{-1}\alpha}) + 2\pair{nv^{-1}\alpha^{\vee}+\zeta}{\rho}
   \le \ell(v) + \ell(r_{v^{-1}\alpha}) + 
       2n\pair{v^{-1}\alpha^{\vee}}{\rho}+2\pair{\zeta}{\rho} \\
 & \le \ell(v) + 2\pair{v^{-1}\alpha^{\vee}}{\rho}-1+
        2n\pair{v^{-1}\alpha^{\vee}}{\rho}+2\pair{\zeta}{\rho} 
   \quad \text{by Lemma~\ref{lem:ineq}} \\
 & \le \ell(v)+2\pair{\zeta}{\rho}-1
   \quad \text{since $-v^{-1}\alpha \in \Delta^{+}$ and $n \ge 1$} \\
 & =\ell^{\si}(x)-1 < \ell^{\si}(x).
\end{align*}
The proof for the case that $\alpha \in \Delta$ 
is a positive root is easier. 

In order to show \eqref{eq:simple}, it suffices to verify that 
$\sil(r_{i}x)=\sil(x) \pm 1$. If $i \in I$, then it is obvious. 
Assume that $i=0$. 
We see from \cite[Proposition~5.11]{LNSSS13a} that 
if $v^{-1}\theta$ is a negative root, then 
$\ell(r_{\theta}v)=\ell(v)-2\pair{-v^{-1}\theta^{\vee}}{\rho}+1$.
The desired equality above for $i=0$ follows easily from 
this equality, together with the computation above. 
This proves the lemma. 
\end{proof}
%
%
\begin{rem} \label{rem:simple}
Let $\lambda \in P^{+}$, and set 
$J:=J_{\lambda}=\bigl\{i \in I \mid 
\pair{\alpha_{i}^{\vee}}{\lambda}=0\bigr\}$. 
For each $x \in (W^{J})_{\af}$ and $i \in I_{\af}$, 
we deduce from Lemma~\ref{lem:r_i} that 
$r_{i}x \in (W^{J})_{\af}$ if and only if 
$\pair{\alpha_{i}^{\vee}}{x\lambda} \ne 0$, and 
from \eqref{eq:simple} that 
%
%
\begin{equation} \label{eq:simple2}
\begin{cases}
x \edge{\alpha_{i}} r_{i}x \quad \text{in $\SB$} 
\iff \pair{\alpha_{i}^{\vee}}{x\lambda} > 0
\iff x\lambda \edge{\alpha_{i}} r_{i}x\lambda \quad \text{in $\LP$}, 
& \\[1.5mm]
r_{i}x \edge{\alpha_{i}} x \quad \text{in $\SB$} 
\iff \pair{\alpha_{i}^{\vee}}{x\lambda} < 0
\iff r_{i}x\lambda \edge{\alpha_{i}} x\lambda \quad 
\text{in $\LP$}. &
\end{cases}
\end{equation}
\end{rem}

We prove the ``diamond lemma'' for $\SB$ 
(cf. \cite[Lemma~4.1]{Lit95} and \cite[Lemma~5.14]{LNSSS13a}). 
%
%
\begin{lem} \label{lem:dia1}
Let $J$ be a subset of $I$. 
Let $x \in (W^{J})_{\af}$, $\beta \in \prr$, and $i \in I_{\af}$. 
Assume that $x \edge{\beta} r_{\beta}x=:y$ and 
$x \edge{\alpha_i} r_{i}x$ in $\SB$.
If we write $y^{-1}\alpha_{i} \in \rr$ as 
$y^{-1}\alpha_{i}=\gamma+n\delta$ 
with $\gamma \in \Delta$ and $n \in \BZ$, 
then $\gamma \notin \Delta_{J}^{+}$. 
Moreover, if $\gamma \in \Delta^{+} \setminus  \Delta_{J}^{+}$, 
then $\beta \ne \alpha_{i}$, and we have dotted edges in $\SB$ 
as in the following diagram: 
\begin{equation*}
\begin{diagram}
\node{x}\arrow{s,l}{\alpha_{i}} \arrow{e,t}{\beta} 
\node{y}\arrow{s,r,..}{\alpha_{i}} \\
\node{r_{i}x} \arrow{e,t,..}{r_{i}\beta} \node{r_{i}y}
\end{diagram}
\end{equation*}
If $\gamma$ is a negative root, 
then $\beta=\alpha_{i}$. 
\end{lem}

\begin{proof}
First, suppose, for a contradiction, that $\gamma \in \Delta_{J}^{+}$.
Then, we must have $y^{-1}\alpha_{i} \in \prr$. Indeed, suppose that 
$y^{-1}\alpha_{i}$ is a negative real root. 
Since $\gamma \in \Delta_{J}^{+}$ by our assumption, we obtain 
$y^{-1}\alpha_{i} \in (\Delta_{J})_{\af} \cap 
(-\Delta_{\af}^{+})$. Since $y \in (W^{J})_{\af}$ 
by the assumption, it follows from the definition of $(W^{J})_{\af}$ 
that $\alpha_{i}=yy^{-1}\alpha_{i}$ is a negative real root, 
which is a contradiction. Thus, we have $y^{-1}\alpha_{i} \in \prr$, 
and hence $y^{-1}\alpha_{i} \in (\Delta_{J})_{\af}^{+}$. 
Here, since $x \in (W^{J})_{\af}$, we see that 
$\prr \ni xy^{-1}\alpha_{i}=r_{\beta}\alpha_{i}=
\alpha_{i}-\pair{\beta^{\vee}}{\alpha_{i}}\beta$. 
Therefore, we deduce that 
$\pair{\beta^{\vee}}{\alpha_{i}} \le 0$; 
in particular, $\beta \ne \alpha_{i}$. 
We write $x^{-1}\alpha_{i}$ and $x^{-1}\beta$ as 
$x^{-1}\alpha_{i}=\gamma_{1}+n_1\delta$ and 
$x^{-1}\beta=\gamma_{2}+n_2\delta$, 
with $\gamma_1,\,\gamma_2 \in \Delta$ and 
$n_1,\,n_2 \in \BZ$, respectively. 
Since we have $x \edge{\beta} r_{\beta}x=y$ 
and $x \edge{\alpha_{i}} r_{i}x$ in $\SB$ by the assumption, 
we see by Lemma~\ref{lem:sil} that 
$\gamma_1,\,\gamma_2 \in \Delta^{+}$. 
Also, since $x$, $r_{\beta}x=xr_{x^{-1}\beta}$, and 
$r_{i}x=xr_{x^{-1}\alpha_{i}}$ 
are contained in $(W^{J})_{\af}$, we see 
from Proposition~\ref{prop:P} that $x^{-1}\beta,\,x^{-1}\alpha_{i} 
\notin (\Delta_{J})_{\af}$, and hence $\gamma_{1},\,\gamma_{2} \in 
\Delta^{+} \setminus \Delta_{J}^{+}$. 
Therefore, we deduce that 
$\gamma_{1}-\pair{\beta^{\vee}}{\alpha_{i}}\gamma_{2} 
\notin \Delta_{J}^{+}$, since 
$\pair{\beta^{\vee}}{\alpha_{i}} \le 0$ as seen above. 
Because
\begin{align*}
\gamma+n\delta & =
y^{-1}\alpha_{i} = x^{-1}r_{\beta}\alpha_{i}=
x^{-1}\alpha_{i}-\pair{\beta^{\vee}}{\alpha_{i}}x^{-1}\beta \\
& =
\bigl\{\gamma_{1}-\pair{\beta^{\vee}}{\alpha_{i}}\gamma_{2}\bigr\}
+\bigl\{n_{1}-\pair{\beta^{\vee}}{\alpha_{i}}n_{2}\bigr\}\delta,
\end{align*}
we obtain $\gamma=\gamma_{1}-\pair{\beta^{\vee}}{\alpha_{i}}\gamma_{2}
\notin \Delta_{J}^{+}$, which contradicts our assumption that 
$\gamma \in \Delta_{J}^{+}$. 
Thus, we conclude that $\gamma \notin \Delta_{J}^{+}$, as desired. 

Assume that 
$\gamma \in \Delta^{+} \setminus  \Delta_{J}^{+}$. 
Suppose, for a contradiction, that $\beta=\alpha_{i}$. 
Then we have $\gamma+n\delta=
y^{-1}\alpha_{i} = x^{-1}r_{\beta}\alpha_{i}=
x^{-1}r_{i}\alpha_{i}=-x^{-1}\alpha_{i}$.
Since $x \edge{\alpha_{i}} r_{i}x$ in $\SB$ by the assumption, 
it follows immediately from Lemma~\ref{lem:sil} that 
$x^{-1}\alpha_{i}=\gamma_{1}+n_{1}\delta$ for some 
$\gamma_{1} \in \Delta^{+}$ and $n_{1} \in \BZ$.
Hence we obtain $\gamma=-\gamma_{1} \in (-\Delta^{+})$, 
which contradicts our assumption.
Thus, we conclude that $\beta \ne \alpha_{i}$. 
Also, since $\pair{\alpha_{i}^{\vee}}{y\lambda}=
\pair{y^{-1}\alpha_{i}^{\vee}}{\lambda}=
\pair{\gamma^{\vee}}{\lambda} > 0$, 
it follows immediately from \eqref{eq:simple2} that 
$y \edge{\alpha_{i}} r_{i}y$ in $\SB$. Therefore, we have 
\begin{align*}
\sil(r_{i}y) & =
\underbrace{\sil(r_{i}y)-\sil(y)}_{=1} + 
\underbrace{\sil(y)-\sil(x)}_{=1} + \sil(x) \\[1.5mm]
& = \sil(x)+2=\sil(r_{i}x)-1+2=\sil(r_{i}x)+1.
\end{align*}
From this equation, we conclude that 
$r_{i}x \edge{r_{i}\beta} r_{i}y$. 

Assume that $\gamma$ is a negative root, and 
suppose, for a contradiction, that $\beta \ne \alpha_{i}$. 
Then we have $r_{i}\beta \in \Delta_{\af}^{+}$.
Since $x \edge{\beta} y$ in $\SB$ by the assumption, 
it follows from Lemma~\ref{lem:sil} that $\sil(r_{i}x) < \sil(r_{i}y)$. 
Since $\pair{\alpha_{i}^{\vee}}{r_{i}y\lambda}=
-\pair{y^{-1}\alpha_{i}^{\vee}}{\lambda}=
-\pair{\gamma^{\vee}}{\lambda} > 0$, 
we see by \eqref{eq:simple2}
that $r_{i}y \edge{\alpha_{i}} y$ in $\SB$. Therefore, we have
$\sil(x) < \sil(r_{i}x) < \sil(r_{i}y) < \sil(y)$, 
which contradicts the equality $\sil(y)=\sil(x)+1$. 
This completes the proof of the lemma. 
\end{proof}

The following lemma can be proved in exactly the same way as 
Lemma~\ref{lem:dia1}.
%
%
\begin{lem} \label{lem:dia2}
Let $J$ be a subset of $I$. 
Let $x \in (W^{J})_{\af}$, $\beta \in \prr$, and $i \in I$. 
Assume that $y:=r_{\beta}x \edge{\beta} x$ and 
$r_{i}x \edge{\alpha_i} x$ in $\SB$.
If we write $y^{-1}\alpha_{i} \in \rr$ as
$y^{-1}\alpha_{i}=-\gamma+n\delta$ 
with $\gamma \in \Delta$ and $n \in \BZ$, 
then $\gamma \notin \Delta_{J}^{+}$. 
Moreover, if $\gamma \in \Delta^{+} \setminus  \Delta_{J}^{+}$, 
then $\beta \ne \alpha_{i}$, and 
we have dotted edges in $\SB$ as in the following diagram: 
\begin{equation*}
\begin{diagram}
\node{r_i y}\arrow{s,l,..}{\alpha_{i}} \arrow{e,t,..}{\beta} 
\node{r_i x}\arrow{s,r}{\alpha_{i}} \\
\node{y} \arrow{e,t}{\beta} \node{x}
\end{diagram}
\end{equation*}
If $\gamma$ is a negative root, 
then $\beta=\alpha_{i}$. 
\end{lem}
An inductive argument, which uses 
Lemmas~\ref{lem:dia1} and \ref{lem:dia2}, proves 
the following proposition. 
%
%
\begin{lem} \label{lem:directed}
Let $\lambda \in P^{+}$, and set 
$J:=J_{\lambda}=\bigl\{i \in I \mid 
\pair{\alpha_{i}^{\vee}}{\lambda}=0\bigr\}$. 
Let $0 < a \le 1$ be a rational number. 
Let $x,\,y \in (W^{J})_{\af}$, and assume that 
there exists a directed path 
$x=y_{0} \edge{\beta_{1}} y_{1} \edge{\beta_{2}}
 \cdots \edge{\beta_{k}} y_{k}=y$
from $x$ to $y$ in $\SBa$. Let $i \in I_{\af}$. 

\begin{enu}

\item If $\pair{\alpha_{i}^{\vee}}{y_{m}\lambda} > 0$ 
for all $0 \le m \le k$, or 
if $\pair{\alpha_{i}^{\vee}}{y_{m}\lambda} < 0$ 
for all $0 \le m \le k$, then there exists a directed path 
from $r_{i}x$ to $r_{i}y$ in $\SBa$ of the form: 
\begin{equation*}
r_{i}x=r_{i}y_{0} \edge{r_{i}\beta_{1}} 
r_{i}y_{1} \edge{r_{i}\beta_{2}}
 \cdots \edge{r_{i}\beta_{k}} r_{i}y_{k}=r_{i}y.
\end{equation*}

\item Assume that $\pair{\alpha_{i}^{\vee}}{x\lambda} > 0$, 
and $\pair{\alpha_{i}^{\vee}}{y_{m}\lambda} \le 0$ 
for some $1 \le m \le k$; let $l$ denote the minimum of all such $m$'s. 
Then, $\beta_{l}=\alpha_{i}$, and 
there exists a directed path 
from $r_{i}x$ to $y$ in $\SBa$ of the form:
\begin{equation*}
r_{i}x=r_{i}y_{0} \edge{r_{i}\beta_{1}} 
 \cdots 
\edge{r_{i}\beta_{l-1}} 
r_{i}y_{l-1}=y_{l} \edge{\beta_{l+1}}
 \cdots \edge{\beta_{k}} y_{k}=y.
\end{equation*}

\item Assume that $\pair{\alpha_{i}^{\vee}}{y\lambda} < 0$, 
and $\pair{\alpha_{i}^{\vee}}{y_{m}\lambda} \ge 0$ 
for some $0 \le m \le k-1$; let $l$ denote 
the maximum of all such $m$'s. 
Then, $\beta_{l+1}=\alpha_{i}$, and 
there exists a directed path 
from $x$ to $r_{i}y$ in $\SBa$ of the form:
\begin{equation*}
x=y_{0} \edge{\beta_{1}} 
 \cdots 
\edge{\beta_{l}} 
y_{l}= r_{i} y_{l+1} \edge{r_i \beta_{l+2}}
 \cdots \edge{r_i \beta_{k}} r_{i}y_{k}=r_{i}y.
\end{equation*}
\end{enu}
\end{lem}
%
%
\begin{lem} \label{lem:zt}
Let $J$ be a subset of $I$. 
Let $\xi \in \jad$, and $\beta \in \prr$. 
If $z_{\xi}t_{\xi} \edge{\beta} 
r_{\beta}z_{\xi}t_{\xi}$ in $\SB$, then 
$\beta=\alpha_{i}$ for some $i \in I \setminus J$. 
\end{lem}

\begin{proof}
Write $\beta$ as $\beta=\alpha+n\delta$, 
with $\alpha \in \Delta$ and $n \in \BZ_{\ge 0}$. 
Since $z_{\xi}t_{\xi},\,r_{\beta}z_{\xi}t_{\xi} \in (W^{J})_{\af}$ by 
the assumption, we see by Lemma~\ref{lem:xt} that 
$r_{\beta} \in (W^{J})_{\af}$, 
which implies that $\beta \notin (\Delta_{J})_{\af}$ and 
hence $\alpha \notin \Delta_{J}$. 
Also, since $\sil(r_{\beta}z_{\xi}t_{\xi}) = 
\sil(z_{\xi}t_{\xi})+1 > \sil(z_{\xi}t_{\xi})$ by the assumption, 
we see from Lemma~\ref{lem:sil} that 
$z_{\xi}^{-1}\alpha$ is a positive root. 
Because $\alpha \notin \Delta_{J}$ as seen above, 
and $z_{\xi} \in W_{J}$, we deduce that 
$\alpha \in \Delta^{+}$, and hence 
$\alpha \in \Delta^{+} \setminus \Delta_{J}^{+}$. 
We claim that $r_{\alpha} \in W^{J}$. 
Indeed, if $n=0$, then we have $r_{\alpha}=r_{\beta} 
\in (W^{J})_{\af}$. Hence we obtain 
$r_{\alpha}=\Pi^{J}(r_{\alpha})=\mcr{r_{\alpha}}$ 
by Lemma~\ref{lem:PJ}\,(1), which implies that 
$r_{\alpha} \in W^{J}$. Assume that $n \ge 1$. 
By \eqref{eq:refl}, we have $r_{\beta}=
r_{\alpha}t_{n\alpha^{\vee}}$. 
Since $r_{\beta} \in (W^{J})_{\af}$ as seen above, 
we have $n\alpha^{\vee} \in \jad$ by \eqref{eq:W^J_af}. 
Therefore, we have 
$\pair{n\alpha^{\vee}}{\gamma} \in \bigl\{0,\,-1\bigr\}$ 
for all $\gamma \in \Delta_{J}^{+}$, and in particular, 
$\pair{\alpha^{\vee}}{\gamma} \le 0$ 
for all $\gamma \in \Delta_{J}^{+}$. 
From this, we see that $r_{\alpha}\gamma=
\gamma-\pair{\alpha^{\vee}}{\gamma}\alpha \in \Delta^{+}$ 
for all $\gamma \in \Delta_{J}^{+}$ since $\alpha \in \Delta^{+}$, 
which implies that 
$r_{\alpha} \in W^{J}$ by \eqref{eq:W^J}, as claimed. 
Now we compute
\begin{align*}
\sil(r_{\beta}z_{\xi}t_{\xi}) & = 
\sil(r_{\alpha}z_{\xi}t_{\xi+nz_{\xi}^{-1}\alpha^{\vee}}) = 
\ell(r_{\alpha}z_{\xi})+2\pair{\xi+nz_{\xi}^{-1}\alpha^{\vee}}{\rho} \\
& = 
\ell(r_{\alpha})+
\underbrace{\ell(z_{\xi})+2\pair{\xi}{\rho}}_{=\sil(z_{\xi}t_{\xi})}+
2n\pair{z_{\xi}^{-1}\alpha^{\vee}}{\rho} \quad 
 \text{since $r_{\alpha} \in W^{J}$ and $z_{\xi} \in W_{J}$}.
\end{align*}
Since $\sil(r_{\beta}z_{\xi}t_{\xi})=\sil(z_{\xi}t_{\xi})+1$ 
by the assumption, it follows from this equation that 
$\ell(r_{\alpha})+2n\pair{z_{\xi}^{-1}\alpha^{\vee}}{\rho}=1$.
Here, recall that $z_{\xi}^{-1}\alpha$ is a positive root, 
and that $n \ge 0$. Therefore, we obtain $\ell(r_{\alpha})=1$ and $n=0$, 
which implies that $\beta=\alpha_{i}$ 
for some $i \in I \setminus J$. This proves the lemma. 
\end{proof}

%
\subsection{Proof of Proposition~\ref{prop:pi_}.} 
\label{subsec:SB-LP}

In this subsection, we fix $\lambda \in P^+$, and set 
$J:=J_{\lambda} = \bigl\{ i \in I \mid 
\pair{\alpha_i^{\vee}}{\lambda}=0\bigr\}$. 

Let $\eta = 
(x_1,\,x_2,\,\ldots,\,x_s ; a_0,\,a_1,\,\ldots,\,a_s) \in \sLS$; 
recall that 
\begin{align*}
\ol{\eta}=
(x_1 \lambda,\,x_2 \lambda,\,\ldots,\,x_s\lambda ; 
 a_0,\,a_1,\,\ldots,\,a_s ).
\end{align*}
Hence it suffices to show that 
$x_u \lambda > x_{u+1} \lambda$, and that 
there exists a directed path from 
$x_{u+1} \lambda$ to $x_{u} \lambda$ in $\LPb{a_{u}}$. 
This follows immediately from the next proposition.

\begin{prop} \label{prop:sib-lv0}
Let $0 < a \le 1$ be a rational number, $x \in (W^J)_{\af}$, and $\beta \in \prr$. Then, $x \edge{\beta} r_{\beta} x$  in $\SBa$ if and only if 
$x \lambda \edge{\beta} r_{\beta} x \lambda $ in $\LPa$.
\end{prop}

\begin{proof}
In view of conditions \eqref{eq:Xa} and \eqref{eq:siba}, 
we need only show that 
%
%
\begin{equation} \label{eq:iff1}
\text{$x \edge{\beta} r_{\beta} x$  in $\SB$ if and only if 
$x \lambda \edge{\beta} r_{\beta} x \lambda $ in $\LP$};
\end{equation}
write $\beta \in \prr$ as
$\beta=\alpha+n\delta$ with $\alpha \in \Delta$ 
and $n \in \BZ_{\ge 0}$.
We deduce from \cite[Lemma~1.4]{AK} that 
there exist $i_{1},\,i_{2},\,\dots,\,i_{k} \in I_{\af}$ 
such that
%
%
\begin{equation} \label{eq:ql}
x\lambda=\mu_{k} \edge{\alpha_{i_k}} \mu_{k-1} \edge{\alpha_{i_{k-1}}} 
\cdots \edge{\alpha_{i_2}} \mu_{1} \edge{\alpha_{i_1}} \mu_{0} 
\quad \text{in $\LP$},
\end{equation}
with $\mu_{0} \equiv \lambda \mod \BQ\delta$, where 
$\mu_{m}=r_{i_{m+1}} \cdots r_{i_{k}}x\lambda$ 
for $0 \le m \le k$. We will show \eqref{eq:iff1} 
by induction on the length $k$ of the directed path \eqref{eq:ql}. 

\paragraph{Step 1.}
If $k=0$, then we have $x\lambda \equiv 
\lambda \mod \BQ\delta$, which implies that 
$x \in (W^{J})_{\af}$ is of the form $x=z_{\xi}t_{\xi}$ 
for some $\xi \in \jad$ (see \eqref{eq:W^J_af}). 
First, let us show the ``only if'' part of \eqref{eq:iff1} 
in this case. By Lemma~\ref{lem:zt}, we have 
$\beta=\alpha_{i}$ for some $i \in I \setminus J$. 
Hence it follows immediately from \eqref{eq:simple2} that 
$x \lambda \edge{\alpha_i=\beta} r_{i} x \lambda =
r_{\beta} x \lambda $ in $\LP$.

Next, let us show the ``if part'' of \eqref{eq:iff1} 
in the case that $x=z_{\xi}t_{\xi}$ with $\xi \in \jad$. 
We see from \cite[Lemma~2.11]{NS08} that 
$\beta$ is of the form $\beta=\alpha$ with 
$\alpha \in \Delta^{+}$, or $\beta=-\alpha+\delta$ 
with $\alpha \in \Delta^{+}$. 
Since $\pair{\beta^{\vee}}{x\lambda} > 0$, 
and $x\lambda \equiv \lambda \mod \BQ\delta$, 
it follows immediately that 
$\beta=\alpha \in \Delta^{+} \setminus \Delta_{J}^{+}$. 
Let $\mcr{r_{\alpha}}=r_{j_{p}} \cdots r_{j_{1}}$ be 
a reduced expression of $\mcr{r_{\alpha}} \in W$;
since $x\lambda \equiv \lambda \mod \BQ\delta$, 
we have
\begin{equation*}
\pair{ \alpha_{j_{q}}^{\vee} }{
r_{j_{q-1}} \cdots r_{j_{1}}x\lambda }
=
\pair{ \alpha_{j_{q}}^{\vee} }{
r_{j_{q-1}} \cdots r_{j_{1}}\lambda } > 0
\end{equation*}
for all $1 \le q \le p$. Then, by \eqref{eq:simple2}, 
we deduce that
\begin{equation} \label{eq:step1-d1}
x\lambda \edge{\alpha_{j_1}} r_{j_1}x\lambda 
\edge{\alpha_{j_2}} \cdots \edge{\alpha_{j_p}} 
r_{j_{p}} \cdots r_{j_{1}}x\lambda=
\mcr{r_{\alpha}}x\lambda \quad \text{in $\LP$}.
\end{equation}
Noting that $x\lambda \equiv \lambda \mod \BQ\delta$, 
we see that $\mcr{r_{\alpha}}x\lambda = 
r_{\alpha}x\lambda = r_{\beta}x\lambda$, and hence 
\eqref{eq:step1-d1} is a directed path from 
$x\lambda$ to $r_{\beta}x\lambda$. 
However, we have 
$x\lambda \edge{\beta} r_{\beta}x\lambda$ by the assumption. 
From these, we deduce that 
$p=1$, and hence $\mcr{r_{\alpha}}=r_{i}$ 
for some $i \in I$. Since $\pair{\alpha^{\vee}}{\lambda}=
\pair{\beta^{\vee}}{x\lambda} > 0$
and $r_{\alpha}\lambda=r_{i}\lambda$, 
we see that $\beta=\alpha=\alpha_{i}$ (see Remark~\ref{rem:cover}). 
Therefore, it follows from \eqref{eq:simple2} that 
$x \edge{\alpha_{i}=\beta} r_{i}x=r_{\beta}x$ in $\SB$. 

\paragraph{Step 2.}
%
Assume that the length $k$ of the directed path \eqref{eq:ql} 
is greater than $0$; for simplicity of notation, 
we set $i:=i_{k}$ in \eqref{eq:ql}. 
Since we have $x\lambda \edge{\alpha_{i}} r_{i}x\lambda$ 
in $\LP$ as the initial edge of \eqref{eq:ql}, 
we have $\pair{\alpha_{i}^{\vee}}{x\lambda} > 0$. 
Hence it follows from \eqref{eq:simple2} 
that $x \edge{\alpha_{i}} r_{i}x$ in $\SB$. 

First, let us show the ``only if'' part of \eqref{eq:iff1}. 
Set $y:=r_{\beta}x$, and write $y^{-1}\alpha_{i}$ as 
$y^{-1}\alpha_{i}=\gamma+m\delta$, with $\gamma \in \Delta$ 
and $m \in \BZ$; we see by Lemma~\ref{lem:dia1} that 
$\gamma \notin \Delta_{J}^{+}$. If $\gamma$ is a negative root, 
then $\beta=\alpha_{i}$ by Lemma~\ref{lem:dia1}, and hence we have
$x\lambda \edge{\alpha_{i}=\beta} r_{i}x\lambda=r_{\beta}x\lambda$. 
Assume that $\gamma \in \Delta^{+} \setminus \Delta_{J}^{+}$. 
Then, by Lemma~\ref{lem:dia1}, $\beta \ne \alpha_{i}$, and 
we obtain the following diagram in $\SB$: 
\begin{equation*}
\begin{diagram}
\node{x}\arrow{s,l}{\alpha_{i}} \arrow{e,t}{\beta} 
\node{y}\arrow{s,r,..}{\alpha_{i}} \\
\node{r_{i}x} \arrow{e,t,..}{r_{i}\beta} \node{r_{i}y}
\end{diagram}
\end{equation*}
By our induction hypothesis applied to $r_{i}x$, we have 
$r_{i}x\lambda \edge{r_{i}\beta} r_{i}y\lambda$ in $\LP$. 
Also, since $y \edge{\alpha_{i}} r_{i}y$ in $\SB$, 
it follows from \eqref{eq:simple2} that 
$y\lambda \edge{\alpha_{i}} r_{i}y\lambda$ in $\LP$. 
Because $x\lambda \edge{\alpha_{i}} r_{i}x\lambda$ in $\LP$ 
as the initial edge of \eqref{eq:ql}, 
we deduce that $x\lambda \edge{\beta} y\lambda$ by 
\cite[Lemma~4.1\,c)]{Lit95}, as desired. 

Next, let us show the ``if'' part of \eqref{eq:iff1}. 
Recall that $\pair{\alpha_{i}^{\vee}}{x\lambda} > 0$. 
If $\pair{\alpha_{i}^{\vee}}{y\lambda} \le 0$, then 
it follows from \cite[Corollary~1 in \S4]{Lit95} that 
$\beta=\alpha_{i}$, and hence $x \edge{\alpha_{i}=\beta} 
r_{i}x=r_{\beta}x$. Assume that 
$\pair{\alpha_{i}^{\vee}}{y\lambda} > 0$. 
Then we have $y \edge{\alpha_{i}} r_{i}y$ 
by \eqref{eq:simple2}. 
It follows immediately from \cite[Lemma~4.1\,c)]{Lit95}
that $\beta \ne \alpha_{i}$, and 
$r_{i}x\lambda \edge{r_{i}\beta} r_{i}y\lambda$. 
By our induction hypothesis applied to $r_{i}x$, we have 
$r_{i}x \edge{r_{i}\beta} r_{i}y$. 
Therefore, we deduce from Lemma~\ref{lem:dia2} that 
$x \edge{\beta} y$, as desired. 
This completes the proof of the proposition. 
\end{proof}

The next corollary follows immediately 
from Proposition~\ref{prop:sib-lv0} and 
\cite[Lemma 2.11]{NS08}. 
%
%
\begin{cor} \label{cor:pi_}
Let $J$ be a subset of $I$. 
Let $x \in (W^{J})_{\af}$ and $\beta \in \prr$ be such that 
$x \edge{\beta} r_{\beta}x$. Then, $\beta$ is either of 
the following forms: $\beta=\alpha$ with $\alpha \in \Delta^{+}$, 
or $\beta=\alpha+\delta$ with $-\alpha \in \Delta^{+}$.
Moreover, if $x=wz_{\xi}t_{\xi}$ with $w \in W^{J}$ and 
$\xi \in \jad$ (see \eqref{eq:W^J_af}), then $w^{-1}\alpha \in 
\Delta^{+} \setminus \Delta_{J}^{+}$ in both cases.
\end{cor}
%
%
\subsection{Proof of Theorem \ref{thm:stability}.}
\label{subsec:pr-thm:stability}

We prove part (1) only for $e_i$; 
the proof for $f_i$ is similar. 
Let $\eta = (x_1,\,\ldots,\,x_s ; a_0,\,\ldots,\,a_s) \in \sLS$ 
and $i \in I_{\af}$ be such that $m_i^{\ol{\eta}} \le -1$. 
Define $t_0,\,t_1 \in [0,1]$ by \eqref{eq:t-e} 
(for $\ol{\eta}$ and $i \in I_{\af}$), and 
let $1 \le p \le q \le s$ be such that 
$a_{p-1} \le t_0 < a_q$ and $t_1 = a_q$. 
By the definition of $e_{i}$, we have
\begin{equation*}
0= a_{0} \rsa{x_1} \cdots \rsa{x_{p-1}} a_{p-1} \rsa{x_p} 
\underbrace{t_{0} \rsa{r_{i}x_p} a_{p} \rsa{r_{i}x_{p+1}} \cdots 
\rsa{r_{i}x_{q}} a_{q}=t_{1}}_{\text{``reflected'' by $r_{i}$}}
\rsa{x_{q+1}} \cdots \rsa{x_{s}} a_{s}=1; 
\end{equation*}
if $t_0 = a_{p-1}$, then we drop $x_p$ and $a_{p-1}$, and 
if $r_i x_q = x_{q+1}$, then we drop $x_{q+1}$ and $a_{q}=t_1$.
We need to show that 
\begin{itemize}

\item[(i)]
$r_i x_u \in (W^J)_{\af}$ for all $p \le u \le q$;

\item[(ii)]
if $t_0 \neq a_{p-1}$ (resp., $t_0 = a_{p-1}$ and $p >1$), 
then there exists a directed path 
from $r_i x_p$ to $x_p$ (resp., to $x_{p-1}$) 
in $\SBb{t_{0}}$; 

\item[(iii)]
for each $p \le u \le q-1$, 
there exists a directed path from 
$r_i x_{u+1}$ to $r_i x_u$ in $\SBb{a_u}$;

\item[(iv)]
if $r_i x_q \neq x_{q+1}$, then 
there exists a directed path 
from $x_{q+1}$ to $r_i x_q$ 
in $\SBb{t_{1}}=\SBb{a_{q}}$. 

\end{itemize}

\noindent
{\it Proof of {\rm (i)}.} 
As mentioned in Definition~\ref{dfn:ro}\,(1), 
the function $H_i^{\ol{\eta}} (t)$ is 
strictly decreasing on $[t_0,\,t_1]$. 
Therefore, we see that $\pair{\alpha_i^{\vee}}{x_u \lambda} < 0$ 
for all $p \le u \le q$. This implies that 
$x_u^{-1} \alpha_i \notin (\Delta_J)_{\af}$, and hence 
$r_i x_u \in (W^J)_{\af}$ by Lemma \ref{lem:r_i}. \bqed

\noindent
{\it Proof of {\rm (ii)}.} 
Since $\pair{\alpha_i^{\vee}}{x_p \lambda} < 0$ as above, 
we have $r_i x_p \edge{\alpha_{i}} x_{p}$ in $\SB$ 
by \eqref{eq:simple2}. 
By applying \cite[Lemma 4.5\,c)]{Lit95} 
to $\ol{\eta} \in \LS$ and $t_0 \in [0,1]$, we deduce that 
$t_0 \pair{\alpha_i^{\vee}}{x_p \lambda} \in \BZ$, 
which implies that the edge 
$r_i x_p \edge{\alpha_{i}} x_{p}$ above is an edge in $\LPb{t_0}$. 
Thus we have shown (ii) in the case that $t_0 \neq a_{p-1}$. 
Assume next that $t_0 = a_{p-1}$ and $p >1$. 
By the assumption, there exists a directed path from $x_p$ to $x_{p-1}$ 
in $\SBb{a_{p-1}}=\SBb{t_{0}}$. 
By concatenating this directed path with
$r_i x_p \edge{\alpha_{i}} x_{p}$ obtained above, 
we obtain a directed path from $r_i x_p$ to $x_{p-1}$ in $\SBb{t_{0}}$.
Thus we have shown (ii). \bqed

\noindent
{\it Proof of {\rm (iii)}.}
Fix $p \le u \le q-1$. 
Let $x_{u+1} = y_0 \edge{\beta_1} y_1 \edge{\beta_2} \cdots 
\edge{\beta_{k}} y_k = x_u$ be a directed path 
from $x_{u+1}$ to $x_u$ in $\SBb{a_u}$. 
Since $H^{\ol{\eta}}_{i}(t)$ is strictly decreasing on 
$[t_{0},\,t_{1}]$, we see that 
$\pair{\alpha_{i}^{\vee}}{x_{u+1}\lambda} < 0$. 
Also, since 
$H^{\ol{\eta}}_{i}(t_{0})=
m^{\ol{\eta}}_{i}+1$ and 
$H^{\ol{\eta}}_{i}(t_{1})=
m^{\ol{\eta}}_{i}$, we see that 
$H^{\ol{\eta}}_{i}(a_{u}) \notin \BZ$. 
Therefore, it follows from \cite[Remark~4.6]{Lit95} that 
$\pair{\alpha_{i}^{\vee}}{y_{l}\lambda} > 0$ 
for all $0 \le l \le k$, and hence that 
there exists a directed path from $r_{i}x_{u+1}$ to 
$r_{i}x_{u}$ in $\SBb{a_{u}}$ 
by Lemma~\ref{lem:directed}\,(1). \bqed

\noindent
{\it Proof of {\rm (iv)}.}
By the definition, there exists a directed path from 
$x_{q+1}$ to $x_{q}$ in $\SBb{a_{q}}$. 
By the definition of $t_{1}$, we see that 
$\pair{\alpha_{i}^{\vee}}{x_{q+1}\lambda} \ge 0$ 
and $\pair{\alpha_{i}^{\vee}}{x_{q}\lambda} < 0$. 
Therefore, it follows from Lemma~\ref{lem:directed}\,(3) 
that there exists a directed path from $x_{q+1}$ to 
$r_{i}x_{q}$ in $\SBb{a_{q}}$. \bqed

Thus we have proved part (1) of 
Theorem \ref{thm:stability}. 
Next, let us prove part (2). 
We see from the definition of root operators 
that for all $\eta \in \sLS$ and $i \in I_{\af}$, 
\begin{equation} \label{eq:epsilon-phi}
\begin{cases}
\wt(\eta) = \wt (\ol{\eta}), & \\[1.5mm]
\ol{e_{i}\eta}=e_{i}\ol{\eta},\,
\ol{f_{i}\eta}=f_{i}\ol{\eta}, & \\[1.5mm]
\ve_i (\eta) = \ve_i (\ol{\eta}),\,
\vp_i (\eta) = \vp_i (\ol{\eta}),
\end{cases}
\end{equation}
where we understand that $\ol{\bzero}=\bzero$. 
From \eqref{eq:epsilon-phi}, we can easily deduce that 
the set $\sLS$ satisfies the axioms for crystals, 
except for the axiom that 
\begin{align} \label{eq:crystal}
\text{$e_i \eta_1 = \eta_2$ if and only if $\eta_1 = f_i \eta_2$
for $\eta_1,\,\eta_2 \in \sLS$ and $i \in I_{\af}$}.
\end{align}
Now we define $t_0^{1},\,t_1^{1} \in [0,1]$ 
(resp., $t_0^{2},\,t_1^{2} \in [0,1]$) 
by \eqref{eq:t-e} (resp., \eqref{eq:t-f})
for $\ol{\eta_1}$ (resp., $\ol{\eta_{2}}$) 
and $i \in I_{\af}$; 
note that $H^{ \ol{\eta_{1}} }_{i}(t)$ 
(resp., $H^{ \ol{\eta_{2}} }_{i}(t)$) is 
strictly decreasing (resp., increasing) on 
$[t_{0}^{1},\,t_{1}^{1}]$ 
(resp., $[t_{0}^{2},\,t_{1}^{2}]$). 
Then we deduce from the definitions that 
$t_{0}^{1}=t_{0}^{2}$ and $t_{1}^{1}=t_{1}^{2}$. 
Therefore, \eqref{eq:crystal} follows immediately
from the definition of the root operators 
$e_{i}$ and $f_{i}$. This completes the proof of 
Theorem~\ref{thm:stability}.
%
%
\begin{rem} \label{rem:SLS-LS}
\mbox{}
\begin{enu}
\item 
By \eqref{eq:epsilon-phi}, the map 
$\ol{\phantom{\eta}}:\sLS \rightarrow \LS$, 
$\eta \mapsto \ol{\eta}$, is a strict morphism of crystals 
in the sense of \cite[\S 1.5]{Kas94}; also, this map is surjective.

\item
By \eqref{eq:epsilon-phi} and \cite[Lemma 2.1\,c)]{Lit95}, 
we have $\ve_i (\eta) = -m_i^{\ol{\eta}}$ and 
$\vp_i (\eta ) = H_i^{\ol{\eta}}(1) - m_i^{\ol{\eta}}$ 
for all $\eta \in \sLS$ and $i \in I_{\af}$.
\end{enu}
\end{rem}
%
%
\section{Proof of Proposition \ref{prop:conn-isom}.}
\label{sec:pr-prop:conn-isom}

Throughout this section, we fix $\lambda \in P^+$, 
and set $J:=J_{\lambda} = 
\bigl\{i \in I \mid \pair{\alpha_i^{\vee}}{\lambda}=0 \bigr\}$.
%
%
\subsection{Extremal elements in $\Bo$ and $\sLSo$.} 
\label{subsec:extremal}

We know from \cite[\S7]{Kas94} that 
the affine Weyl group $W_{\af}$ acts on $\B$ by
\begin{align}\label{eq:W-action}
r_{i} b := 
\begin{cases}
f_i^{n} b 
 & \text{if $n=\pair{\alpha_i^{\vee}}{\wt(b)} \ge 0$}, \\[1.5mm]
e_i^{-n} b 
 & \text{if $n=\pair{\alpha_i^{\vee}}{\wt(b)} \le 0$},
\end{cases}
\end{align}
for each $b \in \B$ and $i \in I_{\af}$.

\begin{prop}[{cf. \cite[Conjecture 5.11]{Kas02b}}; 
see also Remark \ref{rem:stabilizer}] 
\label{prop:stab}
The following equality holds: 
\begin{equation*}
(W_J)_{\af} = \bigl\{ x \in W_{\af} 
  \mid xu_{\lambda } = u_{\lambda} \bigr\}.
\end{equation*}
\end{prop}
\begin{proof} 
Write $\lambda \in P^+$ as 
$\lambda = \sum_{i \in I} m_i \varpi_i$ 
with $m_i \in \BZ_{\ge 0}$, $i \in I$. 
Then we know from \cite[Remark 4.17]{BN04} 
(see also \cite[\S 13]{Kas02b}) that 
there exists an embedding 
$\Psi : \Bo \hookrightarrow 
\bigotimes_{i \in I} \CB(\varpi_i)^{\otimes m_i}$ 
of crystals such that $\Psi (u_{\lambda }) = 
\bigotimes_{i \in I} u_{\varpi_i}^{\otimes m_i}$. 
Recall that for each $i \in I$, $u_{\vpi_{i}} \in \CB(\vpi_{i})$ 
is an extremal element of weight $\vpi_{i}$ 
in the sense of \cite[Definition~8.1.1]{Kas94}; 
in particular, we have
$\wt (xu_{\vpi_{i}}) = x\vpi_{i}$,
$\ve_{j}(xu_{\vpi_{i}}) = 
 \max\bigl\{0,\,-\pair{\alpha_{j}^{\vee}}{x\vpi_{i}}\bigr\}$, and 
$\vp_{j}(xu_{\vpi_{i}}) = 
 \max\bigl\{0,\,\pair{\alpha_{j}^{\vee}}{x\vpi_{i}}\bigr\}$ 
for every $x \in W_{\af}$ and $j \in I_{\af}$. 
From these, using the tensor product rule for crystals, 
we can show by induction on $\ell(x)$ 
(see also \cite[Lemma~1.6]{AK}) that 
$\Psi (xu_{\lambda }) = 
\bigotimes_{i \in I} (xu_{\varpi_i})^{\otimes m_i}$ 
for all $x \in W_{\af}$. Therefore, we deduce that 
\begin{align}\label{eq:cap}
\bigl\{ x \in W_{\af} \mid xu_{\lambda } = u_{\lambda} \bigr\} = 
\bigcap_{i \in I \setminus J} 
  \bigl\{ x \in W_{\af} \mid x u_{\varpi_i} = u_{\varpi_i} \bigr\} . 
\end{align}
Also, we know from \cite[Lemma 5.6]{Kas02b} that 
\begin{align}\label{eq:Kas}
\bigl\{ x \in W_{\af} \mid x u_{\varpi_i} = u_{\varpi_i}\bigr \} = 
\langle r_{\beta} \mid \beta \in \prr,\ \pair{\beta^{\vee}}{\varpi_i} =0 \rangle.
\end{align}
If $\beta \in (\Delta_J)_{\af}^+$, 
then $\pair{\beta^{\vee}}{\varpi_i} =0$, 
and hence $r_{\beta} \in 
\bigl\{ x \in W_{\af} \mid xu_{\varpi_i} = u_{\varpi_i} \bigr\}$ 
for all $i \in I \setminus J$. Because $(W_J)_{\af} = 
\langle r_{\beta} \mid \beta \in (\Delta_J)_{\af}^+ \rangle$ 
by Remark~\ref{rem:stabilizer}, it follows immediately 
that $(W_J)_{\af} \subset 
\bigl\{ x \in W_{\af} \mid xu_{\lambda } = u_{\lambda} \bigr\}$.

Let us show the opposite inclusion. 
Let $x \in W_{\af}$ be such that 
$x u_{\lambda } = u_{\lambda}$. 
By \eqref{eq:cap}, we have $x u_{\varpi_i} = u_{\varpi_i}$ 
for all $i \in I \setminus J$; 
in particular, $x \varpi_i = \varpi_i$ 
for all $i \in I \setminus J$ since the weight of 
$x u_{\varpi_i}$ is equal to $x \varpi_i$. 
We write $x$ as $x = w t_{\xi}$ with $w \in W$ and $\xi \in Q^{\vee}$. 
Then, for all $i \in I \setminus J$, we have 
$\varpi_i = x \varpi_i = w \varpi_i - \pair{\xi}{\varpi_i}\delta $ 
by \eqref{eq:lv0action}, and hence $w \varpi_i = \varpi_i$ and 
$\pair{\xi}{\varpi_i} = 0$. 
Therefore, we deduce that $w \in W_J \subset (W_J)_{\af}$, 
and $\xi \in Q_J^{\vee}$, which implies that 
$t_{\xi} \in (W_J)_{\af}$. 
Thus we obtain $x \in (W_J)_{\af}$, 
and hence $(W_J)_{\af} \supset 
\bigl\{ x \in W_{\af} \mid x u_{\lambda } = u_{\lambda} \bigr\}$. 
This completes the proof of the proposition. 
\end{proof}

Recall that $u_{\lambda} \in \CB_{0}(\lambda)$ is an extremal element 
of weight $\lambda$ in the sense of \cite[Definition~8.1.1]{Kas94}.
From Proposition~\ref{prop:stab}, we see that the set $\bigl\{yu_{\lambda} \mid 
y \in W_{\af} \bigr\}$ is in bijective correspondence with 
the quotient set $W_{\af}/(W_{J})_{\af}$, and hence with the set $(W^{J})_{\af}$ 
by Proposition~\ref{prop:P}. 
Hence we set $u_{x}:=xu_{\lambda}$ for $x \in (W^{J})_{\af}$; remark that 
\begin{equation} \label{eq:ext-u}
\wt(u_x) = x \lambda, \qquad
\ve_i (u_x) = 
  \max \bigl\{0,\,-\pair{\alpha_i^{\vee}}{x\lambda}\bigr\}, \qquad
\vp_i (u_x) = 
  \max \bigl\{0,\,\pair{\alpha_i^{\vee}}{x\lambda}\bigr\}
\end{equation}
for all $x \in (W^J)_{\af}$ and $i \in I_{\af}$. 
It follows from Proposition~\ref{prop:stab} that 
$yu_{\lambda}=u_{\PJ(y)}$ for all $y \in W_{\af}$. 

Now, we set $\eta_x := (x ; 0,1) \in \sLS$ for $x \in (W^J)_{\af}$.
By Remark \ref{rem:SLS-LS}\,(2) and \eqref{eq:ext-u}, 
we see that for each $x \in (W^J)_{\af}$ and $i \in I_{\af}$,
\begin{equation} \label{eq:ext-eta}
\begin{split}
& \wt(\eta_x ) = x \lambda =\wt(u_x ), \qquad
  \ve_i (\eta_x) = 
   \max \bigl\{0,\,-\pair{\alpha_i^{\vee}}{x\lambda}\bigr\}=
  \ve_i (u_x), \\
& 
\vp_i (\eta_x)=
  \max \bigl\{0,\,\pair{\alpha_i^{\vee}}{x\lambda}\bigr\}=
\vp_i (u_x).
\end{split}
\end{equation}
Next, for $x \in (W^J)_{\af}$ and $i \in I_{\af}$, we set
\begin{equation*}
r_{i}\eta_{x}:=
 \begin{cases}
  f_i^{n} \eta_x & 
  \text{if $n=\pair{\alpha_i^{\vee}}{x\lambda} \ge 0$}, \\[1.5mm] 
  e_i^{-n} \eta_x &
  \text{if $n=\pair{\alpha_i^{\vee}}{x\lambda} \le 0$};
 \end{cases}
\end{equation*}
observe that $r_{i}\eta_{x} \ne \bzero$ 
for any $x \in W_{\af}$ and $i \in I_{\af}$. 
Then we have
%
%
\begin{equation} \label{eq:etax}
r_{i_{p}}r_{i_{p-1}} \cdots r_{i_{2}}r_{i_{1}}\eta_{x} =
\eta_{\Pi^{J}(r_{i_{p}}r_{i_{p-1}} \cdots r_{i_{2}}r_{i_{1}}x)}
\end{equation}
for all $x \in W_{\af}$ and $i_{1},\,\dots,\,i_{p} \in I_{\af}$. 
Let us show \eqref{eq:etax} by induction on $p$. 
If $p=0$, then \eqref{eq:etax} is obvious. Assume that $p > 0$. 
By our induction hypothesis, we have 
$r_{i_{p-1}} \cdots r_{i_{2}}r_{i_{1}}\eta_{x} =
\eta_{\Pi^{J}(r_{i_{p-1}} \cdots r_{i_{2}}r_{i_{1}}x)}$; 
for simplicity of notation, we set 
$y:=\Pi^{J}(r_{i_{p-1}} \cdots r_{i_{2}}r_{i_{1}}x) \in (W^{J})_{\af}$. 
If $n=\pair{\alpha_{i_p}^{\vee}}{y \lambda} > 0$, then 
$y^{-1} \alpha_{i_p} \notin (\Delta_J)_{\af}$, and hence 
$r_{i_p} y \in (W^J)_{\af}$ by Lemma \ref{lem:r_i}. 
Therefore, we obtain
\begin{equation*}
r_{i_p} y = \Pi^{J}(r_{i_p} y) = 
\Pi^{J}(r_{i_p}\Pi^{J}(r_{i_{p-1}} \cdots r_{i_{2}}r_{i_{1}}x)) = 
\Pi^{J}(r_{i_p}r_{i_{p-1}} \cdots r_{i_{2}}r_{i_{1}}x).
\end{equation*}
Also, it is easily shown by induction on $k$ that 
$f_{i_p}^k \eta_y = (r_{i_p} y,\,y \,;\, 0,\,k/n,\,1)$
for $0 \le k \le n$; in particular, we obtain $f_{i_p}^n \eta_y = 
\eta_{r_{i_p} y}$. From these, we deduce that 
\begin{equation*}
r_{i_{p}}r_{i_{p-1}} \cdots r_{i_{2}}r_{i_{1}}\eta_{x} =
r_{i_{p}}\eta_{y}=f_{i_p}^n \eta_y =\eta_{r_{i_p} y} = 
\eta_{\Pi^{J}(r_{i_{p}}r_{i_{p-1}} \cdots r_{i_{2}}r_{i_{1}}x)},
\end{equation*}
as desired. The proof for the case that 
$n=\pair{\alpha_{i_p}^{\vee}}{y \lambda} < 0$ is similar. 
If $n=\pair{\alpha_{i_p}^{\vee}}{y \lambda}=0$, 
then we have $r_{i_p}\eta_{y}=\eta_{y}$ by the definition. 
Also, we see that $y^{-1} \alpha_{i_p} \in (\Delta_J)_{\af}$, 
and hence $r_{y^{-1}\alpha_{i_p}} \in (W_{J})_{\af}$ 
(see Remark~\ref{rem:stabilizer}).
Hence it follows that 
$\Pi^{J}(r_{i_p}y)=\Pi^{J}(yr_{y^{-1}\alpha_{i_p}})=
\Pi^{J}(y)=y$. Therefore, we deduce that 
\begin{equation*}
r_{i_{p}}r_{i_{p-1}} \cdots r_{i_{2}}r_{i_{1}}\eta_{x} =
r_{i_{p}}\eta_{y}=\eta_{y}=\eta_{\Pi^{J}(r_{i_p}y)}
=
\eta_{\Pi^{J}(r_{i_{p}}r_{i_{p-1}} \cdots r_{i_{2}}r_{i_{1}}x)},
\end{equation*}
as desired. 

\begin{rem} \label{rem:etax}
It follows from \eqref{eq:etax} that 
$\eta_{x} \in \sLSo$ for all $x \in (W^{J})_{\af}$.
\end{rem} 
%
%
\subsection{$N$-multiple maps.}

\begin{prop}\label{prop:Nmulti}
Let $N \in \BZ_{>0}$. There exists a unique injective map 
$\sigma_N : \Bo \hookrightarrow \Bo^{\otimes N}$ 
such that $\sigma_N (u_{\lambda }) = u_{\lambda }^{\otimes N}$, and
\begin{align}
& \wt(\sigma_N (b)) = N \wt(b), &
& \ve_i (\sigma_N (b)) = N \ve_i (b), &
& \vp_i (\sigma_N (b)) = N \vp_i (b), \label{eq:sigmaB1} \\
& \sigma_N (e_i b) = e_i^N \sigma_N (b), &
& \sigma_N (f_i b) = f_i^N \sigma_N (b), \label{eq:sigmaB2}
\end{align}
for $b \in \Bo$ and $i \in I_{\af}$, 
where we understand that $\sigma_N (\bzero) = \bzero$.
\end{prop}

\begin{proof}
We know from \cite[Theorem~3.7]{NS03} that 
there exists an injective map $\iota_{N} : \Bo \hookrightarrow 
\CB_{0}(N\lambda)$ such that $\iota_{N} (u_{\lambda}) = 
u_{N\lambda }$, and
\begin{align*}
& \wt(\iota_N (b)) = N \wt(b), &
& \ve_i (\iota_N (b)) = N \ve_i (b), &
& \vp_i (\iota_N (b)) = N \vp_i (b) , \\
& \iota_N (e_i b) = e_i^N \iota_{N} (b) , &
& \iota_N (f_i b) = f_i^N \iota_{N} (b),
\end{align*}
for $b \in \Bo$ and $i \in I_{\af}$, 
where we understand that $\iota_N (\bzero) = \bzero$.
Write $\lambda $ as 
$\lambda = \sum_{i \in I} m_i \varpi_i$, 
with $m_{i} \in \BZ_{\ge 0}$, $i \in I_{0}$.
We deduce from \cite[Remark~4.17]{BN04} 
(see also \cite[\S13]{Kas02b}) that 
there exists an embedding $\CB_{0}(N\lambda) 
\hookrightarrow 
\ti{\CB}:=
\bigotimes_{i \in I} \CB(\varpi_i)^{\otimes Nm_i}$ of 
crystals that maps $u_{N\lambda }$ to 
$\ti{u}:=\bigotimes_{i \in I} u_{\varpi_i}^{\otimes Nm_i}$, 
Also, we know from \cite[\S 10]{Kas02b} that 
for every $i,\,j \in I$, there exists an isomorphism 
$\CB(\vpi_{i}) \otimes \CB(\vpi_{j}) 
\stackrel{\sim}{\rightarrow}
\CB(\vpi_{j}) \otimes \CB(\vpi_{i})$ of crystals that maps 
$u_{\vpi_{i}} \otimes u_{\vpi_{j}}$ to 
$u_{\vpi_{j}} \otimes u_{\vpi_{i}}$. 
Therefore, we obtain an isomorphism 
$\ti{\CB} \stackrel{\sim}{\rightarrow} \ha{\CB}:=
 \bigl(\bigotimes_{i \in I} 
 \CB(\varpi_i)^{\otimes m_i}\bigr)^{\otimes N}$
of crystals that maps 
$\ti{u}$ to 
$\ha{u}:=
 \bigl(\bigotimes_{i \in I} 
 u_{\varpi_i}^{\otimes m_i}\bigr)^{\otimes N}$. 
Combining the above, we see that 
$\CB_{0}(N\lambda)$ is isomorphic, as a crystal, to 
the connected component of $\ha{\CB}$ containing $\ha{u}$. 

From \cite[Remark~4.17]{BN04}, we deduce that 
there exists an embedding $\CB_{0}(\lambda)^{\otimes N} 
\hookrightarrow \ha{\CB}$ of crystals that maps 
$u_{\lambda}^{\otimes N}$ to $\ha{u}$; 
note that the connected component 
of $\ha{\CB}$ containing $\ha{u}$ 
(which is isomorphic as a crystal to $\CB_{0}(N\lambda)$) 
is contained in the image of the embedding 
$\CB_{0}(\lambda)^{\otimes N} \hookrightarrow \ha{\CB}$ 
of crystals. Hence we obtain an embedding 
$\iota':\CB_{0}(N\lambda) \hookrightarrow 
\CB_{0}(\lambda)^{\otimes N}$ that maps 
$u_{N\lambda}$ to $u_{\lambda}^{\otimes N}$. 
Now it is clear that the composite
$\iota' \circ \iota_{N}:\CB_{0}(\lambda) \hookrightarrow 
\CB_{0}(\lambda)^{\otimes N}$ satisfies conditions 
\eqref{eq:sigmaB1} and \eqref{eq:sigmaB2} 
required of $\sigma_{N}$. The uniqueness follows from 
the connectedness of $\CB_{0}(\lambda)$. 
This proves the proposition. 
\end{proof}

Because $\sigma_{MN}=\sigma_{M}^{\otimes N} \circ \sigma_{N}$ 
for all $M,\,N \in \BZ_{> 0}$, we can prove 
the following proposition in exactly the same way as 
\cite[Proposition~8.3.2\,(3)]{Kas02a} and 
\cite[Proposition~3.12]{NS03}.

\begin{prop}\label{prop:similarB}
Let $b \in \Bo$. There exists $N_b \in \BZ_{>0}$ 
such that for every multiple $N \in \BZ_{>0}$ of $N_b$, 
the element $\sigma_N (b) \in \CB(\lambda)^{\otimes N}$ 
is of the form $\sigma_N (b)=u_{x_1} \otimes 
u_{x_2} \otimes \cdots \otimes u_{x_N}$ 
for some $x_1 , x_2 , \ldots , x_N \in (W^J)_{\af}$.
\end{prop}

Since $\pair{c}{\lambda}=0$, 
we see that
$\bigl\{ \pair{\beta^{\vee}}{\lambda} \mid 
  \beta \in \Delta_{\af} \bigr\} = 
\bigl\{ \pair{\alpha^{\vee}}{\lambda} \mid 
  \alpha \in \Delta \bigr\}$
is a finite set. Define $N_{\lambda } \in \BZ_{>0}$ to be 
the least common multiple of the integers in the finite set 
$\bigl\{ \pair{\beta^{\vee}}{\lambda} \mid 
  \beta \in \Delta_{\af}^{+} \bigr\} \setminus \bigl\{ 0 \bigr\}$. 
%
%
\begin{lem}\label{lem:Na}
Let $N \in \BZ_{>0}$ be a multiple of $N_{\lambda}$. 
If $\eta = (x_1,\,\ldots,\,x_s ; a_0,\,\ldots,\,a_s) \in \sLS$, 
then $Na_u \in \BZ$ for all $0 \le u \le s$. 
\end{lem}

\begin{proof}
If $u = 0$ or $s$, then the assertion is obvious. 
Assume that $1 \le u \le s-1$. By the definition of a SiLS path, 
there exists a directed path from $x_{u+1}$ to $x_u$ in $\SBb{a_u}$;
in particular, there exist $x,\,y \in (W^{J})_{\af}$ 
and $\beta \in \prr$ such that $x \edge{\beta} y$ in $\SBb{a_{u}}$. 
By the definition of $N_{\lambda}$, it suffices to show that 
$a_u \pair{x^{-1} \beta^{\vee}}{\lambda} \in 
\BZ \setminus \{ 0 \}$. But, 
since $x \edge{\beta} y$ is contained in $\SBb{a_u}$, 
we have $a_u \pair{x^{-1} \beta^{\vee}}{\lambda} = 
a_u \pair{\beta^{\vee}}{x \lambda} \in \BZ$ by the definition. 
Suppose now that $\pair{x^{-1} \beta^{\vee}}{\lambda}= 0$. 
Then, $x^{-1} \beta \in (\Delta_J)_{\af}$, and hence 
$r_{x^{-1}\beta} \in (W_J)_{\af}$ by Remark~\ref{rem:stabilizer}.
Therefore, $y = r_{\beta} x = x r_{x^{-1} \beta} \notin (W^J)_{\af}$, 
which is a contradiction. This proves the lemma. 
\end{proof}

Let $N \in \BZ_{>0}$ be a multiple of $N_{\lambda}$. 
We define $\sigma_{N}:\sLS \hookrightarrow \sLS^{\otimes N}$ as follows. 
Let $\eta = (x_1,\,\ldots,\,x_s ; a_0,\,\ldots,\,a_s) \in \sLS$. 
Note that $k_{u}:=Na_{u} \in \BZ$ for all $0 \le u \le s$
by Lemma~\ref{lem:Na}. Now we set
\begin{equation*}
\sigma_{N}(\eta):=
\underbrace{
  \eta_{x_1} \otimes \cdots \otimes \eta_{x_1}
}_{\text{$(k_1 - k_0)$ times}} \otimes 
\underbrace{
  \eta_{x_2} \otimes \cdots \otimes \eta_{x_2}
}_{\text{$(k_2 - k_0)$ times}} \otimes \cdots \otimes 
\underbrace{
  \eta_{x_s} \otimes \cdots \otimes \eta_{x_s}
}_{\text{$(k_s - k_{s-1})$ times}} \in \sLS^{\otimes N}.
\end{equation*}

\begin{prop}\label{prop:Nmulti2}
Keep the notation and setting above. 
The map $\sigma_{N}:\sLS \hookrightarrow \sLS^{\otimes N}$ above 
is an injective map such that 
$\sigma_N (\eta_{e}) = \eta_{e}^{\otimes N}$, and
\begin{align}
& \wt(\sigma_N (\eta)) = N \wt(\eta), &
& \ve_i (\sigma_N (\eta)) = N \ve_i (\eta), &
& \vp_i (\sigma_N (\eta)) = N \vp_i (\eta), \label{eq:SN1} \\
& \sigma_N (e_i \eta) = e_i^N \sigma_N (\eta),&
& \sigma_N (f_i \eta) = f_i^N \sigma_N (\eta), \label{eq:SN2}
\end{align}
for all $\eta \in \sLS$ and $i \in I_{\af}$, 
where we understand that $\sigma_N (\bzero) = \bzero$.
\end{prop}

\begin{proof}
First, we define a map 
$\sLS^{\otimes N} \rightarrow \LS^{\otimes N}$ by 
$\eta_{1} \otimes \cdots \otimes \eta_{N} \mapsto 
 \ol{\eta_{1}} \otimes \cdots \otimes \ol{\eta_{N}}$; 
by abuse of notation, we also denote this map by 
$\ol{\phantom{\eta}}:\sLS^{\otimes N} \rightarrow \LS^{\otimes N}$. 
We see from Remark~\ref{rem:SLS-LS}\,(1) and 
the tensor product rule for crystals that 
this map $\ol{\phantom{\eta}}:
\sLS^{\otimes N} \rightarrow \LS^{\otimes N}$ 
is a strict morphism of crystals. 

For $\pi_{1},\,\dots,\,\pi_{N} \in \BB(\lambda)$, 
define the concatenation 
$\pi_{1} \ast \cdots \ast \pi_{N}$ by 
\begin{align*} \label{eq:cat}
& (\pi_{1} \ast \cdots \ast \pi_{N})(t):=
  \sum_{L=1}^{K-1} 
  \pi_{L}(1)+ \pi_{K}(Nt-K+1) \nonumber \\[-1.5mm]
& \hspace{60mm} \text{for \, } \frac{K-1}{N} \le t \le 
\frac{K}{N} \text{ and } 1 \le K \le N,
\end{align*}
and set $\BB(\lambda)^{\ast N}:=
\bigl\{\pi_{1} \ast \cdots \ast \pi_{N} \mid 
\pi_{K} \in \BB(\lambda),\,1 \le K \le L\bigr\}$. 
For $\ti{\pi} \in \BB(\lambda)$, 
we set $\wt(\ti{\pi}):=\ti{\pi}(1)$. 
Also, for $\ti{\pi} \in \BB(\lambda)$ and $i \in I_{\af}$, 
we define $e_{i}\ti{\pi}$ and $f_{i}\ti{\pi}$ 
in exactly the same way as in Definition~\ref{dfn:ro}; 
we deduce from \cite[Lemma~2.7]{Lit95} that 
$e_{i}\ti{\pi},\,f_{i}\ti{\pi} \in 
\BB(\lambda)^{\ast N} \cup \{\bzero\}$. 
Now, for $\ti{\pi} \in \BB(\lambda)$ and $i \in I_{\af}$, 
define $\ve_i (\ti{\pi})$ and $\vp_i (\ti{\pi})$ 
as in \eqref{eq:vevp}. Then we deduce from 
\cite[Lemma~2.7]{Lit95} that 
the set $\BB(\lambda)^{\ast N}$, 
equipped with the maps $\wt$, $e_{i}$, $f_{i}$, $i \in I$, 
and $\ve_{i}$, $\vp_{i}$, $i \in I$, is a crystal with 
weights in $P_{\af}$, and that the map 
$\kappa:\BB(\lambda)^{\otimes N} \rightarrow \BB(\lambda)^{\ast N}$, 
$\pi_{1} \otimes \cdots \otimes \pi_{N} \mapsto 
 \pi_{1} \ast \cdots \ast \pi_{N}$, is 
an isomorphism of crystals. 

For $\pi \in \BB(\lambda)$, we define $\iota_{N}(\pi)$ by 
$(\iota_{N}(\pi))(t):=N\pi(t)$ for $t \in [0,1]$; 
it is easily seen that $N\pi \in \BB(N\lambda)$. 
Thus we obtain an injective map $\iota_{N}:\BB(\lambda) 
\hookrightarrow \BB(N\lambda)$. We know from 
\cite[Lemma~2.4]{Lit95} that 
\begin{align*}
& \wt(\iota_N (\pi)) = N \wt(\pi), &
& \ve_i (\iota_N (\pi)) = N \ve_i (\pi), &
& \vp_i (\iota_N (\pi)) = N \vp_i (\pi), \\
& \iota_N (e_i \pi) = e_i^N \iota_N (\pi), &
& \iota_N (f_i \pi) = f_i^N \iota_N (\pi), 
\end{align*}
for $\pi \in \sLS$ and $i \in I_{\af}$, 
where we understand that $\iota_{N}(\bzero)=\bzero$.

Observe that for every $\eta \in \sLS$, the element 
$\iota_{N}(\ol{\eta}) \in \BB(N\lambda)$ is identical to 
the element $\kappa(\ol{\sigma_{N}(\eta)}) \in \BB(\lambda)^{\ast N}$ 
(as a piecewise-linear, continuous map from $[0,1]$ to 
$\BR \otimes_{\BZ} P$). The equalities in \eqref{eq:SN1} 
follow immediately from this fact and Remark~\ref{rem:SLS-LS}\,(2). 
Let us show the equalities in \eqref{eq:SN2}. 
We give a proof only for $e_{i}$, $i \in I_{\af}$;
the proof for $f_{i}$, $i \in I_{\af}$, is similar. 
Let $\eta \in \sLS$ and $i \in I_{\af}$ be such that 
$e_{i}\eta \ne \bzero$; note that $e_{i}\ol{\eta} \ne \bzero$, 
and $e_{i}^{N}\sigma_{N}(\eta) \ne \bzero$. 
Write $\sigma_{N}(\eta) \in \sLS^{\otimes N}$ as 
$\sigma_{N}(\eta) 
  = \eta_{y_{1}} \otimes \cdots \otimes \eta_{y_{N}}$
with $y_{1},\,\dots,\,y_{N} \in (W^{J})_{\af}$, and assume that 
%
%
\begin{equation} \label{eq:sn1}
e_{i}^{N}\sigma_{N}(\eta) =
e_{i}^{N}(\eta_{y_{1}} \otimes \cdots \otimes \eta_{y_{N}})=
e_{i}^{p_{1}}\eta_{y_{1}} \otimes \cdots \otimes e_{i}^{p_{N}}\eta_{y_{N}}
\end{equation}
for some $p_{1},\,\dots,\,p_{N} \in \BZ_{\ge 0}$ 
such that $p_{1}+\cdots+p_{N}=N$. Then we see that 
%
%
\begin{equation} \label{eq:sn2}
e_{i}^{N} \kappa(\ol{\sigma_{N}(\eta)}) =
\kappa(\ol{e_{i}^{N} \sigma_{N}(\eta)}) =
e_{i}^{p_{1}}\ol{\eta_{y_{1}}} \ast \cdots \ast
e_{i}^{p_{N}}\ol{\eta_{y_{N}}}.
\end{equation}

Next, let us define $t_{0},\,t_{1} \in [0,1]$ by \eqref{eq:t-e}, 
with $\pi=\ol{\eta}$. It follows 
from Lemma~\ref{lem:Na} and the definition of 
the root operator $e_{i}$ that $t_{0}=L/N$ and $t_{1}=K/N$ 
for some $0 \le L < K \le N$; 
note that $\pair{\alpha_{i}^{\vee}}{y_{M}\lambda} < 0$ 
for all $L+1 \le M \le K$ since the function $H^{\ol{\eta}}_{i}(t)$ 
is strictly decreasing on $[t_{0},\,t_{1}]$. 
It is easily seen from the definition of the root operator $e_{i}$ 
that $\sigma_{N}(e_{i}\eta) \in \sLS^{\otimes N}$ is of the form: 
%
%
\begin{equation} \label{eq:sn3}
\sigma_{N}(e_{i}\eta)=
 \eta_{y_{1}} \ast \cdots \ast \eta_{y_{L}}
 \ast
 \underbrace{\eta_{r_{i}y_{L+1}} \ast \cdots \ast
 \eta_{r_{i}y_{K}}}_{\text{``reflected'' by $r_{i}$}}
 \ast \eta_{y_{K+1}} \ast \cdots \ast \eta_{y_N},
\end{equation}
and hence 
$\kappa(\ol{\sigma_{N}(e_{i}\eta)}) \in \LS^{\ast N}$ 
is of the form:
\begin{equation} \label{eq:sn4}
\kappa(\ol{\sigma_{N}(e_{i}\eta)})=
 \ol{\eta_{y_{1}}} \ast \cdots \ast \ol{\eta_{y_{L}}}
 \ast
 \underbrace{\ol{\eta_{r_{i}y_{L+1}}} \ast \cdots \ast
 \ol{\eta_{r_{i}y_{K}}}}_{\text{``reflected'' by $r_{i}$}}
 \ast \ol{\eta_{y_{K+1}}} \ast \cdots \ast \ol{\eta_{y_N}}. 
\end{equation}
Here, we have 
\begin{equation*}
\kappa(\ol{\sigma_{N}(e_{i}\eta)}) = 
\iota_{N}(\ol{e_{i}\eta}) = 
\iota_{N} (e_{i}\ol{\eta}) =
e_{i}^{N} \iota_{N}(\ol{\eta})=
e_{i}^{N} \kappa(\ol{\sigma_{N}(\eta)}). 
\end{equation*}
Combining this with \eqref{eq:sn2} and \eqref{eq:sn4}, 
we obtain
\begin{align*}
 &  \ol{\eta_{y_{1}}} \ast \cdots \ast \ol{\eta_{y_{L-1}}}
 \ast
 \underbrace{\ol{\eta_{r_{i}y_{L}}} \ast \cdots \ast
 \ol{\eta_{r_{i}y_{K}}}}_{\text{``reflected'' by $r_{i}$}}
 \ast \ol{\eta_{y_{K+1}}} \ast \cdots \ast \ol{\eta_{y_{N}}} = 
e_{i}^{p_{1}}\ol{\eta_{y_{1}}} \ast \cdots \ast
e_{i}^{p_{N}}\ol{\eta_{y_{N}}}.
\end{align*}
Therefore, if we set 
$n_{M}:=\pair{\alpha_{i}^{\vee}}{y_{M}\lambda} < 0$
for $L+1 \le M \le K$, then
\begin{equation*}
p_{M}=
 \begin{cases}
 -n_{M} & \text{if $L+1 \le M \le K$}, \\[1.5mm]
 0 & \text{otherwise},
 \end{cases}
\end{equation*}
for $1 \le M \le N$. From this and \eqref{eq:sn1}, 
we deduce that 
\begin{align*}
e_{i}^{N}\sigma_{N}(\eta) 
& =
\eta_{y_{1}} \otimes \cdots \otimes \eta_{y_{L}} \otimes 
e_{i}^{-n_{L+1}}\eta_{y_{L+1}} \otimes \cdots \otimes
e_{i}^{-n_{K}}\eta_{y_{K}} \otimes \eta_{y_{K+1}} \otimes
\cdots \otimes \eta_{y_{N}} \\[1.5mm]
& =
\eta_{y_{1}} \otimes \cdots \otimes \eta_{y_{L}} \otimes 
\eta_{r_i y_{L+1}} \otimes \cdots \otimes
\eta_{r_i y_{K}} \otimes \eta_{y_{K+1}} \otimes
\cdots \otimes \eta_{y_{N}} 
\quad \text{by \eqref{eq:etax}} \\[1.5mm]
& = \sigma_{N}(e_{i}\eta) \quad \text{by \eqref{eq:sn3}}.
\end{align*}
This completes the proof of the proposition.
\end{proof}
%
%
\subsection{Proof of Proposition~\ref{prop:conn-isom}.}
\label{subsec:prf-connisom}

\begin{lem}\label{lem:N-multi}
Let $X$ be a monomial in the Kashiwara operators, i.e., 
$X = g_m g_{m-1} \cdots g_2 g_1$ with 
$g_k \in \bigl\{ e_i,\,f_i \mid i \in I_{\af} \bigr\}$ 
for $1 \le k \le m$. 
\begin{enu}
\item
If $X u_{\lambda } \neq \bzero$, then 
$X \eta_e \neq \bzero$. 
We set $b_{k}:=g_k g_{k-1} \cdots g_1u_{\lambda}$ 
for $0 \le k \le m$, and 
let $N \in \BZ_{> 0}$ be a common multiple of 
$N_{b_{k}}$, $0 \le k \le m$, and $N_{\lambda}$
(see Proposition~\ref{prop:similarB} and the comment 
preceding Lemma~\ref{lem:Na}). 
If $\sigma_N (X u_{\lambda }) = 
u_{x_1} \otimes \cdots \otimes u_{x_N}$ 
with $x_1,\,\ldots,\,x_N \in (W^J)_{\af}$, then 
$\sigma_N (X \eta_e) = 
\eta_{x_1} \otimes \cdots \otimes \eta_{x_N}$.

\item
If $X \eta_e \neq \bzero$, then $X u_{\lambda } \neq \bzero$. 
Let $N \in \BZ_{>0}$ be a multiple of $N_{\lambda}$, 
and write $\sigma_N (X\eta_e)$ as 
$\sigma_N (X \eta_e) = 
\eta_{x_1} \otimes \cdots \otimes \eta_{x_N}$ 
for some $x_1,\,\ldots,\,x_N \in (W^J)_{\af}$. 
Then, $\sigma_N (X u_{\lambda }) = 
u_{x_1} \otimes \cdots \otimes u_{x_N}$.
\end{enu}
\end{lem}

\begin{proof}
We give a proof only for part (1); 
the proof for part (2) is similar. 
We proceed by induction on $m$. 
If $m=0$, then the assertion is obvious. 
Assume that $m > 0$, and take $i \in I_{\af}$ for which $g_{m}=e_{i}$ or $f_{i}$. 
Set $Y := g_{m-1} \cdots g_2 g_1$, and 
write $\sigma_N (Yu_{\lambda })=\sigma_N (b_{m-1})$ as
$\sigma_N (Yu_{\lambda }) = 
u_{y_1} \otimes \cdots \otimes u_{y_N}$ 
for some $y_1,\,\ldots,\,y_N \in (W^J)_{\af}$. 
Then, by our induction hypothesis, 
$Y \eta_e \neq \bzero$, and 
$\sigma_N (Y\eta_e) = \eta_{y_1} \otimes \cdots \otimes \eta_{y_N}$.
Also, we see from the tensor product rule for crystals, 
using \eqref{eq:ext-eta}, that 
$\ve_{i}(\sigma_N (Yu_{\lambda}))=\ve_{i}(\sigma_N (Y\eta_e))$ and 
$\vp_{i}(\sigma_N (Yu_{\lambda}))=\vp_{i}(\sigma_N (Y\eta_e))$, 
and hence that $\ve_{i}(Yu_{\lambda})=\ve_{i}(Y\eta_e)$ and 
$\vp_{i}(Yu_{\lambda})=\vp_{i}(Y\eta_e)$ by 
\eqref{eq:sigmaB1}, \eqref{eq:SN1}. Therefore, 
$Xu_{\lambda}=g_{m}Yu_{\lambda} \ne \bzero$ implies that 
$X\eta_{e}=g_{m}Y\eta_{e} \ne \bzero$. Furthermore, if
\begin{align*}
\sigma_N (Xu_{\lambda }) = 
g_m^N (u_{y_1} \otimes \cdots \otimes u_{y_N}) = 
g_m^{p_1} u_{y_1} \otimes \cdots \otimes g_m^{p_N} u_{y_N}
\end{align*}
for some $p_1,\,\ldots,\,p_N \in \BZ_{\ge 0}$ 
with $p_1 + \cdots + p_N = N$, then it follows 
from the tensor product rule for crystals, 
together with \eqref{eq:ext-eta}, that 
\begin{align*}
\sigma_N (X\eta_{e}) = 
g_m^N (\eta_{y_1} \otimes \cdots \otimes \eta_{y_N}) = 
g_m^{p_1} \eta_{y_1} \otimes \cdots \otimes g_m^{p_N} \eta_{y_N}. 
\end{align*}
Since $g_m^{p_1} u_{y_1} \otimes \cdots \otimes g_m^{p_N} u_{y_N} = 
\sigma_N (Xu_{\lambda}) = 
u_{x_1} \otimes \cdots \otimes u_{x_N}$ by the assumption, 
we deduce, by using \eqref{eq:etax}, that 
$g_m^{p_L} \eta_{y_L} = \eta_{x_L}$ for all $1 \le L \le N$. 
This proves part (1). 
\end{proof}

\begin{proof}[{Proof of Proposition \ref{prop:conn-isom}}]
It suffices to show the following 
for monomials $X$, $Y$ in the Kashiwara operators 
(cf. \cite[Proof of Theorem~4.1]{Kas96} and 
\cite[Proof of Theorem~5.1]{NS03}):
\begin{itemize}

\item[(i)]
$X u_{\lambda } \neq \bzero$ in $\Bo$ if and only if 
$X \eta_e \neq \bzero$ in $\sLSo$; 

\item[(ii)] 
$X u_{\lambda } = Y u_{\lambda }$ in $\Bo$ 
if and only if $X \eta_e = Y \eta_e$ in $\sLSo$.
\end{itemize}
Part (i) has already been shown in Lemma~\ref{lem:N-multi}. 
Let us show part (ii). 
Assume that $X u_{\lambda } = Y u_{\lambda } \neq \bzero$. 
By Lemma \ref{lem:N-multi}\,(1), 
we have $X \eta_e \neq \bzero$ and $Y \eta_e \neq \bzero$. 
Take a multiple $N \in \BZ_{>0}$ of $N_{\lambda}$ 
such that the assumption of 
Lemma \ref{lem:N-multi}\,(1) is satisfied for both of 
the elements $Xu_{\lambda}$ and $Yu_{\lambda}$, and 
write $\sigma_N (X u_{\lambda})$ and $\sigma_N (Y u_{\lambda })$ as:
\begin{align*}
& \sigma_N (X u_{\lambda }) = 
  u_{x_1} \otimes u_{x_2} \otimes \cdots \otimes u_{x_N}, &
&\sigma_N (Y u_{\lambda }) = 
  u_{y_1} \otimes u_{y_2} \otimes \cdots \otimes u_{y_N}, 
\end{align*}
for some $x_1,\,\ldots,\,x_N \in (W^J)_{\af}$ and 
$y_1,\,\ldots,\,y_N \in (W^J)_{\af}$. 
Then, by Lemma \ref{lem:N-multi}\,(1),
\begin{align*}
& \sigma_N (X \eta_e) = 
  \eta_{x_1} \otimes \eta_{x_2} \otimes \cdots \otimes \eta_{x_N}, &
& \sigma_N (Y \eta_e) = 
  \eta_{y_1} \otimes \eta_{y_2} \otimes \cdots \otimes \eta_{y_N}.
\end{align*}
Since $X u_{\lambda } = Y u_{\lambda }$, 
we have $x_L = y_L$ for all $1 \le L \le N$. 
Therefore, we deduce that 
$\sigma_N (X \eta_e) = \sigma_N (Y \eta_e)$, 
and hence $X \eta_e = Y \eta_e$ by the injectivity of $\sigma_N$. 
Thus, we have proved the ``only if" part of part (ii). 
The ``if" part can be shown similarly; 
use Lemma~\ref{lem:N-multi}\,(2) instead of Lemma~\ref{lem:N-multi}\,(1). 
This completes the proof of Proposition \ref{prop:conn-isom}. 
\end{proof}
%
%
\begin{rem} \label{rem:etae-ext}
Since $\CB_{0}(\lambda) \cong \BB^{\si}_{0}(\lambda)$ as crystals, 
we can define the action of the Weyl group $W_{\af}$
on $\BB^{\si}_{0}(\lambda)$ by the same formula as \eqref{eq:W-action} 
for the one on $\CB_{0}(\lambda)$. Moreover, from equation \eqref{eq:etax},
we see that the element $\eta_{e}=(e \,;\,0,\,1) \in \BB^{\si}_{0}(\lambda)$
is an extremal element. 
\end{rem}
%
%
\section{Directed paths in $\SBa$.}
\label{sec:directed-path}
%
%
\subsection{Some technical lemmas.}
\label{subsec:technical}

Let $J$, $K$ be subsets of $I$ such that $K \subset J$; 
we see by the definitions that $(W^{J})_{\af} \subset (W^{K})_{\af}$. 
%
%
\begin{lem} \label{lem:PJK}
Let $J$, $K$ be subsets of $I$ such that $K \subset J$, 
and let $x \in (W^{K})_{\af}$ and $\beta \in \prr$ be 
such that $x \edge{\beta} r_{\beta}x$ in $\SBo{K}$. 
Then there exists a directed path from 
$\PJ(x)$ to $\PJ(r_{\beta}x)$ in $\SB$. 
\end{lem}

\begin{proof}
Let $\lambda$ and $\Lambda$ be (arbitrary) elements in $P^{+}$
such that $J_{\lambda}=
\bigl\{ i \in I \mid \pair{\alpha_i^{\vee}}{\lambda}= 0 \bigr\}=J$, 
and $J_{\Lambda}=\bigl\{ i \in I \mid 
\pair{\alpha_i^{\vee}}{\Lambda}= 0 \bigr\}=K$. 
%
%
\begin{claim} \label{c:PJK1}
If $\beta$ is a simple root, then the assertion of the lemma holds. 
\end{claim}

\noindent
{\it Proof of Claim~\ref{c:PJK1}.}
Assume that $\beta=\alpha_{i}$ for some $i \in I_{\af}$. 
Because $x \edge{\alpha_i} r_{i}x$ in $\SBo{K}$, 
we see from \eqref{eq:simple2} that 
$\pair{\alpha_{i}^{\vee}}{x\Lambda} > 0$, and hence
$\pair{\alpha_{i}^{\vee}}{x\lambda} \ge 0$ since $K \subset J$. 
Assume that $\pair{\alpha_{i}^{\vee}}{x\lambda} > 0$. 
Since $x\lambda=\PJ(x)\lambda$, we see 
from Lemma~\ref{lem:r_i} and \eqref{eq:simple2} that 
$r_{i}\PJ(x) \in (W^{J})_{\af}$, and 
$\PJ(x) \edge{\alpha_{i}} r_{i}\PJ(x)$ in $\SB$;
note that $r_{i}\PJ(x)=\PJ(r_{i}\PJ(x))=\PJ(r_{i}x)$. 
Thus we obtain 
$\PJ(x) \edge{\alpha_{i}} \PJ(r_{i}x)$ in $\SB$. 
If $\pair{\alpha_{i}^{\vee}}{x\lambda} = 0$, then 
we see that $x^{-1}\alpha_{i} \in (\Delta_{J})_{\af}$, 
which implies that $\PJ(r_{i}x)=
\PJ(xr_{x^{-1}\alpha_{i}})=\PJ(x)$.  
This proves Claim~\ref{c:PJK1}. \bqed

\vspace{3mm}

Let us consider the case of general $\beta \in \prr$. 
By \cite[Lemma~1.4]{AK}, 
there exist $i_{1},\,i_{2},\,\dots,\,i_{k} \in I_{\af}$ such that
\begin{equation*}
x\Lambda=\mu_{k} \edge{\alpha_{i_k}} \mu_{k-1} \edge{\alpha_{i_{k-1}}} 
\cdots \edge{\alpha_{i_2}} \mu_{1} \edge{\alpha_{i_1}} \mu_{0}, 
\end{equation*}
with $\mu_{0} \equiv \Lambda \mod \BQ\delta$, where 
$\mu_{m}=r_{i_{m+1}} \cdots r_{i_{k}}x\Lambda$ 
for $0 \le m \le k$. We show the assertion 
by induction on $k$. 
Assume that $k=0$; observe that in this case, 
$x=z_{\xi}t_{\xi}$ for some $\xi \in \kad$ since 
$x\Lambda \equiv \Lambda \mod \BQ\delta$. 
It follows from Lemma~\ref{lem:zt} that 
$\beta=\alpha_{i}$ for some $i \in I \setminus K$, 
and hence the assertion follows immediately 
from Claim~\ref{c:PJK1}.
Assume that $k > 0$. For simplicity of notation, 
we set $y:=r_{\beta}x \in (W^{K})_{\af}$ and $i:=i_{k}$. 
Then, by \eqref{eq:simple2}, 
we have $x \edge{\alpha_{i}} r_{i}x$ in $\SBo{K}$. 
Write $y^{-1}\alpha_{i}=\alpha+n\delta$ with 
$\alpha \in \Delta$ and $n \in \BZ$; 
by Lemma~\ref{lem:dia1}, $\alpha \notin \Delta_{K}^{+}$.
If $\alpha$ is a negative root, then 
$\beta=\alpha_{i}$ by Lemma~\ref{lem:dia1}, and hence 
the assertion follows immediately from Claim~\ref{c:PJK1}.
Assume that $\alpha \in \Delta^{+} \setminus 
\Delta_{K}^{+}$. By Lemma~\ref{lem:dia1} applied to 
the subset $K$ of $I$, we obtain the following diagram in $\SBo{K}$:
\begin{equation} \label{CD:PJK2}
\begin{diagram}
\node{x}\arrow{s,l}{\alpha_{i}} \arrow{e,t}{\beta} 
\node{y}\arrow{s,r,..}{\alpha_{i}} \\
\node{r_{i}x} \arrow{e,t,..}{r_{i}\beta} \node{r_{i}y}
\end{diagram}
\end{equation}
We see from Claim~\ref{c:PJK1} that 
there exists a directed path from 
$\PJ(x)$ to $\PJ(r_{i}x)$ in $\SB$. 
Also, by our induction hypothesis applied to $r_{i}x$, 
there exists a directed path from 
$\PJ(r_{i}x)$ to $\PJ(r_{i}y)$ in $\SB$. 
Concatenating these, we obtain 
a directed path from $\PJ(x)$ to $\PJ(r_{i}y)$ in $\SB$, 
which proves the lemma in the case that 
$\PJ(r_{i}y)=\PJ(y)$. Now, let us assume that 
$\PJ(r_{i}y) \ne \PJ(y)$. We see from \eqref{CD:PJK2} 
and \eqref{eq:simple2} that $\pair{\alpha_{i}^{\vee}}{y\Lambda} > 0$, 
and hence $\pair{\alpha_{i}^{\vee}}{y\lambda} \ge 0$. 
Since $\PJ(yr_{y^{-1}\alpha_{i}})=
\PJ(r_{i}y) \ne \PJ(y)$ by our assumption, 
it follows from Proposition~\ref{prop:P}, 
together with Remark~\ref{rem:stabilizer}, that 
$\pair{\alpha_{i}^{\vee}}{y\lambda} \ne 0$, and hence 
$\pair{\alpha_{i}^{\vee}}{y\lambda} > 0$.
Since $\PJ(y)\lambda = y\lambda$, 
we see from Lemma~\ref{lem:r_i} that 
$r_{i}\PJ(y) \in (W^{J})_{\af}$; 
note that $r_{i}\PJ(y)=\PJ(r_{i}\PJ(y))=\PJ(r_{i}y)$, and 
$\pair{\alpha_{i}^{\vee}}{\PJ(r_{i}y)\lambda} < 0$. 
Similarly, we can show that 
$\pair{\alpha_{i}^{\vee}}{\PJ(x)\lambda} \ge 0$ 
since $x \edge{\alpha_{i}} r_{i}x$ in $\SBo{K}$. 
Therefore, by applying Lemma~\ref{lem:directed}\,(3) 
to the directed path from $\PJ(x)$ to $\PJ(r_{i}y)$ in $\SB$ 
obtained above, we see that 
there exists a directed path from $\Pi^{J}(x)$ to 
$r_{i}\PJ(r_{i}y)=\PJ(y)$ in $\SB$. This proves the lemma. 
\end{proof}

Here, we make the following remark.
%
%
\begin{rem} \label{rem:weight}
Let $J$ be a subset of $I$. 
Let $x,\,y \in (W^{J})_{\af}$ and $\beta \in \prr$ be 
such that $x \edge{\beta} y$ in $\SB$. 
Write $x$ and $y$ as 
$x=wz_{\xi}t_{\xi}$ and 
$y=vz_{\zeta}t_{\zeta}$, with $w,\,v \in W$ and 
$\xi,\,\zeta \in \jad$ (see \eqref{eq:W^J_af}). 
We see from Corollary~\ref{cor:pi_} that 
$\beta=w\gamma+n\delta$ for some $\gamma \in \Delta^{+}
\setminus \Delta_{J}^{+}$ and $n \in \bigl\{0,\,1\bigr\}$. 
Also, it follows easily that 
$\zeta-\xi = nz_{\xi}^{-1}\gamma^{\vee}$. 
In particular, we have 
$\pJ{\zeta-\xi}=n\pJ{\gamma^{\vee}} \in \QJp{J^{c}}=
\sum_{j \in J^{c}} \BZ_{\ge 0} \alpha_{j}^{\vee}$, 
where $J^{c}:=I \setminus J$, and 
$\pJ{\,\cdot\,}=\pJ{\,\cdot\,}_{J^{c}}:
Q^{\vee} \twoheadrightarrow 
Q_{J^{c}}^{\vee} = \bigoplus_{j \in J^{c}} \BZ\alpha_{j}^{\vee}$ 
denote the projection from $Q^{\vee}=Q_{J^{c}}^{\vee} \oplus Q_{J}^{\vee}$
onto $Q_{J^{c}}^{\vee}$ with kernel $Q_{J}^{\vee}$. 
\end{rem}
%
%
\begin{lem} \label{lem:di1}
Let $J$ be a subset of $I$. 
For each $i \in J^{c}=I \setminus J$ and $\xi \in \jad$, 
there exists a positive real root $\beta \in \prr$ 
of the form $\beta=-\gamma+\delta$, 
with $\gamma \in 
\Delta^{+} \setminus \Delta_{J}^{+}$,
satisfying the conditions that 
$\pJ{\gamma^{\vee}}=\alpha_{i}^{\vee}$, 
$r_{\beta}z_{\xi}t_{\xi} \in (W^{J})_{\af}$, and 
$r_{\beta}z_{\xi}t_{\xi} \edge{\beta} z_{\xi}t_{\xi}$ in $\SB$.
\end{lem}

\begin{proof}
Let $i \in I \setminus J$, and let $\xi \in \jad$. 
Since $\sil(r_{i}t_{\xi-\alpha_{i}^{\vee}})=
\ell(r_{i})+2\pair{\xi-\alpha_{i}^{\vee}}{\rho}=
2\pair{\xi}{\rho}-1=\sil(t_{\xi})-1$, 
we see that 
$r_{i}t_{\xi-\alpha_{i}^{\vee}}=
r_{\beta'}t_{\xi} \edge{\beta'} t_{\xi}$ in 
$\SBo{\emptyset}$, with $\beta'=-\alpha_{i}+\delta$. 
Therefore, by Lemma~\ref{lem:PJK}, there exists a directed path 
from $\PJ(r_{i}t_{\xi-\alpha_{i}^{\vee}})$ to $\PJ(t_{\xi})$ in $\SB$. 
Since $i \in I \setminus J$, we see from Lemma~\ref{lem:J-adj}\,(2) that
$\PJ(r_{i}t_{\xi-\alpha_{i}^{\vee}})=r_{i}z_{\xi-\alpha_{i}^{\vee}}t_{\zeta'}$, 
with $\zeta'=\xi-\alpha_{i}^{\vee}+
\phi_{J}(\xi-\alpha_{i}^{\vee})$. 
Also, we see from Lemma~\ref{lem:J-adj}\,(1), (2) that 
$\PJ(t_{\xi})=z_{\xi}t_{\xi}$. 
Remark that 
$\pJ{\xi-\zeta'} = \alpha_{i}^{\vee}$. 
Let $x \edge{\beta} y:=z_{\xi}t_{\xi}$, $x \in (W^{J})_{\af}$, 
$\beta \in \prr$, be the final edge of the directed path 
in $\SB$ from $\PJ(r_{i}t_{\xi-\alpha_{i}^{\vee}})$ to 
$\PJ(t_{\xi})=z_{\xi}t_{\xi}$ above:
%
%
\begin{equation} \label{eq:dp1}
\PJ(r_{i}t_{\xi-\alpha_{i}^{\vee}}) \rightarrow \cdots \rightarrow 
x \edge{\beta} y=z_{\xi}t_{\xi}
\quad \text{in $\SB$}.
\end{equation}
Here, let us write $x$ as $x=vz_{\zeta}t_{\zeta}$, 
with $v \in W^{J}$ and $\zeta \in \jad$. 
By applying Remark~\ref{rem:weight} to 
each edge of the directed path from 
$\PJ(r_{i}t_{\xi-\alpha_{i}^{\vee}})=
 r_{i}z_{\xi-\alpha_{i}^{\vee}}t_{\zeta'}$ to 
$x=vz_{\zeta}t_{\zeta}$ (resp., the final edge 
$x=vz_{\zeta}t_{\zeta} \edge{\beta} y=z_{\xi}t_{\xi}$), 
we obtain $\pJ{\zeta-\zeta'} \in \QJp{J^c}$ (resp., 
$\pJ{\xi-\zeta} \in \QJp{J^c}$). Therefore, we deduce that
\begin{equation*}
\alpha_{i}^{\vee} = \pJ{\xi-\zeta'} = 
 \underbrace{\pJ{\xi-\zeta}}_{\in \QJp{J^c}} + 
 \underbrace{\pJ{\zeta-\zeta'}}_{\in \QJp{J^c}},
\end{equation*}
which implies that 
$\pJ{\xi-\zeta} = 0$ or $\alpha_{i}^{\vee}$. 
Now, let $\lambda \in P^{+}$ be an (arbitrary) element 
such that $J_{\lambda}=J$. 
It follows immediately 
from Proposition~\ref{prop:sib-lv0} that 
$x\lambda \edge{\beta} y\lambda=z_{\xi}t_{\xi}\lambda$ in $\LP$. 
Since $z_{\xi}t_{\xi}\lambda \equiv \lambda \mod \BQ\delta$, 
we see that $\beta$ is of the form 
$\beta=-\gamma+\delta$ for some 
$\gamma \in \Delta^{+} \setminus \Delta_{J}^{+}$. 
Hence a direct computation shows that 
$\zeta=\xi-z_{\xi}^{-1}\gamma^{\vee}$; 
note that $\pJ{z_{\xi}^{-1}\gamma^{\vee}}=
\pJ{\gamma^{\vee}} \ne 0$. 
Since $\pJ{\xi-\zeta} = 0$ or $\alpha_{i}$ as seen above, 
we conclude that $\pJ{\gamma^{\vee}}=\alpha_{i}$. 
This proves the lemma. 
\end{proof}
%
%
\subsection{Directed paths from a translation to another translation.}

In this subsection, we fix $\lambda \in P^+$, 
and set $J:=J_{\lambda}= 
\bigl\{ i \in I \mid \pair{\alpha_i^{\vee}}{\lambda}= 0 \bigr\}$.
Recall that $\pJ{\,\cdot\,}=\pJ{\,\cdot\,}_{J^c}:
Q^{\vee} \twoheadrightarrow Q_{J^{c}}^{\vee}$ denote
the projection from $Q^{\vee}=Q_{J^{c}}^{\vee} \oplus Q_{J}^{\vee}$
onto $Q_{J^{c}}^{\vee}$ with kernel $Q_{J}^{\vee}$, where $J^{c}=I \setminus J$. 
Also, for a rational number $0 < a \le 1$, we set 
\begin{equation*}
\Jca:=
\bigl\{ i \in J^{c} \mid 
 a \pair{\alpha_i^{\vee}}{\lambda} \in \BZ \bigr\}, \qquad
\Ia :=
\bigl\{ i \in I \mid 
 a \pair{\alpha_i^{\vee}}{\lambda} \in \BZ \bigr\}=
\Jca \cup J.
\end{equation*}
%
%
\begin{lem} \label{lem:di3}
Let $\zeta,\,\xi \in \jad$, and let $0 < a \le 1$ 
be a rational number. If $\pJ{\xi-\zeta} \in \QJp{\Jca}$, 
then there exists a directed path 
from $z_{\zeta}t_{\zeta}$ to $z_{\xi}t_{\xi}$ in $\SBa$. 
\end{lem}

\begin{proof}
For $\alpha^{\vee}=\sum_{i \in I} c_{i}\alpha_{i}^{\vee}$ 
with $c_{i} \in \BZ_{\ge 0}$, $i \in I$, we set 
$\Ht(\alpha^{\vee}) := \sum_{i \in I} c_{i} \in \BZ_{\ge 0}$. 
We prove the lemma 
by induction on $\Ht(\pJ{\xi-\zeta})$. Assume first that 
$\Ht(\pJ{\xi-\zeta}) = 0$. Then, we have $\pJ{\xi-\zeta}=0$, 
which implies that $\xi-\zeta \in Q_{J}^{\vee}$, and hence 
$\xi=\zeta$ by Lemma~\ref{lem:J-adj}\,(1). 
Thus there exists a directed path (of length zero) 
from $z_{\zeta}t_{\zeta}$ to 
$z_{\xi}t_{\xi} = z_{\zeta}t_{\zeta}$ in $\SBa$. 

Assume next that $\Ht(\pJ{\xi-\zeta}) > 0$. 
Take $i \in \Jca$ such that 
$\pJ{\xi-\zeta}-\alpha_{i}^{\vee} \in \QJp{\Jca}$, and set 
$\xi':=\xi-\alpha_{i}^{\vee}+\phi_{J}(\xi-\alpha_{i}^{\vee}) 
\in \jad$ (see Lemma~\ref{lem:J-adj}\,(1)); note that 
$\pJ{\xi'}=\pJ{\xi}-\alpha_{i}^{\vee}$, and hence 
$\Ht (\pJ{\xi'-\zeta}) = \Ht (\pJ{\xi-\zeta})-1$. 
Therefore, by our induction hypothesis, there exists a directed 
path from $z_{\zeta}t_{\zeta}$ to $z_{\xi'}t_{\xi'}$ in $\SBa$.
Hence it suffices to show that there exists a directed 
path from $z_{\xi'}t_{\xi'}$ to $z_{\xi}t_{\xi}$ in $\SBa$.
We take $\beta=-\gamma+\delta$ of Lemma~\ref{lem:di1} 
(for the $\xi \in \jad$ and the $i \in \Jca$ above). 
Since $\pJ{\gamma^{\vee}}=\alpha_{i}^{\vee}$ and $i \in \Jca$, 
we see that the edge $r_{\beta}z_{\xi}t_{\xi} 
\edge{\beta} z_{\xi}t_{\xi}$ in $\SB$ is in fact an edge in $\SBa$. 
Write $r_{\beta}z_{\xi}t_{\xi}$ as 
$r_{\beta}z_{\xi}t_{\xi}=wz_{\xi''}t_{\xi''}$, 
with $w \in W^{J}$ and $\xi'' \in \jad$. 
Since $r_{\beta}z_{\xi}t_{\xi} = 
r_{\gamma}t_{-\gamma^{\vee}}z_{\xi}t_{\xi}$ by \eqref{eq:refl},
we have $\xi''=\xi-z_{\xi}^{-1}\gamma^{\vee}$. 
Because $\pJ{\xi'-\xi''}= (\pJ{\xi}-\alpha_{i}^{\vee})-
(\pJ{\xi}-\pJ{z_{\xi}^{-1}\gamma^{\vee}})=
(\pJ{\xi}-\alpha_{i}^{\vee})-
(\pJ{\xi}-\alpha_{i}^{\vee})=0$, 
we deduce from Lemma~\ref{lem:J-adj}\,(1) 
that $\xi''=\xi'$. Also, we see that 
$w \in W^{J}$ is equal to $r_{\gamma}z_{\xi}z_{\xi'}^{-1}$.
Since $\pJ{\gamma^{\vee}}=\alpha_{i}^{\vee}$ 
and $i \in \Jca$, it follows immediately that 
$\gamma \in \Delta_{\Ia}$, and hence 
$r_{\gamma} \in W_{\Ia} = \langle r_{j} \mid j \in \Ia \rangle$.
This implies that $w=r_{\gamma}z_{\xi}z_{\xi'}^{-1} \in W_{\Ia}$ 
since $z_{\xi}z_{\xi'}^{-1} \in W_{J} \subset W_{\Ia}$. 
Therefore, if $w=r_{i_{1}}r_{i_{2}} \cdots r_{i_{k}}$ is 
a reduced expression of $w$, then $i_{l} \in \Ia$ 
for all $1 \le l \le k$. 
From this, it follows 
(see also the argument in Step 1 in 
the proof of Proposition~\ref{prop:sib-lv0}) that
\begin{equation*}
z_{\xi'}t_{\xi'} \edge{\alpha_{i_{k}}} 
r_{i_{k}} z_{\xi'}t_{\xi'} \edge{\alpha_{i_{k-1}}}
r_{i_{k-1}}r_{i_{k}} z_{\xi'}t_{\xi'} \edge{\alpha_{i_{k-2}}} \cdots 
\edge{\alpha_{i_1}} r_{i_{1}}r_{i_{2}} \cdots r_{i_{k}}z_{\xi'}t_{\xi'}=
wz_{\xi'}t_{\xi'}
\end{equation*}
is a directed path in $\SBa$. Concatenating 
this directed path with 
$wz_{\xi'}t_{\xi'}=r_{\beta}z_{\xi}t_{\xi} 
\edge{\beta} z_{\xi}t_{\xi}$ above, we obtain a directed path 
from $z_{\xi'}t_{\xi'}$ to $z_{\xi}t_{\xi}$ in $\SBa$, 
as desired. 
\end{proof}

%
\begin{prop} \label{prop:di2}
Let $\zeta,\,\xi \in \jad$, and let $0 < a \le 1$ 
be a rational number. There exists a directed path of nonzero length
from $z_{\zeta}t_{\zeta}$ to $z_{\xi}t_{\xi}$ in $\SBa$ 
if and only if the set $\Jca=\bigl\{ i \in J^{c} \mid 
 a \pair{\alpha_i^{\vee}}{\lambda} \in \BZ \bigr\}$ is not empty, 
and $\pJ{\xi-\zeta} \in \QJp{\Jca} \setminus \{0\}$.
\end{prop}

\begin{proof}
We prove the ``only if'' part. Let 
\begin{equation*}
z_{\zeta}t_{\zeta}=y_0 \edge{\beta_1} y_1 \edge{\beta_2} 
\cdots \edge{\beta_k} y_k= z_{\xi}t_{\xi}
\end{equation*}
be a directed path of nonzero length from $z_{\zeta}t_{\zeta}$ 
to $z_{\xi}t_{\xi}$ in $\SBa$. By Lemma~\ref{lem:zt}, 
we see that $\beta_{1}=\alpha_{i}$ for some $i \in J^{c}$. 
Since $a\pair{\alpha_{i}^{\vee}}{\lambda} = 
a\pair{\beta_1^{\vee}}{z_{\zeta}t_{\zeta}\lambda} \in \BZ$, 
it follows immediately that $i \in \Jca$, and in particular, 
$\Jca \ne \emptyset$. 
We write $y_{l} \in (W^{J})_{\af}$, $0 \le l \le k$, 
and $\beta_{l} \in \prr$ (see Corollary~\ref{cor:pi_}), 
$1 \le l \le k$, as: 
\begin{equation*}
\begin{cases}
y_{l}=w_{l}z_{\xi_{l}}t_{\xi_{l}} & 
\text{with $w_{l} \in W^{J}$ and $\xi_{l} \in \jad$}, \\[1.5mm]
\beta_{l}=w_{l-1}\gamma_{l}+n_{l}\delta & 
\text{with $\gamma_{l} \in \Delta^{+} \setminus \Delta_{J}^{+}$ 
and $n_{l} \in \bigl\{0,\,1\bigr\}$}.
\end{cases}
\end{equation*}
Then we deduce from Remark~\ref{rem:weight} that
\begin{equation} \label{eq:weight2}
\xi-\zeta = \sum_{l=1}^{k} (\xi_{l}-\xi_{l-1})=
\sum_{l=1}^{k} n_{l}z_{\xi_{l-1}}^{-1} \gamma_{l}^{\vee}, 
\quad \text{and hence} \quad
[\xi-\zeta]=
\sum_{l=1}^{k} n_{l} [\gamma_{l}^{\vee}] \in Q_{J^{c}}^{\vee +}.
\end{equation}
Also, using \eqref{eq:W^J_af}, we see by direct computation that 
$w_{l}=\mcr{w_{l-1}r_{\gamma_{l}}}$ for all $1 \le l \le k$. 
Let us show by induction on $l$ that 
$w_{l} \in W_{\Ia}$ and $\gamma_{l} \in \Delta_{\Ia}$ for all $1 \le l \le k$. 
Since $w_{0}=e$, and since $\beta_{1}=\alpha_{i}$ with $i \in \Jca \subset \Ia$ 
as seen above, it follows immediately that 
$\gamma_{1}=\alpha_{i} \in \Delta_{\Ia}$ and 
$w_{1}=\mcr{ w_{0}r_{\gamma_{1}} }=\mcr{r_{i}} \in W_{\Ia}$. 
Assume that $l > 0$. 
By Proposition~\ref{prop:sib-lv0}, we have 
$y_{l-1}\lambda \edge{\beta_{l}} y_{l}\lambda$ in $\LPa$. 
We deduce from \cite[Lemma~3.11]{NS08}, 
together with \cite[Lemma~2.13]{NS08}, that 
$w_{l-1}\gamma_{l} \in \Delta_{\Ia}$. 
%
Since $w_{l-1} \in W_{\Ia}$ by our induction hypothesis, 
we obtain $\gamma_{l} \in \Delta_{\Ia}$. Also, 
since $w_{l}=\mcr{w_{l-1}r_{\gamma_{l}}}$ 
with $w_{l-1} \in W_{\Ia}$ and $\gamma_{l} \in \Delta_{\Ia}$, 
it follows immediately that $w_{l} \in W_{\Ia}$. 
Thus we have shown that 
$w_{l} \in W_{\Ia}$ and $\gamma_{l} \in \Delta_{\Ia}$ for all $1 \le l \le k$. 
Combining this fact and \eqref{eq:weight2}, we conclude that 
$\pJ{\xi-\zeta} \in \QJp{\Jca}$. 
Suppose, for a contradiction, that $\pJ{\xi-\zeta} = 0$. Then, 
we have $\xi-\zeta \in Q_{J}^{\vee}$.
This implies that $\xi=\zeta$ by Lemma~\ref{lem:J-adj}\,(1), 
which contradicts the assumption. 
Thus we have proved the ``only if'' part. 
The ``if'' part follows immediately from Lemma~\ref{lem:di3}. 
This completes the proof of the proposition.
\end{proof}
%
%
\section{Connected components of $\sLS$.}
\label{sec:ConnComp}
Throughout this section, we fix $\lambda \in P^+$, 
and set $J:=J_{\lambda}= 
\bigl\{ i \in I \mid \pair{\alpha_i^{\vee}}{\lambda}= 0 \bigr\}$.
Write $\lambda \in P^+$ as 
$\lambda = \sum_{i \in I} m_i \varpi_i$ 
with $m_i \in \BZ_{\ge 0}$, $i \in I$, and set 
\begin{equation*}
\Turn(\lambda ) := 
 \bigl\{ k/m_i \mid 
 i \in J^{c}=I \setminus J, \ 0 \le k \le m_i \bigr\};
\end{equation*}
note that $m_{i}=\pair{\alpha_{i}^{\vee}}{\lambda}$. 
Also, for simplicity of notation, we set 
$T_{\xi} := \PJ (t_{\xi}) = z_{\xi} t_{\xi} \in (W^J)_{\af}$ 
for $\xi \in \jad$.
%
%
\subsection{An extremal element in each connected component.}
\label{subsec:ExtConnComp}

The next proposition follows immediately 
from Proposition \ref{prop:di2}.

\begin{prop}\label{prop:conn}
Let $\xi_1,\,\ldots,\,\xi_{s-1},\,\xi_{s} \in \jad$. 
An element 
\begin{align} \label{eq:extr}
\eta = 
(T_{\xi_1}, \ldots , T_{\xi_{s-1}}, T_{\xi_{s}} \,;\, 
 a_0 , a_1 , \ldots , a_{s-1} , a_s )
\end{align}
is contained in $\sLS$ if and only if 
$a_u \in \Turn(\lambda)$ for all $0 \le u \le s$ and 
$\pJ{\xi_u - \xi_{u+1}} \in \QJp{\Jcb{a_{u}}} \setminus \{0\}$ 
for all $1 \le u \le s-1$.  
\end{prop}
%
%
\begin{prop}\label{prop:unique}
Each connected component of $\sLS$ contains 
a unique element of the form \eqref{eq:extr} with $\xi_{s}=0$. 
\end{prop}

In order to prove this proposition, 
we need some lemmas. Let $N \in \BZ_{>0}$. 
For simplicity of notation, we set 
$[y_1 , y_2 , \ldots , y_N] := 
\eta_{y_1} \otimes \eta_{y_2} 
\otimes \cdots \otimes \eta_{y_N} \in \sLS^{\otimes N}$ 
for $y_1,\,y_2,\,\ldots,\,y_N \in (W^J)_{\af}$.

\begin{lem}\label{lem:tensor}
Let $N \in \BZ_{>0}$ be a multiple of $N_{\lambda}$ 
(see the comment preceding Lemma~\ref{lem:Na}). 
Let $\eta \in \sLS$, and 
write $\sigma_N (\eta ) \in \sLS^{\otimes N}$ as
$\sigma_N (\eta ) = 
[y_1 , y_2 , \ldots , y_N]$, with 
$y_1 , y_2 , \ldots , y_N \in (W^J)_{\af}$.
Let $X$ be a monomial in the root operators 
$e_i$ and $f_i$, $i \in I_{\af}$, and 
assume that $X \eta \neq \bzero$. Then, 
$\sigma_N (X \eta ) = [v_1 y_1 , v_2 y_2 , \ldots , v_N y_N ]$
for some $v_1 , v_2 , \ldots , v_N \in W_{\af}$ 
such that $v_M y_M \in (W^J)_{\af}$, $1 \le M \le N$.
\end{lem}

\begin{proof}
It suffices to show the assertion 
in the case when $X = e_i$ or $f_i$, $i \in I_{\af}$; 
the assertion for a general $X$ follows immediately 
by induction. From the definition of root operators, 
we obtain 
\begin{align*}
\sigma_N (X \eta ) = 
[y_1 , \ldots , y_{L} , r_i y_{L+1} , \ldots , r_i y_K , 
 y_{K+1} , \ldots , y_N]
\end{align*}
for some $1 \le L < K \le N$, with $r_i y_M \in (W^J)_{\af}$ 
for all $L+1 \le M \le K$. This proves the lemma.
\end{proof}

Let $N \in \BZ_{>0}$ be a multiple of $N_{\lambda}$. 
Let $\eta \in \sLS$ be of the form \eqref{eq:extr}. 
Then, $\sigma_N (\eta )$ is of the form
$\sigma_N (\eta ) = [T_{\zeta_1} , T_{\zeta_2} , \ldots , T_{\zeta_N}]$
for some $\zeta_1 , \zeta_2 , \ldots , \zeta_N \in \jad$. 
Let $X$ be a monomial in the root operators
$e_i$ and $f_i$, $i \in I_{\af}$, 
such that $X \eta \neq \bzero$; 
by Lemma~\ref{lem:tensor}, $\sigma_N (X \eta)$ is of the form
$\sigma_N (X \eta) = 
 [v_1 T_{\zeta_1} , v_2 T_{\zeta_2} , \ldots , v_N T_{\zeta_N}]$
for some $v_1 , v_2 , \ldots , v_N \in W_{\af}$ 
such that $v_M T_{\zeta_M} = 
v_M z_{\zeta_M} t_{\zeta_M} \in (W^J)_{\af}$, $1 \le M \le N$; 
note that $v_M \in (W^J)_{\af}$ for all $1 \le M \le N$ 
by Lemma~\ref{lem:xt}.

\begin{lem}\label{lem:vT}
Keep the notation and setting above. 
Let $\eta ' \in \sLS$ be another element of the form \eqref{eq:extr}, 
and write $\sigma_N (\eta ')$ as 
$\sigma_N (\eta ') = 
 [T_{\zeta '_1} , T_{\zeta '_2} , \ldots , T_{\zeta '_N}]$
for some $\zeta '_1 , \zeta '_2 , \ldots , \zeta '_N \in \jad$. 
Then, $X \eta ' \neq \bzero$, and 
$\sigma_N (X \eta ') = 
[v_1 T_{\zeta '_1} , v_2 T_{\zeta '_2} , \ldots , v_N T_{\zeta '_N}]$;
note that $v_M T_{\zeta'_M} \in (W^J)_{\af}$ 
for all $1 \le M \le N$ by Lemma~\ref{lem:xt} 
since $v_M \in (W^J)_{\af}$. 
\end{lem}

\begin{proof}
Let $X = g_m g_{m-1} \cdots g_2 g_1$, 
where $g_k \in \bigl\{ e_i , f_i \mid i \in I_{\af} \bigr\}$ 
for each $1 \le k \le m$. We show the assertion 
by induction on $m$. If $m=0$, then the assertion is obvious 
since $X = \mathrm{id}$. Assume that $m > 0$. 
Set $Y := g_{m-1} \cdots g_2 g_1$. 
Since $X \eta \neq \bzero$, it follows 
that $Y \eta \neq \bzero$. 
By Lemma~\ref{lem:tensor}, we can write $\sigma_N (Y \eta)$ as
$\sigma_N (Y \eta) = 
 [u_1 T_{\zeta_1} , u_2 T_{\zeta_2} , \ldots , u_N T_{\zeta_N}]$
for some $u_1 , u_2 , \ldots , u_N \in W_{\af}$ 
such that $u_M T_{\zeta_M} = u_M z_{\zeta_M} t_{\zeta_M} 
\in (W^J)_{\af}$, $1 \le M \le N$. 
By our induction hypothesis, we have $Y \eta ' \neq \bzero$, and 
$\sigma_N (Y \eta ') = 
 [u_1 T_{\zeta '_1} , u_2 T_{\zeta '_2} , \ldots , u_N T_{\zeta '_N}]$.
Now, let us take $i \in I_{\af}$ for which $g_m = e_i$ or $f_i$. 
We see from the definition of 
the root operator $e_i$ or $f_i$ that
$\sigma_N (X \eta ) = \sigma_N (g_m Y \eta )$ is of the form:
\begin{equation*}
\sigma_N (X \eta) = 
[u_1 T_{\zeta_1} , \ldots , u_{L} T_{\zeta_{L}} , 
 r_i u_{L+1} T_{\zeta_{L+1}} , \ldots , r_i u_K T_{\zeta_K} , 
 u_{K+1} T_{\zeta_{K+1}} , \ldots , u_N T_{\zeta_N}]
\end{equation*}
for some $0 \le L < K \le N$, 
with $r_i u_M T_{\zeta_M} \in (W^J)_{\af}$ for all $L+1 \le M \le K$; 
remark that $L$ and $K$ are determined 
by the function $H_i^{\ol{Y \eta}} (t) = 
\pair{\alpha_i^{\vee}}{\ol{Y \eta}(t)}$ 
(see the definitions of $t_0,\,t_1 \in [0,1]$ in 
Definition \ref{def:e_if_i}). Because 
\begin{align*}
\pair{\alpha_i^{\vee}}{u_M T_{\zeta'_M} \lambda} = 
\pair{\alpha_i^{\vee}}{u_M \lambda} = 
\pair{\alpha_i^{\vee}}{u_M T_{\zeta_M} \lambda} 
\qquad \text{for all $1 \le M \le N$}, 
\end{align*}
we deduce that $H_i^{\ol{Y \eta}} (t) = H_i^{\ol{Y \eta '}} (t)$ 
for all $t \in [0,1]$, which implies that $t_0 , t_1$ for $Y \eta '$ 
coincide with those for $Y \eta$. 
Therefore, it follows from the definition of the root operator 
$e_i$ or $f_i$ that $X \eta ' = g_m Y \eta ' \neq \bzero$, and 
that 
\begin{align*}
\sigma_N (X \eta ' ) 
&= \sigma_N (g_m Y \eta ' ) \\
&= [u_1 T_{\zeta '_1} , \ldots , u_{L} T_{\zeta '_{L}} , 
 r_i u_{L+1} T_{\zeta '_{L+1}} , \ldots , 
 r_i u_{K} T_{\zeta '_K} , 
 u_{K+1} T_{\zeta '_{K+1}} , \ldots , u_N T_{\zeta '_N}] .
\end{align*}
This proves the lemma.
\end{proof}

\begin{proof}[Proof of Proposition \ref{prop:unique}.]
First, we prove that each connected component of 
$\sLS$ contains an element of the form \eqref{eq:extr} with $\xi_{s}=0$. 
Let $\eta \in \sLS$; recall that $\ol{\eta} \in \LS$ by Proposition~\ref{prop:pi_}. 
From \cite[Theorem 3.1\,(2)]{NS08}, we know that 
there exists a monomial $X$ in the root operators $e_i$ and $f_i$, 
$i \in I_{\af}$, such that $(X \ol{\eta}) (t) \equiv 
t \lambda \mod \BR \delta$ for $t \in [0,1]$. 
Because the map $\ol{\phantom{\eta}}:\sLS \rightarrow \LS$ is 
a strict morphism of crystals (see Remark \ref{rem:SLS-LS}\,(1)), 
it follows immediately that $X \eta \neq \bzero$, and 
that $X \eta \in \sLS$ is of the form:
\begin{align*}
X \eta = 
(T_{\xi_1 '}, T_{\xi '_2}, \ldots , T_{\xi '_{s-1}} , T_{\xi} \,;\, 
 a_0 , a_1 , \ldots , a_{s-1} , a_s)
\end{align*}
for some $\xi'_1,\,\xi'_2,\,\ldots,\,\xi'_{s-1},\,\xi \in \jad$. 

Since $\eta_{T_{\xi}} = (T_{\xi}\,;\,0,1) \in \sLSo$ 
by Remark~\ref{rem:etax}, there exists 
a monomial $Y$ in the root operators 
$e_i$ and $f_i$, $i \in I_{\af}$, 
such that $Y \eta_{T_{\xi}} = \eta_e$. 
Let $N$ be a multiple of $N_{\lambda}$ 
(see the comment preceding Lemma~\ref{lem:Na}). Then,
$\sigma_N (\eta_{T_{\xi}}) 
 = [T_{\xi} , \ldots , T_{\xi}] \in \sLS^{\otimes N}$ and 
$\sigma_N (\eta_e) 
 = [e,\ldots ,e] \in \sLS^{\otimes N}$. 
Also, we see from Lemma \ref{lem:tensor} that
$\sigma_N (Y \eta_{T_{\xi}}) = 
 [v_1 T_{\xi} , v_2 T_{\xi} , \ldots , v_N T_{\xi}]$
for some $v_1 , v_2 , \ldots , v_N \in W_{\af}$ such that 
$v_M T_{\xi} \in (W^J)_{\af}$ for all $1 \le M \le N$. 
Since $Y \eta_{T_{\zeta}} = \eta_e$, and hence 
$\sigma_N (Y \eta_{T_{\zeta}}) = \sigma_N (\eta_e)$, 
we deduce that $v_M T_{\xi} = e$ for all $1 \le M \le N$. 
Thus, we obtain $v_M = T_{\xi}^{-1}$ for all $1 \le M \le N$.

It follows that $\sigma_N (X \eta)$ is of the form 
$\sigma_N (X \eta) = 
[T_{\zeta '_1} , T_{\zeta '_2} , \ldots , T_{\zeta '_N}]$
for some $\zeta '_1 , \zeta '_2 , \ldots , \zeta '_N \in \jad$
with $\zeta '_N = \xi$. 
Therefore, from Lemma \ref{lem:vT} applied to $\eta_{T_{\xi}}$ and $X\eta$, 
we deduce that
\begin{equation*}
\sigma_N (YX \eta) 
 = [v_1 T_{\zeta '_1} , v_2 T_{\zeta '_2} , \ldots , v_N T_{\zeta '_N}] 
 = [T_{\xi}^{-1} T_{\zeta '_1} , T_{\xi}^{-1} T_{\zeta '_2} , \ldots , 
   \underbrace{T_{\xi}^{-1} T_{\zeta '_N}}_{=e}], 
\end{equation*}
where $v_M T_{\zeta_M '} = T_{\xi}^{-1} T_{\zeta '_M} 
\in (W^J)_{\af}$ for all $1 \le M \le N$; 
using \eqref{eq:W^J_af}, we see by direct computation that 
$T_{\xi}^{-1} T_{\zeta '_M} \in (W^{J})_{\af}$ is 
of the form: 
%
%
\begin{equation} \label{eq:Txi}
T_{\xi}^{-1} T_{\zeta '_M}=
z_{\zeta_{M}}t_{\zeta_{M}}=T_{\zeta_{M}}
\quad \text{for some $\zeta_{M} \in \jad$}.
\end{equation}
Hence we obtain 
$\sigma_N (YX \eta) = 
 [T_{\zeta_1} , T_{\zeta_2} , \ldots , T_{\zeta_{N-1}} , e]$.
Because the final factor of $\sigma_{N}(YX\eta)$ is identical to $e=T_{0}$, 
we conclude that $YX \eta$ is of the form \eqref{eq:extr} with $\xi_{s}=0$. 
Thus, we have proved that each connected component of $\sLS$ 
contains an element of the form \eqref{eq:extr} with $\xi_{s}=0$.

Next, we prove the uniqueness statement. 
Let $\eta, \eta ' \in \sLS$, $\eta \neq \eta '$, 
be of the form \eqref{eq:extr} with $\xi_{s}=0$, and 
suppose that $X \eta = \eta '$ for some monomial $X$ 
in the root operators $e_i$ and $f_i$, $i \in I_{\af}$. 
As above, let $N \in \BZ_{>0}$ be a multiple of $N_{\lambda}$. 
Then, $\sigma_N (\eta)$ and $\sigma_N (\eta ')$ are 
of the form (note that $T_0=e$):
\begin{equation*}
\sigma_N (\eta ) = 
[T_{\zeta_1} , T_{\zeta_2} , \ldots , T_{\zeta_{N-1}} , T_0] , \qquad
\sigma_N (\eta ') = 
[T_{\zeta '_1} , T_{\zeta '_2} , \ldots , T_{\zeta '_{N-1}} , T_0], 
\end{equation*}
for some $\zeta_M , \zeta '_M \in \jad$, $1 \le M \le N-1$, 
respectively. Since $\eta \neq \eta '$ and 
the map $\sigma_N : \sLS \hookrightarrow \sLS^{\otimes N}$ 
is injective, 
there exists $1 \le M \le N-1$ such that $\zeta_M \neq \zeta '_M$; 
let $1 \le L \le N-1$ be the maximum of all such $M$'s.
Then we deduce from Lemma~\ref{lem:J-adj}\,(1) that 
$\zeta_L' - \zeta_L \notin Q_J^{\vee}$, and hence 
$\pJ{\zeta_L' - \zeta_L} \neq 0$; here, 
by interchanging $\eta$ and $\eta'$ if necessary, 
we may assume that 
$\pJ{\zeta_L' - \zeta_L} \notin \QJp{J^{c}}$. 

From Lemma \ref{lem:tensor}, we have 
\begin{align*}
\sigma_N (X \eta ) = 
[v_1 T_{\zeta_1} , v_2 T_{\zeta_2} , \ldots , 
 v_{N-1} T_{\zeta_{N-1}} , v_N T_0]
\end{align*}
for some $v_1 , v_2 , \ldots , v_N \in W_{\af}$ 
such that $v_M T_{\zeta_M} \in (W^J)_{\af}$, 
$1 \le M \le N-1$, and $v_N T_0 \in (W^J)_{\af}$; 
note that $v_1 , v_2 , \ldots , v_N \in (W^J)_{\af}$ 
by Lemma \ref{lem:xt}. Since $X \eta = \eta '$ by our assumption, 
it follows that $v_M T_{\zeta_M} = T_{\zeta '_M}$ 
for all $1 \le M \le N-1$, and $v_N T_0 = T_0$. 
Therefore, we obtain $v_M = T_{\zeta '_M} T_{\zeta_M}^{-1}$ 
for all $1 \le M \le N-1$, and $v_N = e$; 
in particular, $v_{L+1} = \cdots = v_{N-1} = v_N = e$ 
by the definition of $L$. 
Also, the same reasoning as for \eqref{eq:Txi} shows that 
for each $1 \le M \le N$, 
$v_M = T_{\zeta '_M} T_{\zeta_M}^{-1} \in (W^J)_{\af}$ is of the form 
$v_M = T_{\zeta ''_M}$ for some $\zeta ''_M \in \jad$ 
such that $\pJ{\zeta ''_M} = \pJ{\zeta '_M - \zeta_M}$. 
Therefore, from Lemma \ref{lem:vT} applied to $\eta$ and $\eta_{e}$, 
we deduce that $X \eta_e \neq \bzero$, and that 
\begin{align*}
\sigma_N (X \eta_e) = [v_1 , v_2 , \ldots , v_N ] = 
[T_{\zeta ''_1} , T_{\zeta ''_2} , \ldots , T_{\zeta ''_L} , e , \ldots , e], 
\end{align*}
which implies that $X \eta_e \in \sLS$ is of the form:
\begin{align*}
X \eta_e = 
(T_{\xi_1} , \ldots , T_{\xi_{s-1}} , e ; 
 a_0 , a_1 , \ldots , a_{s-1} , a_s ),
\end{align*}
for some $\xi_1 , \ldots , \xi_{s-1} \in \jad$, 
with $\xi_{s-1} = \zeta ''_L$. Hence we obtain 
$\pJ{\xi_{s-1} - 0} = \pJ{\zeta ''_L} = 
\pJ{\zeta '_L - \zeta_L} \notin \QJp{J^{c}}$,
which contradicts Proposition~\ref{prop:conn}. 
This completes the proof of Proposition~\ref{prop:unique}.
\end{proof}
%
%
\subsection{Proof of Proposition~\ref{prop:BNforSLS}.}
\label{subsec:prf-BNforSLS}

Recall that $\Turn(\lambda) = 
 \bigl\{ k / m_i \mid i \in J^{c}=I \setminus J,\ 0 \le k \le m_i \bigr\}$; 
we enumerate the elements of $\Turn(\lambda)$ in increasing order as:
\begin{equation*}
\Turn(\lambda) = 
 \bigl\{ 0 = \tau_0 < \tau_1 < \cdots < \tau_p = 1 \bigr\}.
\end{equation*}
Let $0 \le q \le p$. Note that $i \in \Jcb{\tau_q}$ 
if and only if $i \notin J$ and 
$\tau_q = k / m_i$ for some $0 \le k \le m_i$. 

\begin{prop}\label{prop:1-1corr}
There exists a bijective correspondence between 
the set $\Par(\lambda)$ and the set of all connected components of $\sLS$.
\end{prop}

\begin{proof}
Let $\Conn(\lambda)$ denote the set of 
all connected components of $\sLS$. First, we define a map 
$\Theta : \Conn(\lambda ) \rightarrow \Par(\lambda )$ as follows. 
Take an arbitrary $C \in \Conn(\lambda )$. 
By Propositions~\ref{prop:conn} and \ref{prop:unique}, 
the connected component $C$ contains a unique element $\eta$ 
of the form
$\eta = 
 (T_{\xi_1} , T_{\xi_2} , \ldots , T_{\xi_{s-1}} ,e \,;\, 
 a_0 ,a_1 , \ldots , a_{s-1} , a_s )$, 
with $\xi_1 , \xi_2 , \ldots , \xi_{s-1} \in \jad$, 
such that 
$a_{u} \in \Turn(\lambda)$ for all $0 \le u \le s$, and 
such that $\pJ{\xi_{u}-\xi_{u+1}} \in \QJp{\Jcb{a_{u}}}$ 
for all $1 \le u \le s-1$, where we set $\xi_s := 0$. 
For each $0 \le u \le s$, let $0 \le q_u \le p$ be 
such that $a_u = \tau_{q_u}$. 
Then we define $\zeta_q$, $1 \le q \le p$, by 
$\zeta_q := \xi_u$ if $q_{u-1} + 1 \le q \le q_u$,
that is, 
\begin{align*}
\underbrace{\zeta_1 , \ldots , \zeta_{q_1}}_{:= \xi_1},\ 
\underbrace{\zeta_{q_1 +1} , \ldots , \zeta_{q_2}}_{:= \xi_2},\ 
\ldots,\ 
\underbrace{\zeta_{q_{s-2} + 1} , \ldots , \zeta_{q_{s-1}}}_{:= \xi_{s-1}},\ 
\underbrace{\zeta_{q_{s-1} + 1} , \ldots , \zeta_p}_{:= \xi_s = 0};
\end{align*}
remark that for all $1 \le q \le p-1$, 
\begin{equation} \label{eq:nonnegative}
\pJ{ \zeta_{q} - \zeta_{q+1} } \in 
\QJp{\Jcb{\tau_{q}}}.
\end{equation}
Fix $i \in J^{c}=I \setminus J$, and 
let $c^{(i)}_{q}$, $1 \le q \le p$, 
be the coefficient of $\alpha_{i}^{\vee}$ in $\zeta_q$; 
we see from \eqref{eq:nonnegative} that 
$c_1^{(i)} \ge c_2^{(i)} \ge \cdots \ge c_{p-1}^{(i)} \ge c_p^{(i)} =0$,
and that $c_q^{(i)} = c_{q+1}^{(i)}$ 
for $1 \le q \le p -1$ such that 
$i \notin \Jcb{\tau_{q}}$, i.e., 
$\tau_q \notin \bigl\{ k / m_i \mid 0 \le k \le m_i \bigr\}$.
For each $1 \le k \le m_{i}$, 
let $1 \le p_{k} \le p$ be such that $\tau_{p_k}=k/m_i$. 
Then we define $\rho^{(i)}_{k}$, $1 \le k \le m_{i}-1$, 
in the following way: 
\begin{align*}
& 
\underbrace{c_1^{(i)} = \cdots =c_{p_1}^{(i)}}_{=:\rho^{(i)}_{1}} \ge 
\underbrace{c_{p_1+1}^{(i)} = \cdots =c_{p_2}^{(i)}}_{=:\rho^{(i)}_{2}} 
\ge \cdots \cdots \\[3mm]
& \hspace*{30mm} \cdots \cdots \ge
\underbrace{c_{p_{m_{i}-2}+1}^{(i)} = \cdots =c_{p_{m_{i}-1}}^{(i)}}_{%
  =:\rho^{(i)}_{m_{i}-1}} \ge 
c_{p_{m_{i}-1}+1}^{(i)} = \cdots =c_{p_{m_{i}}}^{(i)}=0.
\end{align*}
Thus we obtain a partition 
$\rho^{(i)} := (\rho^{(i)}_1 \ge \rho^{(i)}_2 \ge \cdots \ge \rho^{(i)}_{m_i -1})$ 
of length less than $m_i$. Now we define $\Theta (C) := 
(\rho^{(i)})_{i \in I} \in \Par(\lambda )$, 
where for every $j \in J$, we define $\rho^{(j)}$ 
to be the empty partition.

Next, we define a map 
$\Xi : \Par(\lambda ) \rightarrow \Conn(\lambda )$ as follows. 
Take an arbitrary $\bm{\rho} = (\rho^{(i)}) \in \Par(\lambda)$, 
with $\rho^{(i)} = 
(\rho^{(i)}_1 \ge \rho^{(i)}_2 \ge \cdots \ge \rho^{(i)}_{m_i -1})$ 
for $i \in I \setminus J$. 
Define $\zeta_{q} \in Q^{\vee}$, $1 \le q \le p$, inductively by
\begin{align*}
\zeta_{p}:=0, \quad \text{and} \quad
\zeta_{q} : = \zeta_{q+1}+
 \sum_{i \in \Jcb{\tau_q}}
   (\rho^{(i)}_{\tau_q m_i} - \rho^{(i)}_{\tau_q m_i + 1})
   \alpha_{i}^{\vee}
\quad
\text{for $1 \le q \le p-1$};
\end{align*}
note that for $1 \le q \le p -1$, 
if $i \in \Jcb{\tau_q}$, then 
$\tau_q m_i=\tau_q \pair{\alpha_i^{\vee}}{\lambda} \in \BZ$, 
with $1 \le \tau_q m_i \le m_i-1$. 
We write the set 
$\bigl\{ 1 \le q \le p-1 \mid \zeta_q \neq \zeta_{q+1} \bigr\}$ 
in the form $\bigl\{ q_1 < q_2 < \cdots < q_{s-1} \bigr\}$, i.e., 
\begin{align*}
\zeta_1 = \cdots = \zeta_{q_1} \neq 
\zeta_{q_1 +1} = \cdots = \zeta_{q_2} \neq 
\cdots \cdots \neq 
\zeta_{q_{s-1} + 1} = \cdots = \zeta_{p} = 0.
\end{align*}
Then we define (for the definition of $\phi_J (\zeta_{p_u})$, 
see Lemma \ref{lem:J-adj})
\begin{align*}
\begin{cases}
\xi_s := 0, \qquad 
\xi_u := \zeta_{q_u} + \phi_J (\zeta_{q_u}) \quad 
\text{for $1 \le u \le s-1$}, \\[1.5mm]
a_0 := 0, \qquad 
a_u := \tau_{q_u} \quad \text{for $1 \le u \le s-1$}, \qquad 
a_s := 1;
\end{cases}
\end{align*}
it follows from Proposition \ref{prop:conn} that 
\begin{align} \label{eq:eta_rho}
\eta_{\bm{\rho}} := 
(T_{\xi_1} , T_{\xi_2} , \ldots , T_{\xi_{s-1}} , e \,;\, 
a_0 , a_1 , \ldots , a_{s-1} , a_s ) \in \sLS. 
\end{align}
Now we define $\Xi (\bm{\rho})$ to be 
the connected component of $\sLS$ 
containing this $\eta_{\bm{\rho}}$.

We deduce from the definitions that 
the maps $\Theta$ and $\Xi$ are inverses of each other. 
This completes the proof of the proposition.
\end{proof}

\begin{proof}[Proof of Proposition~\ref{prop:BNforSLS}.]
For $\bm{\rho} \in \Par(\lambda )$, 
let $\eta_{\bm{\rho}} \in \sLS$ be as defined by \eqref{eq:eta_rho}, 
which is a unique element of the form \eqref{eq:extr} with $\xi_{s}=0$ 
contained in the connected component 
$\BB_{\bm{\rho}}^{\si} (\lambda ) := \Xi (\bm{\rho})$. 
We prove that there exists a unique isomorphism 
$\BB_{\bm{\rho}}^{\si} (\lambda ) \stackrel{\sim}{\rightarrow} 
\{ \bm{\rho} \} \otimes \sLSo$ of crystals 
that maps $\eta_{\bm{\rho}}$ to $\bm{\rho} \otimes \eta_e$. 
As in the proof of Proposition \ref{prop:conn-isom}, 
it suffices to show the following 
for monomials $X$, $Y$ in the Kashiwara operators:
\begin{itemize}

\item[(i)]
$X \eta_{\bm{\rho}} \neq \bzero$ 
if and only if $X (\bm{\rho} \otimes \eta_e) \neq \bzero$;

\item[(ii)]
$X \eta_{\bm{\rho}} = Y \eta_{\bm{\rho}}$ 
if and only if $X(\bm{\rho} \otimes \eta_e) = 
Y(\bm{\rho} \otimes \eta_e)$.

\end{itemize}
Part (i) follows immediately from Lemma \ref{lem:vT} and 
the equality $X (\bm{\rho} \otimes \eta_e) = 
\bm{\rho} \otimes X \eta_e$.
Let us show part (ii). 
We give a proof only for the ``only if" part; 
the proof for the ``if" part is similar. 
Assume that $X \eta_{\bm{\rho}} = Y \eta_{\bm{\rho}} \ne \bzero$. 
Let $N \in \BZ_{>0}$ be a multiple of $N_{\lambda}$ 
(see the comment preceding Lemma~\ref{lem:Na}), 
and write $\sigma_N (\eta_{\bm{\rho}})$ as
$\sigma_N (\eta_{\bm{\rho}}) = 
 [T_{\xi_1} , T_{\xi_2} , \ldots , T_{\xi_N}]$ 
 for some $\xi_1 , \xi_2 , \ldots , \xi_N \in \jad$. 
By Lemma~\ref{lem:tensor}, 
$\sigma_N (X \eta_{\bm{\rho}})$ and 
$\sigma_N (Y \eta_{\bm{\rho}})$ are of the form: 
\begin{align*}
\sigma_N (X \eta_{\bm{\rho}}) = 
[u_1 T_{\xi_1} , u_2 T_{\xi_2} , \ldots , u_N T_{\xi_N}] , && 
\sigma_N (Y \eta_{\bm{\rho}}) = 
[v_1 T_{\xi_1} , v_2 T_{\xi_2} , \ldots , v_N T_{\xi_N}],
\end{align*}
for some $u_1 , u_2 , \ldots , u_N \in (W^J)_{\af}$ and 
$v_1 , v_2 , \ldots , v_N \in (W^J)_{\af}$, respectively. 
Then, by Lemma \ref{lem:vT}, we have 
\begin{align*}
\sigma_N (X \eta_e) = [u_1 , u_2 , \ldots , u_N] , &&
\sigma_N (Y \eta_e) = [v_1 , v_2 , \ldots , v_N ],
\end{align*}
respectively. Since $X \eta_{\bm{\rho}} = Y \eta_{\bm{\rho}}$ 
by the assumption, we have $u_M = v_M$ for all $1 \le M \le N$. 
Therefore, we see that $\sigma_N (X \eta_e) = \sigma_N (Y \eta_e)$, 
and hence $X \eta_e = Y \eta_e$ by the injectivity of $\sigma_N$. 
Thus, we obtain $X(\bm{\rho} \otimes \eta_e) = 
\bm{\rho} \otimes X\eta_e=\bm{\rho} \otimes Y\eta_e=
Y(\bm{\rho} \otimes \eta_e)$, as desired. 
Finally, from the existence of the isomorphism 
$\BB_{\bm{\rho}}^{\si} (\lambda ) \stackrel{\sim}{\rightarrow} 
\{ \bm{\rho} \} \otimes \sLSo$ of crystals 
for each $\bm{\rho} \in \Par(\lambda)$, together with 
Proposition~\ref{prop:1-1corr}, we conclude that 
\begin{align*}
\sLS = \bigsqcup_{\bm{\rho} \in \Par(\lambda )} 
\BB_{\bm{\rho}}^{\si} (\lambda ) \cong 
\bigsqcup_{\bm{\rho} \in \Par(\lambda )} 
\{ \bm{\rho} \} \otimes \sLSo = 
\Par(\lambda ) \otimes \sLSo .
\end{align*}
This completes the proof of 
Proposition \ref{prop:BNforSLS}.
\end{proof}
\begin{rem}
Since $\CB(\lambda) \cong \BB^{\si}(\lambda)$ as crystals, 
we can define the action of the Weyl group $W_{\af}$ on 
$\BB^{\si}(\lambda)$ by the same formula as \eqref{eq:W-action} 
for the one on $\mathcal{B}(\lambda)$.
Let $\bm{\rho} \in \Par(\lambda)$. 
Then, as shown in the proof of Proposition~\ref{prop:BNforSLS},
there exists a unique isomorphism 
$\BB_{\bm{\rho}}^{\si} (\lambda ) \stackrel{\sim}{\rightarrow} 
\{ \bm{\rho} \} \otimes \sLSo$ of crystals that maps 
$\eta_{\bm{\rho}}$ to $\bm{\rho} \otimes \eta_e$.
Hence it follows from the tensor product rule for crystals that 
$x(\bm{\rho} \otimes \eta_e) = \bm{\rho} \otimes (x\eta_e)$ 
for every $x \in W_{\af}$.
From this equality, again using the tensor product rule for crystals,
we deduce that $\eta_{\bm{\rho}}$ is an extremal element 
of weight $\lambda + \wt(\bm{\rho})$, 
since $\eta_{e}$ is an extremal element of 
weight $\lambda$ by Remark~\ref{rem:etae-ext}.
\end{rem}
%
%
\appendix

\section{Appendix.} \label{sec:Appendix}

%
\subsection{Relation between the semi-infinite Bruhat graph and 
the quantum Bruhat graph.}
\label{subsec:SB-QB}

In this subsection, we fix a subset $J \subset I$. 
Set $\rho_J := \frac{1}{2}\sum_{\alpha \in \Delta_J^+} \alpha$; 
note that 
%
%
\begin{align}\label{eq:rhoJ}
\pair{\xi}{\rho - \rho_J} = 0 \quad 
\text{for all $\xi \in Q_J^{\vee}$}.
\end{align}

\begin{define}[{\cite[\S4]{LNSSS13a}; see also \cite[\S6]{BFP99}}] 
\label{def:qedge}
Let $J$ be a subset of $I$. 
Define the (parabolic) quantum Bruhat graph $\QB$ to 
be the $(\Delta^{+} \setminus \Delta_J^+)$-labeled, 
directed graph with vertex set $W^J$ and 
$(\Delta^{+} \setminus \Delta_J^+)$-labeled, 
directed edges of the following form: 
$w \edge{\gamma} \mcr{w r_{\gamma}}$ 
for $w \in W^J$ and $\gamma \in \Delta^+ \setminus \Delta _J^+$ 
such that either (i) $\ell(\mcr{w r_{\gamma}}) = \ell (w) + 1$, or 
(ii) $\ell(\mcr{w r_{\gamma}}) = \ell (w) + 1 - 
2 \pair{\gamma^{\vee}}{\rho - \rho _J}$; 
we call an edge $w \edge{\gamma} \mcr{w r_{\gamma}}$ 
satisfying condition (i) (resp., (ii)) 
a Bruhat (resp., quantum) edge. 
\end{define}

Combining Proposition~\ref{prop:sib-lv0} and 
\cite[Theorem~6.5]{LNSSS13a}, we obtain the following.
%
%
\begin{prop} \label{prop:SB-QB}
\mbox{}
\begin{enu}

\item
Let $x = w z_{\xi} t_{\xi} \in (W^J)_{\af}$ with $w \in W^J$ and $\xi \in \jad$, 
and $\beta \in \prr$. Assume that $x \edge{\beta} r_{\beta} x$ in $\SB$; 
note that $\beta = w \gamma + n \delta$ 
for some $\gamma \in \Delta^{+} \setminus \Delta _J^{+}$ and 
$n \in \bigl\{ 0,\,1 \bigr\}$ by Corollary~\ref{cor:pi_}.
If $n = 0$ (resp., $n = 1$), then 
we have a Bruhat edge (reps., quantum edge) 
$w \edge{\gamma} \mcr{wr_{\gamma}}$ in $\QB$.

\item
Let $w \in W^J$, and $\gamma \in \Delta^+ \setminus \Delta _J ^+$. 
Assume that $w \edge{\gamma} \mcr{wr_{\gamma}}$ in $\QB$. 
Set $\beta := w \gamma $ (resp., $\beta := w \gamma + \delta$) 
if the edge is a Bruhat edge (resp., quantum edge). 
Then, $\beta \in \prr$, and $w z_{\xi} t_{\xi} \edge{\beta} 
r_{\beta} w z_{\xi} t_{\xi}$ in $\SB$ for every $\xi \in \jad$.

\end{enu}
\end{prop}
%
%
\subsection{Another definition of the semi-infinite Bruhat order.}
\label{subsec:def}

In this subsection, we fix a subset $J \subset I$. 
For $v t_{\zeta} \in W_{\af}$ 
with $v \in W$ and $\zeta \in Q^{\vee}$, 
we define
\begin{align}\label{eq:silJ}
\sil_J (v t_{\zeta}) := \ell (\mcr{v}) + 
2 \pair{\zeta}{\rho - \rho_J}. 
\end{align}
\begin{lem}\label{lem:length}
The equalities $\sil_J (x) = \sil_J (\PJ (x)) = 
\sil (\PJ (x))$ hold for all $x \in W_{\af}$.
\end{lem}

\begin{proof}
We write $\PJ (x)$ as $\PJ(x) = w z_{\xi} t_{\xi}$, 
with $w \in W^J$ and $\xi \in \jad$. 
The second equality follows from \cite[(3.11)]{LNSSS13a} and 
the equality $\ell (w z_{\xi}) = \ell (w) + \ell (z_{\xi})$.

In order to prove the first equality, 
we write $x = x_1 x_2$ with $x_1 \in (W^J)_{\af}$ 
and $x_2 \in (W_J)_{\af}$; note that $\PJ (x) = x_1$. 
We have $x_1 = w_1 z_{\xi_1} t_{\xi_1}$ 
for some $w_1 \in W^J$ and $\xi_1 \in \jad$ by \eqref{eq:W^J_af}, 
and $x_2 = w_2 t_{\xi_2}$ for some $w_2 \in W_J$ and 
$\xi_2 \in Q_J^{\vee}$ by \eqref{eq:stabilizer}. 
Since $x = x_1 x_2 = w_1 z_{\xi_1} w_2 t_{w_2^{-1} \xi_1 + \xi_2}$, 
we compute
\begin{align*}
\sil_J (x) 
&= \ell (\mcr{w_1 z_{\xi_1} w_2}) + 
   2 \pair{ w_2^{-1} \xi_1 + \xi_2 }{ \rho - \rho_J } & \\
&= \ell (w_1) + 2 \pair{ w_2^{-1} \xi_1 + \xi_2 }{ \rho - \rho_J }
\quad \text{since $w_1 \in W^J$ and $z_{\xi_1} w_2 \in W_J$} \\
&= \ell (w_1) + 2 \pair{ \xi_1 }{ \rho - \rho_J} 
\quad \text{since $\xi_2 \in Q_J^{\vee}$ and $w_2 \in W_J$ 
(see \eqref{eq:rhoJ})} \\
&= \sil_J (w_1 z_{\xi_1} t_{\xi_1}) = \sil_J (x_1) = \sil_J (\PJ (x)). 
\end{align*}
This proves the lemma.
\end{proof}

\begin{prop}
Let $x,\,y \in (W^J)_{\af}$ and $\beta \in \prr$. 
We have $x \edge{\beta} y$ in $\SB$ 
if and only if the following three conditions are satisfied: 
\begin{itemize}

\item[\rm (a)] 
$y = \PJ (r_{\beta} x)$;

\item[\rm (b)] 
$\sil_J (r_{\beta} x) = \sil_J (x) +1$;

\item[\rm (c)]
if we write $x$ as $x = w z_{\xi} t_{\xi}$ 
with $w \in W^J$ and $\xi \in \jad$, then 
$\beta = w \gamma + n \delta$ 
for some $\gamma \in \Delta^+ \setminus \Delta_J^+$ and 
$n \in \bigl\{ 0,\,1 \bigr\}$.

\end{itemize}
\end{prop}

\begin{proof}
The ``only if" part follows immediately 
from Corollary~\ref{cor:pi_} and Lemma~\ref{lem:length}. 
We show the ``if" part. 
By condition (c), we have $r_{\beta} x = 
w r_{\gamma} z_{\xi} t_{\xi + n z_{\xi}^{-1} \gamma^{\vee}}$.
We compute
\begin{align*}
1 & = \sil_J (r_{\beta} x) - \sil_J (x) \quad \text{by (b)} \\
& = 
 \ell (\mcr{w r_{\gamma}}) + 
 2 \pair{\xi + n z_{\xi}^{-1} \gamma^{\vee}}{\rho - \rho_J} - 
 \ell (w) - 2 \pair{\xi}{\rho - \rho_J} \\
& = 
  \ell (\mcr{w r_{\gamma}}) - 
  \ell (w) + 2 n \pair{\gamma^{\vee}}{\rho - \rho_J} \qquad 
  \text{(see \eqref{eq:rhoJ})}.
\end{align*}
From this, using the condition that $n \in \bigl\{0,\,1\bigr\}$, 
we deduce that $w \edge{\gamma} \mcr{wr_{\gamma}}$ in $\QB$; 
observe that this edge is a Bruhat (resp., quantum) edge if and only if 
$n=0$ (resp., $n=1$). Therefore, by Proposition \ref{prop:SB-QB}\,(2), we have 
$x=wz_{\xi}t_{\xi} \edge{\beta} r_{\beta}wz_{\xi}t_{\xi}=r_{\beta}x$ in $\SB$; 
in particular, $r_{\beta}x \in (W^{J})_{\af}$. Thus, by condition (a), 
we obtain $y = \Pi^{J}(r_{\beta} x)=r_{\beta} x$, and hence 
$x \edge{\beta} y$ in $\SB$. This proves the proposition. 
\end{proof}
%
%
\subsection{Relation between the semi-infinite Bruhat order 
and the generic Bruhat order.}
\label{subsec:si-gen}

In this subsection, we assume that $J = \emptyset$; 
note that $(W^J)_{\af} = W_{\af}$. 
Fix an (arbitrary) element $\xi \in Q^{\vee}$ 
such that $\pair{\xi}{\alpha_i} > 0$ 
for all $i \in I$. 
We know from \cite{Pet97} 
(see also \cite[Theorem~5.2]{LNSSS13a}, and 
Proposition \ref{prop:SB-QB}) that 
for $x,\,y \in (W^J)_{\af} = W_{\af}$, 
$x \le_{\si} y$ if and only if there exists $N \in \BZ_{\ge 0}$, 
depending on $x$, $y$, and $\xi$, 
such that $y t_{-n\xi} \le x t_{-n\xi}$ 
(or equivalently, 
$t_{n \xi} y^{-1} \le t_{n \xi} x^{-1}$) 
for all $n \in \BZ_{\ge N}$, where 
$\le$ is the (ordinary) Bruhat order on $W_{\af}$. 
Also, in \cite[\S 1.5]{Lus80}, 
Lusztig introduced a partial order $\le_L$ on $W_{\af}$, 
which we call Lusztig's generic Bruhat order; 
we know from \cite[Claim~4.14 in the proof of Lemma~4.13]{Soe97} 
that $x \le_L y$ if and only if there exists 
$N \in \BZ_{\ge 0}$, depending on $x, y$, and $\xi$, 
such that $t_{n\xi} x \le t_{n\xi} y$ for all 
$n \in \BZ_{\ge N}$. Combining these facts, 
we obtain the following.

\begin{lem}\label{lem:sib-gen}
Let $x , y \in W_{\af}$. We have $x \le_{\si} y$ 
if and only if $y^{-1} \le_L x^{-1}$.
\end{lem}
%
%
{\small
\setlength{\baselineskip}{13pt}
\renewcommand{\refname}{References}

}


\begin{thebibliography}{XXXXX}

\bibitem[AK]{AK}
T. Akasaka and M. Kashiwara, Finite-dimensional representations of 
quantum affine algebras, {\it Publ. Res. Inst. Math. Sci.} {\bf 33} (1997), 
839--867. 

\bibitem[BN]{BN04}
J. Beck and H. Nakajima, 
Crystal bases and two-sided cells of quantum affine algebras, 
{\it Duke Math. J.} {\bf 123} (2004), 335--402. 

\bibitem[BB]{BB05}
A. Bj\"{o}rner and F. Brenti, 
``Combinatorics of Coxeter Groups'', 
Graduate Texts in Mathematics Vol.~231, 
Springer, New York, 2005.


\bibitem[BF]{BF}
A. Braverman and M. Finkelberg, 
Weyl modules and $q$-Whittaker functions, 
{\it Math. Ann.} {\bf 359} (2014), 45--59. 

\bibitem[BFP]{BFP99}
F. Brenti, S. Fomin, and A. Postnikov, 
Mixed Bruhat operators and Yang-Baxter equations 
for Weyl groups, {\it Int. Math. Res. Not.} {\bf 8} (1999), 419--441.

\bibitem[CP]{CP01}
V. Chari and A. Pressley, 
Weyl modules for classical and quantum affine algebras, 
{\it Represent. Theory} {\bf 5} (2001), 191--223. 

\bibitem[FFKM]{FFKM99}
B. Feigin, M. Finkelberg, 
A. Kuznetsov, and I. Mirkovi\'{c}, 
Semi-infinite flags I\hspace{-1.5pt}I:
Local and global intersection cohomology of quasimaps' spaces, 
Differential Topology, Infinite-Dimensional Lie Algebras, 
and Applications, 
Amer. Math. Soc. Transl. Ser.~2, Vol.~194, pp.\,113--148, 
Amer. Math. Soc., Providence, RI, 1990.

\bibitem[FF]{FF90}
B. Feigin and E. Frenkel, 
Affine Kac-Moody algebras and semi-infinite flag manifolds, 
{\it Comm. Math. Phys.} {\bf 128} (1990), 161--189.

\bibitem[HN]{HN06}
D. Hernandez and H. Nakajima, 
Level $0$ monomial crystals, 
{\it Nagoya Math. J.} {\bf 184} (2006), 85--153.

\bibitem[HK]{HK02}
J. Hong and S.-J. Kang, 
``Introduction to Quantum Groups and Crystal Bases'', 
Graduate Studies in Mathematics Vol.~42, 
Amer. Math. Soc., Providence, RI, 2002.

\bibitem[Kac]{Kac90}
V. G. Kac, 
``Infinite Dimensional Lie Algebras'', 3rd Edition, 
Cambridge University Press, Cambridge, UK, 1990.

\bibitem[Kas1]{Kas94}
M. Kashiwara, Crystal bases of modified quantized enveloping algebra, 
{\it Duke Math. J.} {\bf 73} (1994), 383--413. 

\bibitem[Kas2]{Kas95}
M. Kashiwara, On crystal bases, 
{\it in} ``Representations of Groups'' 
(B.N. Allison and G.H. Cliff, Eds.), 
CMS Conf. Proc. Vol.~16, pp.~155--197, Amer. Math. Soc., 
Providence, RI, 1995.

\bibitem[Kas3]{Kas96}
M. Kashiwara, Similarity of crystal bases, 
{\it in} ``Lie Algebras and Their Representations'' 
(S.-J. Kang et al., Eds.), Contemp. Math. Vol.~194, 
pp. 177--186, Amer. Math. Soc., Providence, RI, 1996.

\bibitem[Kas4]{Kas02a}
M. Kashiwara, ``Bases Cristallines des Groupes Quantiques'' 
(Notes by Charles Cochet), Cours Sp\'{e}cialis\'{e}s  Vol.~9, 
Soci\'{e}t\'{e} Math\'{e}matique de France, Paris, 2002. 

\bibitem[Kas5]{Kas02b}
M. Kashiwara, On level-zero representations of 
quantized affine algebras, 
{\it Duke Math. J.} {\bf 112} (2002), 117--175.

\bibitem[LS]{LS10}
T. Lam and M. Shimozono, 
Quantum cohomology of $G/P$ and homology of 
affine Grassmannian, 
{\it Acta Math.} {\bf 204} (2010), 49--90.

\bibitem[LNS$^3$1]{LNSSS13a}
C. Lenart, S. Naito, D. Sagaki, 
A. Schilling, and M. Shimozono, 
A uniform model for Kirillov-Reshetikhin crystals I: 
lifting the parabolic quantum Bruhat graph, arXiv:1211.2042v3, 
to appear in {\it Int. Math. Res. Not.}

\bibitem[LNS$^{3}$2]{LNSSS2}
C. Lenart, S. Naito, D. Sagaki, A. Schilling, and M. Shimozono, 
A uniform model for Kirillov-Reshetikhin crystals I\hspace{-1.5pt}I: 
Alcove model, path model, and $P=X$, preprint 2014, arXiv:1402.2203.

\bibitem[Li1]{Lit94}
P. Littelmann, 
A Littlewood-Richardson rule for symmetrizable 
Kac-Moody algebras, 
{\it Invent. Math.} {\bf 116} (1994), 329--346.

\bibitem[Li2]{Lit95}
P. Littelmann,
Paths and root operators in representation theory, 
{\it Ann. of Math.} (2) {\bf 142} (1995), 499--525.

\bibitem[Lu1]{Lus80}
G. Lusztig, Hecke algebras and Jantzen's generic 
decomposition patterns, {\it Adv. Math.} {\bf 37} (1980), 
121--164.

\bibitem[Lu2]{Lus92}
G. Lusztig, Canonical bases in tensor products, 
{\it Proc. Nat. Acad. Sci. U.S.A.} {\bf 89} (1992), 
8177--8179. 


\bibitem[Mi]{Mih07}
L. C. Mihalcea, On equivariant quantum cohomology of 
homogeneous spaces: Chevalley formulae and algorithms, 
{\it Duke Math. J.} {\bf 140} (2007), 321--350. 

\bibitem[NS1]{NS03}
S. Naito and D. Sagaki, 
Path model for a level-zero extremal weight module over a quantum 
affine algebra, {\it Int. Math. Res. Not.} 
{\bf 2003} (2003), no.~32, 1731--1754.

\bibitem[NS2]{NS05}
S. Naito and D. Sagaki, 
Crystal of Lakshmibai-Seshadri paths
associated to an integral weight of level zero
for an affine Lie algebra,
{\it Int. Math. Res. Not.} {\bf 2005} (2005), no.~14, 815--840.

\bibitem[NS3]{NS06}
S. Naito and D. Sagaki, 
Path model for a level-zero extremal weight module over a quantum 
affine algebra. I\hspace{-1pt}I, 
{\it Adv. Math.} {\bf 200} (2006), 102--124. 

\bibitem[NS4]{NS08}
S. Naito and D. Sagaki, 
Crystal structure on the set of Lakshmibai-Seshadri 
paths of an arbitrary level-zero shape, 
{\it Proc. Lond. Math. Soc.} (3) {\bf 96} (2008), 582--622.

\bibitem[NS5]{NS-Dem}
S. Naito and D. Sagaki, 
Demazure submodules of level-zero extremal weight modules 
and specializations of Macdonald polynomials, 
preprint 2014, arXiv:1404.2436v2. 

\bibitem[N1]{Nak01}
H. Nakajima, Quiver varieties and finite-dimensional 
representations of quantum affine algebras, 
{\it J. Amer. Math. Soc.} {\bf 14} (2001), 145--238. 

\bibitem[N2]{Nak04}
H. Nakajima, Extremal weight modules of quantum affine algebras, 
{\it in} ``Representation Theory of Algebraic Groups and 
Quantum Groups'' (T. Shoji et al., Eds.), 
Adv. Stud. Pure Math. Vol.~40, pp.\,343--369, Math. Soc. Japan, 2004. 

\bibitem[P]{Pet97}
D. Peterson, Quantum cohomology of $G/P$, Lecture Notes, 
Cambridge, MA, Spring: Massachusetts Institute of Technology, 1997.

\bibitem[S]{Soe97}
W. Soergel, Kazhdan-Lusztig polynomials and 
a combinatoric for tilting modules, 
{\it Represent. Theory} {\bf 1} (1997), 83--114. 


\end{thebibliography}
\end{document}